\documentclass{amsart}

\usepackage{amsmath,amsfonts,amssymb,amscd,verbatim,delarray,fullpage}
\usepackage{stmaryrd}
\usepackage{url}
\usepackage[polutonikogreek,english]{babel}

\usepackage{amsmath}              
\usepackage{amsthm}
\usepackage{amsopn,amssymb}       
\usepackage{bookmark}             
\usepackage{enumerate}            
\usepackage{framed}
\usepackage[margin=1in]{geometry} 
\usepackage{tikz-cd}              
\usepackage{mathrsfs}
\usepackage{comment}
\usepackage[     
url=false,
isbn=false,
  hyperref = true,      
  backend  = bibtex,    
  sorting  = nyt,       
  style    = alphabetic,
  maxcitenames=5,
  maxbibnames=8,
]{biblatex}

\usepackage{array}
\usepackage{float}
\newcolumntype{L}{>{$}l<{$}}

\newcommand*{\sheafhom}{\mathcal{H}\kern -.5pt om}
\DeclareMathOperator{\Prob}{Prob}

\DeclareMathOperator{\rred}{red}
\DeclareMathOperator{\im}{Im}
\DeclareMathOperator{\abs}{abs}
\DeclareMathOperator{\pred}{pred}
\DeclareMathOperator{\actual}{actual}
\DeclareMathOperator{\rel}{rel}
\DeclareMathOperator{\CLT}{CLT}

\DeclareMathOperator{\disc}{disc}

\DeclareMathOperator{\GL}{GL}
\DeclareMathOperator{\GSp}{GSp}
\DeclareMathOperator{\Gal}{Gal}

\DeclareMathOperator{\lcm}{lcm}

\DeclareMathOperator{\rad}{rad}

\DeclareMathOperator{\SL}{SL}

\DeclareMathOperator{\SU}{SU}

\DeclareMathOperator{\tr}{tr}

\DeclareMathOperator{\Li}{Li}
\DeclareMathOperator{\USp}{USp}
\DeclareMathOperator{\Frob}{Frob}

\newcommand{\mbq}{\mathbb{Q}}
\newcommand{\mbz}{\mathbb{Z}}

\newcommand{\bF}{\mathbb{F}}
\newcommand{\bN}{\mathbb{N}}
\newcommand{\bQ}{\mathbb{Q}}
\newcommand{\bR}{\mathbb{R}}
\newcommand{\bZ}{\mathbb{Z}}

\newcommand{\dd}{\mathrm{d}}

\newcommand{\legendre}[2]{\left(\frac{#1}{#2}\right)}

\newcommand{\mc}[1]{\mathcal{#1}} 

\newcommand{\ds}{\displaystyle}

\newcommand{\ve}{\varepsilon}

\newcommand{\ga}{\alpha}

\newcommand{\gd}{\delta}
\newcommand{\gD}{\Delta}

\newcommand{\om}{\omega}

\theoremstyle{plain}
\newtheorem{theorem}{Theorem}[section]
\newtheorem{proposition}[theorem]{Proposition}
\newtheorem{conjecture}[theorem]{Conjecture}
\newtheorem{Corollary}[theorem]{Corollary}

\newtheorem{Lemma}[theorem]{Lemma}

\theoremstyle{definition}
\newtheorem{Definition}[theorem]{Definition}

\theoremstyle{remark}
\newtheorem{goal}[theorem]{Goal}
\newtheorem{remark}[theorem]{Remark}
\newtheorem*{remarks}{Remarks}

\numberwithin{equation}{section}
\numberwithin{figure}{section}
\numberwithin{table}{section}
\numberwithin{theorem}{subsection}

\title{The Lang-Trotter Conjecture for products of non-CM elliptic curves}
\author{Hao Chen, Nathan Jones, Vlad Serban}
\thanks{The third author's research was partially funded by START-prize Y-966 of the Austrian Science Fund (FWF) under P.I. Harald Grobner}
\begin{document}

\maketitle


\begin{abstract}
Inspired by the work of Lang-Trotter on the densities of primes with fixed Frobenius traces for elliptic curves defined over $\mbq$ and by the subsequent generalization of Cojocaru-Davis-Silverberg-Stange to generic abelian varieties, we study the analogous question for abelian surfaces isogenous to products of non-CM elliptic curves over $\mbq$.  We formulate the corresponding conjectural asymptotic, provide upper bounds, and explicitly compute (when the elliptic curves lie outside a thin set) the arithmetically significant constants appearing in the asymptotic. This allows us to provide computational evidence for the conjecture. 
\end{abstract}

\section{Introduction} \label{section:introduction}

Let $E_{/\bQ}$ be an elliptic curve and let $p$ denote any prime of good reduction of $E$. The coefficients $a_p(E)$, which for such primes may be defined as the trace of a choice of absolute Frobenius $\Frob_p\in\Gal(\overline{\mathbb{Q}}/\mathbb{Q})$ acting on the Tate module $T_\ell(E)$ of $E$ for $\ell\neq p$, are of central importance in understanding the arithmetic of elliptic curves and related questions. For instance, these integers determine the size of the group over finite fields since $\vert E(\mathbb{F}_p)\vert=1+p-a_p(E)$, and via modularity $a_p(E)$ also equals the $p$-th Fourier coefficient of the weight $2$ newform $f_E$ attached to $E$. They are well-known to satisfy the Hasse bound $\vert a_p(E) \vert \leq 2\sqrt{p}$, and several finer questions about their distribution have been considered. We focus here on a classical conjecture of S. Lang and H. Trotter \cite{LangTrotter} which examines the density of primes whose Frobenius trace $a_p(E)\in\mathbb{Z}$ is equal to a fixed integer. When $E$ does not have complex multiplication, for any trace $t\in \bZ$ they conjecture the asymptotic  
\begin{equation} \label{originallangtrotter}
    \pi_{E,t}(x):=\vert\{p\leq x\colon p\nmid N_E\text{ and }a_p(E)=t\}\vert \sim C(E,t)\cdot \frac{\sqrt x}{\log x},
\end{equation}
where $C(E,t)$ is an explicit constant depending on the compatible system of Galois representations attached to $E$ and where $N_E$ denotes the conductor of $E$. When $E$ has complex multiplication and $t\neq 0$ they similarly conjecture an asymptotic $\pi_{E,t}(x)\sim C'(E,t)\cdot\frac{\sqrt x}{\log x}$, whereas supersingular primes, for which $a_p(E)=0$, can be seen to occur half of the time. Despite some results in its direction, the Lang-Trotter conjecture remains open. \par
More generally, it is natural to study densities of primes with fixed Frobenius trace for higher dimensional motives and compatible systems of $\ell$-adic Galois representations. Considering principally polarized abelian varieties $A_{/\bQ}$ of dimension $g$, the action of the Galois group on torsion points gives rise to Galois representations
\begin{align*}
    \rho_A\colon \Gal(\overline{\mathbb{Q}}/\mathbb{Q}) &\to \GSp_{2g}(\hat\bZ) \\
    \rho_{A,m}\colon \Gal(\overline{\mathbb{Q}}/\mathbb{Q}) &\to \prod_{\ell\mid m}\GSp_{2g}(\bZ_{\ell}) \\
    \bar\rho_{A,m}\colon \Gal(\overline{\mathbb{Q}}/\mathbb{Q}) &\to \GSp_{2g}(\bZ/m\bZ) 
\end{align*}
for any positive integer $m$. The analogue to Lang and Trotter's conjecture for those abelian varieties where $\im(\rho_A)\subset \GSp_{2g}(\hat\bZ)$ is open, which we shall call \emph{generic}, was recently formulated and studied in \cite{CDSS}. 
In our paper, we shift the focus to non-generic abelian varieties and abelian surfaces in particular. With the goal of better understanding the distribution of Frobenius traces in the non-generic case, we examine abelian surfaces that are $\bQ$-isogenous to a product of elliptic curves defined over $\bQ$ and without complex multiplication.\par
Let us first give a more precise formulation of the existing conjectures. For a prime $p$ of good reduction of $A$,  denote more generally $a_p(A) = \tr (\rho_{A,\ell}(\Frob_p))\in \bZ$ (well-defined for $\ell \neq p$). 
The conjectures study for a fixed trace $T$ the asymptotics of  
\[
	\pi_{A,T}(x) := \vert\{p\leq x : p\nmid N_A\text{ and }a_p(A) = T\}\vert,
\]
where $N_A$ denotes the conductor of $A$. The traces $a_p(A)$ are a sum of $2g$ Weil numbers of modulus $\sqrt{p}$ and thus $|a_p(A)| \leq 2g\sqrt{p}$. The distribution of the normalized traces $a_p(A)/\sqrt{p}$ is then governed by a compact real Lie group $K\subset U(2g)$ known as the Sato-Tate group of $A$, which is determined by the $\ell$-adic Galois representations $\rho_{A,\ell}$ attached to $A$.  
The Haar measure on $K$ gives rise to a distribution $\Phi_K\colon [-2g,2g]\to \bR$ for the trace on conjugacy classes of $K$. 
The equidistribution assumption then states:
\begin{conjecture}[Sato--Tate, Serre]\label{conj:equidistrib}
Let $A_{/\bQ}$ be a generic abelian $g$-fold. Then
$\frac{a_p(A)}{\sqrt p}$ is equidistributed in $[-2g,2g]$ with respect to 
$\Phi_{\USp(2g)}$. 
\end{conjecture}

Its content is one of the main statements of the Sato-Tate conjecture and is therefore known for elliptic curves over $\mathbb{Q}$, but for $g\geq 2$ is not known in general. Here and throughout the paper, for any positive integer $m$, any integer $T$ and any algebraic subgroup $G\subseteq\GSp_{2g}$, we use the notation
\begin{equation} \label{defofGSp2gandHsubA}
\begin{split}
G(m) :=& G(\mbz/m\mbz), \\
G(m,T) :=& \{ g \in G(m) : \tr g \equiv T \mod m \}, \\
H_A :=& \im(\rho_{A} ) \subseteq \GSp_{2g}(\hat{\mbz}), \\
H_A(m) :=& \im(\bar\rho_{A,m}) \subseteq \GSp_{2g}(\mbz/m\mbz), \\
H_A(m,T) :=& \{ h \in H_A(m) : \tr h \equiv T \mod m \}.
\end{split}
\end{equation}
Also, provided $H_A=\im (\rho_{A})$ is an open subgroup of $G(\hat{\bZ})\subseteq \GSp_{2g}(\hat{\mbz})$, we let $m_A$ denote the \emph{conductor} of $H_A$, i.e. the least positive integer $m$ such that in the 
commutative diagram 
\begin{equation}
\label{conductorofimage}
\begin{tikzcd}
    \Gal(\overline{\mathbb{Q}}/\mathbb{Q}) \ar[r, "\rho_A"] \ar[dr, "\bar\rho_{A,m}"']
        & G(\hat\bZ) \ar[d, "\rred"] \\
    & G(m) ,
\end{tikzcd}
\end{equation}
we have $H_A := \im(\rho_A) = \rred^{-1} \left( H_A(m) \right)$, and we abbreviate for nonzero trace $T$
\begin{equation} \label{defofmsubAT}
m_{A,T} := m_A\prod_{\ell\vert m_A}\ell^{v_\ell(T)}.
\end{equation}

We state the existing conjectures, when $\im(\rho_A)\subseteq G(\hat\bZ)$ is open for $G=\GSp_{2g}$, uniformly for $T\neq 0$:

\begin{conjecture}[Lang--Trotter ($g = 1$), Cojocaru--Davis--Silverberg--Stange ($g \geq 2$)]\label{conj:LTsummary}
Let $A_{/\bQ}$ be a generic abelian $g$-fold. Assume the equidistribution stated in Conjecture \ref{conj:equidistrib} holds for $A$. Then for $T\neq 0$ we have
\[
\pi_{A,T}(x) \sim C(A,T)\cdot \frac{\sqrt x}{\log x},
\]
where the constant $C(A,T) \geq 0$ is given by
$$ C(A,T)=\underbrace{
		2\Phi_{\USp(2g)}(0)
	}_{{\begin{substack} { \text{Sato--Tate} \\ \text{factor} } \end{substack}}} 
	\cdot \underbrace{
		\frac{m_{A,T} \vert H_A(m_{A,T},T) \vert}{\vert H_A(m_{A,T})\vert}
	}_\text{exceptional factor}
	\cdot \underbrace{
		\prod_{\ell\nmid m_A} \frac{\ell^{v_\ell(T)+1} \vert \GSp_{2g}(\ell^{v_\ell(T)+1},T) \vert}{\vert \GSp_{2g}(\ell^{v_\ell(T)+1})\vert}
	}_\text{universal factors}. $$
\end{conjecture}

Conjecture \ref{conj:LTsummary} contains Lang and Trotter's original conjecture \cite{LangTrotter} when $A$ is an elliptic 
curve without complex multiplication and therefore is generic by a well-known result of J-P. Serre \cite{Serre1972}, whereas in the higher dimensional case Conjecture \ref{conj:LTsummary} is due to \cite{CDSS} for nonzero trace.  \par 

Consider now abelian surfaces $\bQ$-isogenous to products of non-CM elliptic curves $E_1\times E_2$ defined over $\bQ$. 
\begin{remark}
Since all but finitely many Frobenius traces $a_p(A)$ are invariant under $\bQ$-isogeny, when $A\sim_{\mathbb{Q}} E_1\times E_2$ we have that $\pi_{A,T}(x)=\pi_{E_1\times E_2}(x)+O(1)$. Hence for our purposes we may, and by and large in this paper do, just work with products of elliptic curves. 
\end{remark}
We also assume $E_1$ and $E_2$ are non-isogenous over $\overline{\mbq}$, so that denoting
\begin{equation} \label{defofG}
\begin{split}
    G:=& \GL_2 \times_{\det} \GL_2 = \{(g_1,g_2)\in \GL_2\times \GL_2\colon \det(g_1)=\det(g_2)\},
\end{split}
\end{equation}
the Galois image subgroup $\im(\rho_{E_1\times E_2})$ is an open subgroup of $G(\hat{\mbz})$ by \cite[Th.~6]{Serre1972}. Fixing an embedding $G \hookrightarrow \GSp_4$, we may view $H_{E_1\times E_2}:=\im(\rho_{E_1\times E_2})\subseteq G(\hat{\mbz})\subseteq \GSp_{4}(\hat{\mbz})$ and use the notations in \eqref{defofGSp2gandHsubA}.
 An explicit embedding $G \hookrightarrow \GSp_4$ is for instance given by
 \[
 \left( \begin{pmatrix} a & b \\ c & d \end{pmatrix}, \begin{pmatrix} a' & b' \\ c' & d' \end{pmatrix} \right) \mapsto \begin{pmatrix} a & 0 & b & 0 \\ 0 & a' & 0 & b' \\ c & 0 & d & 0 \\ 0 & c' & 0 & d' \end{pmatrix}. 
 \]
 
 In particular, this defines via \eqref{conductorofimage} a conductor $m_{{E_1\times E_2}}$ of the open subgroup $H_{{E_1\times E_2}} \subseteq G(\hat{\mbz})$. 
Furthermore, in this setting the equidistribution result for the normalized traces of Frobenius with respect to the Sato-Tate measure is known (see e.g., \cite[Th.~5.4]{Harris2009}):

\begin{proposition} \label{SatoTateProposition}
Let $E_{1/\bQ}$, $E_{2/\bQ}$ be non-CM and non-isogenous over $\overline\bQ$. 
Then $\frac{a_p(E_1\times E_2)}{\sqrt p}$ is equidistributed with respect 
to $\Phi_{\SU(2)^{\times 2}} = \Phi_{\SU(2)}\ast \Phi_{\SU(2)}$ in $[-4,4]$. 
\end{proposition}

We adapt Lang and Trotter's heuristic to this situation in Section \ref{section:heuristic} of this paper, proving convergence of the constant appearing in the asymptotic. This leads to the analogous conjecture: 
\begin{conjecture}\label{conj:main}
Let $E_{1/\bQ}$, $E_{2/\bQ}$ be non-CM elliptic curves that are not $\overline{\bQ}$-isogenous. Let $A\sim_{\mathbb{Q}} E_1\times E_2$ be an abelian surface isogenous over $\mathbb{Q}$ to their product and let $T\in \mathbb{Z}$, $T\neq 0$. Then 
\begin{equation}\label{eqn:asymptotic}
\pi_{A,T}(x) \sim C(E_1\times E_2,T)\cdot\frac{\sqrt x}{\log x}
\end{equation}
	as $x \rightarrow \infty$, where the constant $C(E_1\times E_2,T) \geq 0$ is given by
\[
C(E_1\times E_2,T)=
	\underbrace{
		2\Phi_{\SU(2)^{\times 2}}(0)
	}_{{\begin{substack} { \text{Sato--Tate} \\ \text{factor} } \end{substack}}} \cdot
	\underbrace{
		\frac{m_{E_1\times E_2,T} \vert H_{E_1\times E_2}(m_{E_1\times E_2,T},T) \vert}{\vert H_{E_1\times E_2}(m_{E_1\times E_2,T})\vert}
	}_\text{exceptional factor} \cdot
	\underbrace{
		\prod_{\ell\nmid m_{E_1\times E_2}} \frac{\ell^{v_\ell(T)+1} \vert G(\ell^{v_\ell(T)+1},T) \vert}{\vert G(\ell^{v_\ell(T)+1})\vert}
	}_\text{universal factors}
\]
where $G=\GL_2 \times_{\det} \GL_2$ and notations are as in \eqref{defofGSp2gandHsubA}, and where $m_{E_1\times E_2}$ and $m_{E_1\times E_2,T}$ are defined as in \eqref{conductorofimage} and \eqref{defofmsubAT}, respectively.
In case $C(E_1\times E_2,T)=0$ we take the asymptotic \ref{eqn:asymptotic} to mean that there are only finitely many primes $p$ with $a_p(A)=T$. 
\end{conjecture}
The Sato--Tate factor appearing above is explicitly given by:
\[
	2\Phi_{\SU(2)^{\times 2}}(0) 
	= 2\int_{-2}^2 \left(\frac{\sqrt{4-t^2}}{2\pi}\right)^2\, \dd t 
	= \frac{1}{2\pi^2} \int_{-2}^2 (4-t^2)\, \dd t 
	= \frac{16}{3\pi^2} .
\]
 Furthermore, we are able in this setting to compute explicit formulas for the factors involved in the constant $C(E_1\times E_2,T)$ in the case where $\im (\rho_{E_1\times E_2})$ and hence $\im(\rho_A)$ are as large as possible.  More precisely, the pair of elliptic curves $(E_1,E_2)$ is called a \emph{Serre pair} if $\left[ G(\hat{\bZ}) : \im (\rho_{E_1\times E_2}) \right] = 4$, noting we always have $\left[ G(\hat{\bZ}) : \im (\rho_{E_1 \times E_2}) \right] \geq 4$ (see, e.g., \cite{MR3071819} or Section \ref{section:bad primes}). For primes not dividing the conductor of the open subgroup $\im(\rho_{E_1\times E_2}) \subseteq G(\hat{\mbz})$ (universal primes), the explicit evaluation of the local factor is done in Section \ref{section:explicit constants} and is valid for \emph{any} pair $(E_1,E_2)$ of non-CM elliptic curves which are not $\overline{\bQ}$-isogenous. Section \ref{section:bad primes} then deals with the more delicate remaining factor, and the formulas are only valid when $(E_1,E_2)$ is a Serre pair.  We summarize our results here: 
 
 \begin{theorem}\label{thm:mainformulas}
 Let $E_{1/\bQ}$, $E_{2/\bQ}$ be non-CM elliptic curves that are not $\overline{\bQ}$-isogenous. Then in the notations of Conjecture \ref{conj:main}, for $T\neq 0$ we obtain: 
 \begin{enumerate}
     \item 
     The factors away from $m_{E_1\times E_2}$ are given by
$$\frac{\ell^{v_\ell(T)+1} \vert G(\ell^{v_\ell(T)+1},T) \vert}{\vert G(\ell^{v_\ell(T)+1})\vert}=
\begin{cases}
	\frac{\ell(\ell^4-\ell^3-2\ell^2+\ell+2)}{(\ell^2-1)^2(\ell-1)}&\text{ if }\ell\nmid m_{E_1\times E_2} \cdot T\\
\frac{\ell^6-\ell^4-3\ell^3-\ell}{\ell^6-2\ell^4+\ell^2}+\frac{E(\ell,T)+\delta_2(v_\ell(T))(\ell^3-1)}{\ell^{4e_\ell(T)+4}(\ell^2-1)^2}&\text{ if }\ell\nmid m_{E_1\times E_2}, \ell\mid T,\\
\end{cases}
$$
with notations for $e_\ell(T), E(\ell,T), \delta_2(v_\ell(T))$ specified below. 

\item Moreover, if $(E_1, E_2)$ is a Serre pair, we compute the factor at the conductor $m_{E_1\times E_2}$ of the open subgroup $\im \rho_{E_1\times E_2} \subseteq G(\hat{\mbz})$. Let
$$D_i=\begin{cases}
\vert \Delta_{i}\vert& \text{if }  \Delta_{i}\equiv 1\mod 4\\
4\vert \Delta_{i}\vert& \text{otherwise}
\end{cases}
$$
where $\Delta_i$ denotes the square-free part of the discriminants of the curves $E_i$. Then $m_{E_1\times E_2}=\lcm(2,D_1,D_2)$. Abbreviating $m=m_{{E_1\times E_2},T}=m_{E_1\times E_2}\prod_{\ell\vert m_{E_1\times E_2}}\ell^{v_\ell(T)}$ and writing $D_i'$ and $m'$ for the odd parts of these quantities, when $\vert D_1'\vert \neq \vert D_2'\vert$ the factor $\frac{m \vert H_{E_1\times E_2}(m,T) \vert}{\vert H_{E_1\times E_2}(m)\vert}$ equals:
\small{$$
\frac{m\cdot (m')^5\varphi(m')}{\left( \rad(m') \right)^5\vert G(m)\vert} \left( \mathfrak{S}_{{\boldsymbol{\psi}}_2}(T) \prod_{\ell \mid m'} \left( A_\ell(T) + 1\right) 
+ 2^{\omega\left( \frac{m'}{|D_1'|} \right) } \mathfrak{S}_{{\boldsymbol{\psi}}_2}^{(1)}(T) \prod_{\ell \mid D_1'} B_\ell(T) \\
+ 2^{\omega\left( \frac{m'}{|D_2'|} \right)} \mathfrak{S}_{{\boldsymbol{\psi}}_2}^{(2)}(T) \prod_{\ell \mid D_2'} B_\ell(T) \right)
$$}
with $\vert G(m)\vert=\prod_{\ell\vert m}\ell^{7v_{\ell}(m)-5}(\ell^2-1)^2(\ell-1)$. We obtain a similar formula when $\vert D_1'\vert = \vert D_2'\vert$ (see Proposition \ref{badprimesmainprop}). 
\end{enumerate}
Throughout these results we have used the following notations:
\begin{equation} 
\begin{split}
e_\ell(T) &:= \lfloor (v_\ell(T)-1)/2 \rfloor, \quad\quad \gd_n(m) :=
\begin{cases}
1 & \text{ if } n \text{ divides } m \\
0 & \text{ otherwise,}
\end{cases} \\
E(\ell,T) &:= \frac{\ell^3(\ell+1)(\ell^{4e_\ell(T)+4}-1) + (\ell^4 + \ell^3 + 2\ell^2 + \ell - 1)(\ell^{4e_\ell(T)}-1)}{(\ell^2+1)(\ell+1)}, \\
A_\ell(T) &:= 
\begin{cases}
 \ell^5 - \ell^3 - 3\ell^2 + 1 + \frac{2E(\ell,T)}{\ell^{4e_\ell(T)+3}} + \gd_2(v_\ell(T)) \frac{(\ell^3-1)}{\ell^{4e_\ell(T)+2}} & \text{ if $\ell \mid T$} \\ 
(\ell^6 - \ell^5 - 2\ell^4 + \ell^3 + 2\ell^2)/(\ell-1) & \text{ if $\ell \nmid T$,}
\end{cases} \\
B_\ell(T) &:= 
\begin{cases}
-3\ell^2 + 1 + \frac{2E(\ell,T)}{\ell^{4e_\ell(T)+3}} + \gd_2(v_\ell(T)) \frac{(\ell^3-1)}{\ell^{4e_\ell(T)+2}} & \text{ if $\ell \mid T$} \\
\ell^2/(\ell-1) & \text{ if $\ell \nmid T$},
\end{cases} \\
\mathfrak{S}_{{\boldsymbol{\psi}}_2}(T) &:=
\left| G\left(2^{v_2(m)},T\right)_{\im {\boldsymbol{\psi}}_2 \subseteq \im {\boldsymbol{\psi}}_{m'}} \right|, \\
\mathfrak{S}_{{\boldsymbol{\psi}}_2}^{(i)}(T) &:=
\left| G\left(2^{v_2(m)},T\right)_{\im {\boldsymbol{\psi}}_2 \subseteq \im {\boldsymbol{\psi}}_{m'}}^{\psi_2^{(i)} = 1} \right| - \left| G\left(2^{v_2(m)},T\right)_{\im {\boldsymbol{\psi}}_2 \subseteq \im {\boldsymbol{\psi}}_{m'}}^{\psi_2^{(i)} = -1} \right|,
\end{split}
\end{equation}
 where in the last two expressions, the subscripts and superscripts denote the subset of $G\left(2^{v_2(m)},T\right)$ cut out by the given conditions, where $\psi^{(i)}$ is a product of a Kronecker symbol and the non-trivial character of order two on $\GL_2(\mathbb{Z}/2\bZ)$, and a subscript $n$ indicates restriction to the $n$-primary part. See Section \ref{section:bad primes} for details. 
 \end{theorem}
 
It has been shown by the second author \cite{MR3071819,MR2563740} that almost all elliptic curves and pairs of curves (ordered by height) over $\bQ$ are Serre curves or Serre pairs, respectively, so that the restriction on computing the exceptional factor in Theorem \ref{thm:mainformulas} is reasonable. In general, this factor is quite intricate to compute and it would likely be hard to produce any kind of formula that works in much greater generality. \par 
Several remarks concerning Conjecture \ref{conj:main} and Theorem \ref{thm:mainformulas} are in order: 
\begin{remarks} \label{remarksss}

\begin{enumerate}[(i)]
    \item One can similarly formulate a conjecture for arbitrary products of non-CM, pairwise non-isogenous elliptic curves $E_1\times\cdots\times E_k$. \\
     \item When two curves are $\bQ$-isogenous, then for all primes $p\nmid N_{E_1}N_{E_2}$ we have $a_p(E_1)=a_p(E_2)$ and the distribution is then already predicted by Lang and Trotter's original conjecture. \\
    \item For trace $T=0$, we expect similarly a constant $C(E_1\times E_2,0)$ such that $\pi_{A,0}(x)\sim C(E_1\times E_2,0)\cdot\sqrt{x}/\log(x)$, but have not computed $C(E_1\times E_2,0)$ explicitly. \\
    \item One may also formulate a conjecture when $E_1$, $E_2$, or both curves have complex multiplication after taking into account the correct Sato-Tate measure. For CM elliptic curves, Lang and Trotter still predict that for $T \neq 0$ there is a constant $C'(E,T)$ such that $\pi_{E,T}(x)\sim C'(E,T)\cdot\sqrt{x}/\log(x)$. On the other hand, when $T=0$, M. Deuring \cite{Deuring} showed that half of the primes have supersingular reduction, i.e. that
    $$\pi_{E,0}(x)\sim 1/2\cdot x/\log(x)$$
    for $E$ with complex multiplication. This reflects the fact that the Sato-Tate measure has a discrete part of mass $1/2$ at zero. When, say, both $E_1$ and $E_2$ are CM and non-isogenous, we thus arrive at at least $1/4$ of primes having trace $T=0$ and the heuristic for general traces has to be adapted according to the Sato-Tate group (e.g., when both curves have CM the Sato-Tate group is $F_{a,b}$ in the notations of \cite{KedlayaFite+} with connected component $U(1)\times U(1)$). \\
    \item It can indeed occur that $C(E_1\times E_2,T)=0$ for some $T$ because $\im(\overline{\rho}_{{E_1\times E_2},m})\subset G(m)$ is small enough not to contain any pairs of matrices with traces summing to $T$. For example, if $E_1$ and $E_2$ both have rational $2$-torsion (implying that $2$ divides $m_{E_1 \times E_2}$), then at all odd primes $p \nmid N_{E_1}N_{E_2}$ this forces $a_p(E_1)\equiv a_p(E_2)\equiv 0 \mod 2$, and the constant $C(E_1\times E_2,T)=0$ vanishes for $T$ odd. For more details and related results and examples, including pairs $(E_1,E_2)$ for which the original Lang-Trotter constants $C(E_1,T)$ and $C(E_2,T)$ are non-zero for all $T$ but nevertheless $C(E_1 \times E_2,T) = 0$ for some $T$, see Section \ref{missingtracessection} below.  The fact that this arithmetic information is captured by the exceptional factor in Conjecture \ref{conj:main} should indicate that it is difficult to compute in general. Still, for Serre pairs, our results show the constant $C(E_1\times E_2,T)$ never vanishes for nonzero trace (see Section \ref{missingtracessection}).
    
   \end{enumerate}
\end{remarks}

 Compared to the study of the Lang-Trotter Conjecture for a single elliptic curve over $\bQ$, there are various additional complications that arise. For instance, one has to reckon with possible entanglements of the division fields of $E_1$ and $E_2$ when the Galois extension $\bQ(\zeta_m)$ is strictly contained in $\bQ(E_1[m])\cap \bQ(E_2[m])$. This increases the index $\left[ G(\hat{\bZ}) : \im (\rho_{E_1 \times E_2})\right]$ which, as Remark \ref{remarksss}.(v) already indicates, can lead to additional congruence obstructions which prevent the realization of some $T\in\bZ$ as Frobenius traces. See Section \ref{missingtracessection} and Proposition \ref{prop:Serrenopair} for some results illustrating this phenomenon. It is therefore also not surprising that the computation of the exceptional factor of the constant
 becomes even more onerous when considering the product $E_1 \times E_2$. 
 Furthermore, it is harder to compute the factors $\lim_{n\to\infty}\ell^n\vert G(\ell^n,T)\vert/\vert G(\ell^n)\vert$ of the Lang-Trotter constant, which do not immediately stabilize for primes $\ell \vert T$. See for instance the difficulties appearing in A. Akbary and J. Parks' \cite{AkbaryParks} work where they study for $E_1\times E_2$ the density of primes with both Frobenius traces fixed \cite[Conjecture 1.2.]{AkbaryParks}.\par

 With the formulas of Theorem \ref{thm:mainformulas} in hand, we are then able in Section \ref{section:numerical} to provide and discuss computational evidence for Conjecture \ref{conj:main}. The computations were performed using the MAGMA computational algebra system \cite{MAGMA} as well as the SageMath mathematics software system \cite{SAGEMATH} and can be reproduced with either software using the results of this paper.  As hoped, we find a strong correlation between the predicted and actual counts of primes with fixed Frobenius trace, providing the first such evidence for higher-dimensional abelian varieties in the non-generic case. Moreover, as predicted, the counts are sensitive to the prime factorization of the trace and to the arithmetic of the abelian surface, lending further credibility to the intricate constant $C(E_1\times E_2)$ appearing in the asymptotic. Extensive computations are more accessible here than in the setting of Akbary and Parks \cite[Remark 1.3.(iv)]{AkbaryParks} or in the setting of generic abelian varieties of dimension $\geq 2$ of \cite{CDSS}, so this paper focuses primarily on explicit calculations. However, just as for elliptic curves, one can establish non-trivial upper bounds for $\pi_{A,T}(x)$ and we expect that one can prove along the lines of \cite{DavidPapp} that Conjecture \ref{conj:main} holds on average. Toward that end, we present an upper bound in Theorem \ref{upperboundtheorem} and relegate the rest to future work. \par 

Similar densities for general compatible systems of Galois representations were studied in \cite{MurtyFrobdistribs} and for primes where the Frobenius fields of two elliptic curves coincide in \cite{BaierPitankar}. See also \cite{KatzLTrevisited} for a discussion of the Lang-Trotter conjecture for elliptic curves and for results over function fields. Returning to abelian surfaces, there is now a complete classification of the Sato-Tate groups \cite{KedlayaFite+} and we hope that the study of the distribution of Frobenius traces for non-generic abelian surfaces can be completed in the future. Indeed, this was our initial motivation, but already the case of products of elliptic curves provided enough substance to warrant this paper.

\par
\subsection*{Acknowledgements}
We would like to express our thanks to Alina Carmen Cojocaru and Andrew Sutherland for pointing us toward this project and engaging in enlightening discussions, and to the organizers of the Arizona Winter School for providing the stimulating environment where the research was initiated. We would also like to thank Daniel Miller and Gabor Wiese for helpful conversations on some aspects of this paper.

\section{The Lang-Trotter heuristic}\label{section:heuristic}
In this section we adapt Lang and Trotter's original heuristic in \cite{LangTrotter} to the situation of products of non-$\overline{\bQ}$-isogenous elliptic curves defined over $\mathbb{Q}$ without complex multiplication. We follow Lang and Trotter's original reasoning as well as the formulation for generic abelian varieties in \cite{CDSS}. We normalize the continuous part of the Sato-Tate density function $\Phi:[-2g,2g]\to \mathbb{R}$ so that $\Phi$ is of total mass one. In general, if the Sato-Tate measure were to have a discontinuous part (e.g., for CM elliptic curves one has $\mu_{disc}=\frac{1}{2}\delta_o$) the heuristic would need to be adapted (in particular, for $T \neq 0$, the value $\Phi(0)$ would need to be replaced by the appropriate limiting value at zero).  \par
\subsection{The probabilistic model}
Fix integers $m \geq 1$ and $T\in\mathbb{Z}$ and recall the notations for the group $G(m) := \GL_2(\mbz/m\mbz) \times_{\det} \GL_2(\mbz/m\mbz)$, the subgroup $H_A(m)$ and the subset $H_A(m,T)$ in \eqref{defofGSp2gandHsubA}. In what follows, we take $A=E_1\times E_2$ for simplicity and typically drop the subscript for brevity, writing $H(m)$ and $H(m,T)$ when $A=E_1\times E_2$ is implicit. 
The ratio
$$F_m(T)=\frac{m\vert H(m,T)\vert}{\vert H(m)\vert}$$
is $m$ times the probability that $\overline{\rho}_{A,m}(\Frob_p)$ has trace $T$ modulo $m$, obtained by applying the \v{C}ebotarev density theorem to the division field generated by the $m$-torsion (we may ignore primes $p\mid m$ or which ramify). The additional factor $m$ ensures that $F_m(T)$ has average value $1$ over the congruence classes of $T$ modulo $m$.
We also define the probability weights $\mu_{p,m}(T)$ by
\begin{equation} \label{formoffsubmofTandp}
\mu_{p,m}(T) :=
\begin{cases}
c_{p,m}\Phi\left(\frac{T}{\sqrt{p}}\right)F_m(T) & \text{ if } \vert T \vert \leq 2g\sqrt{p} \\
0 & \text{ if } \vert T \vert > 2g\sqrt{p},
\end{cases}
\end{equation}
where $c_{p,m}$ is the unique constant chosen so that 
\[
\sum_{ T \in \bZ } \mu_{p,m}(T) = 1,
\]
namely
\[
c_{p,m} := \frac{1}{\sum_{ \vert T\vert \leq 2g\sqrt{p}}\Phi\left( \frac{T}{\sqrt{p}} \right) F_m(T)}.
\]
The \emph{main probabilistic assumption}, as found in \cite{LangTrotter}, is that the function $T \mapsto \mu_{p,m}(T)$ models the probability of the event $\{ a_p(A)=T \}$, taking into account the Sato-Tate distribution and taking into account the \v{C}ebotarev density theorem for \emph{only} the $m$-th division field; in particular, it assumes that these two are statistically independent from one another.
We presently show that this model is in fact consistent with these two equidistribution results.

\begin{Lemma}\label{Riemannsum}
Fix integers $T_0$ and $m\geq 1$, and let $I \subseteq [-2g,2g]$ be any interval. We have
$$
\lim_{p\to\infty}\frac{m}{\sqrt{p}}\sum_{{\begin{substack} { T\equiv T_0 \; (m) \\ \frac{T}{\sqrt{p}} \in I } \end{substack}}} \Phi\left(\frac{T}{\sqrt{p}}\right)= \int_I \Phi(s) \, ds.
$$
\end{Lemma}
\begin{proof}
When $p$ is much larger than $m$, this approaches a Riemann sum computing $\int_I \Phi(s)\, ds$ by splitting the interval $I$ into pieces of length $m/\sqrt{p}$.
\end{proof}
Taking $I = [-2g,2g]$, we obtain the following corollary.
\begin{Corollary} \label{Iequalsminusonetoonecor}
For any $m \geq 1$ and $T_0 \in \mbz$ we have
\[
\lim_{p \to \infty} \frac{m}{\sqrt{p}} \sum_{{\begin{substack} { T\equiv T_0(m) \\ \vert T\vert\leq 2g\sqrt{p} } \end{substack}}} \Phi\left( \frac{T}{\sqrt{p}} \right) = 1.
\]
\end{Corollary}

\begin{Corollary} \label{asymptoticofcsubpmcor}
For all $m\geq 1$ we have
$$
\lim_{p\to\infty}\sqrt{p}\cdot c_{p,m}=1.
$$
\end{Corollary}
\begin{proof}
This is \cite[Lemma 21]{CDSS}, mutatis mutandis.  We include the details here for completeness.  By choice of $c_{p,m}$, and noting that the factor $F_m(T)$ only depends on $T$ modulo $m$, we have
\[
\begin{split}
1 = \sum_{T \in \mbz} \mu_{p,m}(T) 
&= c_{p,m}\sum_{{\begin{substack} { T \in \mbz \\ |T| \leq 2g\sqrt{p}} \end{substack}}} \Phi\left( \frac{T}{\sqrt{p}} \right) F_m(T) \\
&= \sqrt{p} \cdot c_{p,m} \sum_{T_0 \in \mbz/m\mbz} \frac{1}{m}F_m(T_0) \cdot \frac{m}{\sqrt{p}} \sum_{{\begin{substack} { T\equiv T_0(m) \\ \vert T\vert\leq 2g\sqrt{p} } \end{substack}}} \Phi\left( \frac{T}{\sqrt{p}} \right).
\end{split}
\]
Thus, using Corollary \ref{Iequalsminusonetoonecor}, we have
\[
\lim_{p \to \infty} \frac{1}{\sqrt{p} \cdot c_{p,m}} = \sum_{T_0 \in \mbz/m\mbz} \frac{1}{m}F_m(T_0) \lim_{p \to \infty} \frac{m}{\sqrt{p}}  \sum_{{\begin{substack} { T\equiv T_0(m) \\ \vert T\vert\leq 2g\sqrt{p} } \end{substack}}} 
\Phi\left( \frac{T}{\sqrt{p}} \right) = 1,
\]
proving the corollary.
\end{proof}
\begin{Corollary} \label{cpmproposition}
The choice of model \eqref{formoffsubmofTandp} agrees with the \v{C}ebotarev density theorem and the Sato-Tate equidistribution law as $p \to \infty$.
\end{Corollary}
\begin{proof}
By \eqref{formoffsubmofTandp} together with Corollaries \ref{Iequalsminusonetoonecor} and \ref{asymptoticofcsubpmcor}, we find that
\[
\begin{split}
\lim_{p\to\infty}\mu_{p,m}\left( \left\{ T\in\mathbb{Z} : \;  T\equiv T_0\; (m) \right\} \right) 
&= \lim_{p\to \infty} \sum_{{\begin{substack} { T \in \mbz \\ T \equiv T_0 \; (m)} \end{substack}}} c_{p,m} \Phi\left( \frac{T}{\sqrt{p}} \right) F_m(T_0) \\
&= \frac{1}{m}F_m(T_0) \lim_{p\to\infty} \sqrt{p} \cdot c_{p,m} \cdot \frac{m}{\sqrt{p}} \sum_{{\begin{substack} { T \in \mbz \\ T \equiv T_0 \; (m)} \end{substack}}} \Phi\left( \frac{T}{\sqrt{p}} \right) \\
&=
\frac{1}{m} F_m(T_0).
\end{split}
\]
By the law of large numbers, we recover that the density of primes with Frobenius trace $T_0$ modulo $m$ is asymptotically $\displaystyle \frac{1}{m} F_m(T_0)$, in accordance with the \v{C}ebotarev density theorem.
Similarly, for any interval $I\subseteq [-2g,2g]$, by Lemma \ref{Riemannsum} and Corollary \ref{asymptoticofcsubpmcor} we have
\[
\begin{split}
\lim_{p\to\infty}\mu_{p,m}\left( \left\{
\frac{a_p(A)}{\sqrt{p}}\in I
\right\} \right)
&= \lim_{p \to \infty} \sqrt{p} \cdot c_{p,m} \sum_{T_0 \in \mbz/m\mbz} \frac{1}{m}F_m(T_0) \cdot \frac{m}{\sqrt{p}} \sum_{{\begin{substack} {T \equiv T_0 \; (m) \\ \frac{T}{2g\sqrt{p}} \in I} \end{substack}}} \Phi\left( \frac{T}{\sqrt{p}} \right) \\
&=\int_I\Phi(s)ds.
\end{split}
\]
Again the law of large numbers gives
$$
\lim_{x\to \infty}\frac{\vert\{p\leq x : \; \frac{a_p(A)}{\sqrt{p}} \in I\}\vert}{\pi(x)}=\int_I\Phi(s)ds,
$$
which by our choice of $\Phi$ agrees with the Sato-Tate equidistribution.
\end{proof}

We have seen that the measure $\mu_{p,m}$ on $\mbz$ models the probability of the event that $a_p(A)$ takes on a fixed value $T$, i.e. 
\[
\Prob\{ a_p(A) = T \} \approx \mu_{p,m}(T),
\]
and that this model agrees with the Sato-Tate equidistribution law and with the \v{C}ebotarev density theorem for \emph{only} the $m$-th division field of $A$.  It is natural to desire a limiting measure $\mu_p$ on $\mbz$ modeling the probability of the event $\{ a_p(A) = T \}$ which agrees with Sato-Tate equidistribution as well as with the \v{C}ebotarev density theorem for \emph{all} of the division fields of $A$ simultaneously.  This is (heuristically) achieved by taking the point-wise limit as $m$ approaches infinity via any sequence of positive integers that is co-final with respect to divisibility.  We denote this limit by
$\displaystyle \mu_p(T) := \lim_{m\tilde{\to}\infty} \mu_{p,m}(T)$; for the sake of concreteness, we may take
\[
\mu_p(T) := \lim_{m \tilde{\to} \infty} \mu_{p,m}(T) = \lim_{n \rightarrow \infty} \mu_{p,m_n}(T) \quad\quad \left( m_n := \prod_{\ell \leq n} \ell^n \right).
\]
In Subsection \ref{constantconvergencesubsection}, we establish the above pointwise convergence of the  measures $\mu_{p,m}$ by computing the limit 
\[
F(t) := \lim_{m \tilde{\to} \infty} F_m(T) = \lim_{n \rightarrow \infty} F_{m_n}(T),
\]
which gives rise to the limiting probability measure 
\[
\mu_p(T) = c_p \Phi\left( \frac{T}{\sqrt{p}} \right) F(T),
\]
where
\[
c_p = \lim_{m \tilde{\to} \infty} c_{p,m} = \frac{1}{\sum_{ \vert T\vert \leq 2g\sqrt{p}}\Phi\left( \frac{T}{\sqrt{p}} \right) F(T)}.
\]
Note that, using Corollary \ref{asymptoticofcsubpmcor} and interchanging limits, we obtain
\begin{equation} \label{asymptoticforcp}
c_p \sim \frac{1}{\sqrt{p}} \quad \text{ as } \quad p \to \infty.
\end{equation}
Assuming that the above-mentioned limit interchange is justified and that $\mu_p(T)$ models $\Prob\{ a_p(A) = T \}$, we arrive at Conjecture \ref{conj:main} by writing
\begin{equation} \label{howtheheuristicworks}
\begin{split}
\vert\{p\leq x : \; p\nmid N_A, a_p(A)=T\}\vert
&\approx \sum_{p\leq x} \mu_p(T)\\
&= \sum_{p\leq x} c_p \Phi\left( \frac{T}{\sqrt{p}} \right) F(T) \\
& \sim F(T) \sum_{p\leq x} \frac{1}{\sqrt{p}} \cdot \Phi\left( \frac{T}{\sqrt{p}} \right),
\end{split}
\end{equation}
the last asymptotic as $x \rightarrow \infty$ being justified by \eqref{asymptoticforcp} and by our next lemma, which implies that the sum on the right-hand side approaches infinity with $x$.
\begin{Lemma} \label{asymptoticequalssqrtxoverlogxlemma}
We have the asymptotic
$$
\sum_{p\leq x} \frac{1}{\sqrt{p}} \cdot \Phi\left( \frac{T}{\sqrt{p}} \right) \; \sim \; 2\Phi(0)\cdot\frac{\sqrt{x}}{\log x},
$$
as $x \rightarrow \infty$.
\end{Lemma}
\begin{proof}
First, observe that 
\begin{equation} \label{startingproofofasymptoticlemma}
\sum_{p \leq x} \Phi\left( \frac{T}{\sqrt{p}} \right) \frac{1}{\sqrt{p}} = \sum_{p \leq x} \Phi\left( 0 \right) \frac{1}{\sqrt{p}} + \sum_{p \leq x} \left( \Phi\left( \frac{T}{\sqrt{p}} \right) - \Phi\left( 0 \right) \right) \frac{1}{\sqrt{p}}.
\end{equation}
Fix $\varepsilon > 0$.  By continuity of $\Phi$ at $0$, there is a $\delta > 0$ for which 
\[
\left| \frac{T}{\sqrt{p}} - 0 \right| < \delta \; \Longrightarrow \; \left| \Phi\left( \frac{T}{\sqrt{p}} \right) - \Phi(0) \right| < \varepsilon.
\]
We divide the second sum on the right-hand side of \eqref{startingproofofasymptoticlemma},
obtaining the bound
\begin{align*}
\left\vert\sum_{p \leq \frac{T^2}{\delta^2}}\left(\Phi\left(\frac{T}{\sqrt{p}}\right)-\Phi(0)\right)\frac{1}{\sqrt{p}} + \sum_{\frac{T^2}{\delta^2} < p\leq x}\left(\Phi\left(\frac{T}{\sqrt{p}}\right)-\Phi(0)\right)\frac{1}{\sqrt{p}} \right\vert
\leq C_{\delta, T} + 2\varepsilon \cdot \sum_{p \leq x} \frac{1}{2\sqrt{p}}.
\end{align*}
Finally, by partial summation, we have the asymptotic $\displaystyle \sum_{p\leq x}\frac{1}{2\sqrt{p}} \sim \frac{\sqrt{x}}{\log x}$.
Inserting this into \eqref{startingproofofasymptoticlemma}, we find that
\begin{equation*}\sum_{p\leq x}\Phi\left( \frac{T}{\sqrt{p}} \right)\cdot\frac{1}{\sqrt{p}}= 2\Phi(0) \cdot \frac{\sqrt{x}}{\log x}
+ O\left( \varepsilon\cdot\frac{\sqrt{x}}{\log x} \right),
\end{equation*}
with an absolute implied constant.  Since $\varepsilon > 0$ was arbitrary, this establishes the lemma.
\end{proof}
Combining Lemma \ref{asymptoticequalssqrtxoverlogxlemma} with \eqref{howtheheuristicworks}, we arrive at Conjecture \ref{conj:main} in the form
\[
\vert\{p\leq x : \; p\nmid N_A, a_p(A)=T\}\vert \; \sim \; 2\Phi(0) \cdot F(T) \cdot \frac{\sqrt{x}}{\log x} \quad\quad \left( x \rightarrow \infty \right).
\]
Finally, it remains to show that $\displaystyle
F(T) = \lim_{m\tilde{\to}\infty}F_m(T)$ converges, which is the subject of the next subsection. 

\begin{remark} \label{integralexpressionremark}
In fact, the above heuristics lead to the conjecture that, for fixed $T \in \bZ$, we have
\[
\pi_{E_1 \times E_2, T}(x) \sim \frac{C(E_1 \times E_2,T)}{2 \Phi(0)} \int_{\max \left\{ 2, (T/2g)^2 \right\}}^x \Phi \left( \frac{T}{\sqrt{t}} \right) \frac{1}{2\sqrt{t} \log t} \, dt
\]
as $x \rightarrow \infty$.  For numerical computations, this integral expression is more accurate when $T$ is on the scale of $\sqrt{x}$.
\end{remark}


\subsection{Convergence of the constant} \label{constantconvergencesubsection}
We consider the limit that yields the constant in the heuristic and prove it converges and is well-defined for any trace $T\neq 0$. We first show the following:
\begin{proposition}\label{vertical}
Fix $T\neq 0$ an integer. One has the following product expansion, where $m_A$ denotes the conductor of the image of the Galois representation on torsion of $A=E_1\times E_2$, where $H(m)=\im(\overline{\rho}_{A,m})$ and where $H(m,T)$ denotes the trace $T$ elements in $H(m)$ :
$$
F(T) := \lim_{m\tilde{\to} \infty}\frac{m \vert H(m,T)\vert }{\vert H(m)\vert }=\frac{m_{A,T} \vert H(m_{A,T},T)\vert }{\vert H(m_{A,T})\vert }\prod_{\ell \nmid m_A}\frac{\ell^{v_\ell(T)+1}\vert H(\ell^{v_\ell(T)+1},T)\vert }{\vert H(\ell^{v_\ell(T)+1})\vert }.
$$
\end{proposition}
We note that this shows not only that the limit splits into factors over each prime, but also that at each prime $\ell$ the factors eventually stabilize when increasing the powers $d\in\mathbb{N}$ such that $\ell^d\vert m$. This is a non-trivial feature of this setup, see for instance the failure thereof when considering products of elliptic curves where each individual trace is fixed in \cite{AkbaryParks}. The reader may consult, e.g., Theorem 1.5., Proposition 1.8. and Conjecture 1.9. of \cite{AkbaryParks} for the difficulties encountered.\\
In our setting, at primes $\ell \nmid m_A$, the stabilization is connected to the smoothness of the $\ell$-adic analytic manifold $X_{T}(\bZ_\ell)$, where $X_{T}$ denotes the subvariety of $\GL_2\times_{\det}\GL_2$ whose traces sum to $T$. 
Indeed, smoothness of $X_{T}(\bZ_\ell)$ implies that 
$$\vert X_{T}(\mathbb{Z}/\ell^n\mathbb{Z})\vert=V(X_{T}(\bZ_\ell))\cdot \ell^{6n}$$
for a constant $V(X_{T}(\bZ_\ell))$ which computes the volume of the manifold $X_{T}(\bZ_\ell)$, as proven by J-P. Serre in \cite[Th\'eor\`eme 9]{SerreCebo}. The factors at each prime are just a rational multiple of this volume.
A computation of the Jacobian matrix of the defining equations of $X_{T}$ shows the following:
\begin{Lemma}
The $\ell$-adic analytic manifold $X_{T}(\mathbb{Z}_\ell)$ is smooth if and only if the prime $\ell$ does not divide the trace $T$. 
\end{Lemma}
Nevertheless, a result of Oesterl\'e \cite[Th\'eor\`eme 2]{Oesterle} shows that the limit $\lim_{n\to\infty}\ell^{-6n}\cdot\vert X_{T}(\mathbb{Z}/\ell^n\mathbb{Z})\vert$ still converges to the volume $V(X_{T}(\bZ_\ell))$. Proposition \ref{vertical} shows that, in addition to that, at primes $\ell\nmid m_A$ the factors stabilize at exponent $n=v_\ell(T)+1$. The computation of those stable factors will then be treated in Section \ref{section:explicit constants}.\par

We turn to the proof of Proposition \ref{vertical}, proceeding as in \cite[Lemma 12]{CDSS}:
\begin{Lemma}\label{equidistrib}
If $(m, m_A)=1$ or $m_A\vert m$, then for any prime $\ell$ dividing $m$ we have
$$\vert H(m\ell^{v_\ell(T)+1},S)\vert=\vert H(m\ell^{v_\ell(T)+1},T)\vert$$
for any integer $S\equiv T \mod m\ell^{v_\ell(T)}$.
\end{Lemma}
\begin{proof}
Assume $S\equiv T \mod m\ell^{v_\ell(T)}$. Since the $\ell$-adic valuation $v_\ell(m)>0$ is strictly positive it follows that $v_\ell(T)=v_\ell(S)$. We may write $T=T_0\ell^{v_\ell(T)}$ and $S=S_0\ell^{v_\ell(T)}$ for integers $T_0,S_0$. Moreover, writing $m=m_0\ell^{v_\ell(m)}$ there exists by the Chinese remainder theorem $u\in\bZ$ such that 
\begin{equation} \label{definitionofuviaCRT}
u\equiv T_0^{-1}S_0 \mod \ell^{v_\ell(m)+1} \; \text{ and } \; u\equiv 1\mod m_0.
\end{equation}
It follows from this and from the hypothesis $S \equiv T \mod m\ell^{v_\ell(T)}$ that   
\begin{equation} \label{utakesTtoS}
    \begin{split}
    &uT\equiv S \mod \ell^{v_\ell(m) + v_\ell(T)+1}, \\
    &uT\equiv S \mod m_0.
    \end{split}
\end{equation}
Furthermore, by \eqref{definitionofuviaCRT} we have $u\equiv 1 \mod m$, and thus the pair $U := uI_2\times uI_2$ of integer matrices is congruent modulo $m$ to $I_1 \times I_2$, and it therefore lies in $H(m)$.
By definition of $m_A$, it follows that $U \mod m \ell^{v_\ell(T)+1}$ belongs to $H(m \ell^{v_\ell(T)+1})$, and therefore left multiplication by $U$ defines a bijection from $H(m \ell^{v_\ell(T)+1})$ onto itself which, by \eqref{utakesTtoS}, restricts to a bijection between the sets $H(m\ell^{v_\ell(T)+1},T)$ and $H(m\ell^{v_\ell(T)+1},S)$, as desired.
\end{proof}
\begin{Lemma}\label{stable}
If $(m, m_A)=1$ or $m_A\vert m$, then for any prime $\ell$ dividing $m$ and $k\geq 1$ we have:
$$\frac{m\ell^{v_\ell(T)+k}\vert H(m\ell^{v_\ell(T)+k},T)\vert}{\vert H(m\ell^{v_\ell(T)+k})\vert}=\frac{m\ell^{v_\ell(T)}\vert H(m\ell^{v_\ell(T)},T)\vert}{\vert H(m\ell^{v_\ell(T)})\vert}.$$
\end{Lemma}
\begin{proof}
It suffices to establish the case $k=1$ (we can then make the replacement $m \mapsto \ell m$ and repeat the argument). Considering the projection map
\[
\pi:H(m\ell^{v_\ell(T)+1})\to H(m\ell^{v_\ell(T)}),
\]
we note that, by the assumptions on $m$, we have $H(m\ell^{v_\ell(T)+1}) = \pi^{-1}\left( H(m\ell^{v_\ell(T)}) \right)$.  Thus, on the one hand we see that
 $$\vert H(m\ell^{v_\ell(T)+1})\vert=\vert \ker \pi\vert \cdot \vert H(m\ell^{v_\ell(T)})\vert.$$
On the other hand, using Lemma \ref{equidistrib} to show that the counts for all $\ell$ lifts of $T$ from $\mathbb{Z}/m\ell^{v_\ell(T)}\mbz$ to $\mathbb{Z}/m\ell^{v_\ell(T)+1}\mbz$ are the same, we obtain
 $$\vert H(m\ell^{v_\ell(T)+1},T)\vert\cdot \ell = \vert \ker \pi\vert \cdot \vert H(m\ell^{v_\ell(T)},T)\vert.$$
 The claim follows.
\end{proof}
\begin{proof}[Proof of Proposition \ref{vertical}]
We write any $m$ as a product $m=m_1\cdot m_2$ where $m_1\vert m_A^\infty$ and $(m_2,m_A)=1$. 
By the Chinese Remainder theorem, one has 
$$\frac{m \vert H(m,T)\vert }{\vert H(m)\vert }=\frac{m_1\vert H(m_1,T)\vert }{\vert H(m_1)\vert }\prod_{\ell \vert m_2}\frac{\ell^{v_\ell(m_2)}\vert H(\ell^{v_\ell(m_2)},T)\vert }{\vert H(\ell^{v_\ell(m_2)})\vert }$$
and thus taking limits
$$\lim_{m\tilde{\to} \infty}\frac{m \vert H(m,T)\vert }{\vert H(m)\vert }=\lim_{m_1\tilde{\to} \infty}\frac{m_1\vert H(m_1,T)\vert }{\vert H(m_1)\vert }\prod_{\ell \nmid m_A}\lim_{n\to\infty}\frac{\ell^{n}\vert H(\ell^{n},T)\vert }{\vert H(\ell^{n})\vert }.$$
The result now follows from Lemma \ref{stable}.
\end{proof}

It remains to show that the infinite product of rational functions in the primes $\ell\nmid m_A$ converges:
\begin{proposition}\label{infiniteproduct}
Assume $T\neq 0$. The infinite product $\displaystyle \prod_{\ell \nmid m_A}\frac{\ell^{v_\ell(T)+1}\vert H(\ell^{v_\ell(T)+1},T)\vert }{\vert H(\ell^{v_\ell(T)+1})\vert }$ converges.
\end{proposition}
This follows from a straightforward explicit computation of the quantities involved which we record here:
\begin{Lemma}
Let $\ell$ be an odd prime. Then using the quadratic residue symbol we have
\[
	\vert M_{2\times 2}(\mbz/\ell\mbz)^{\tr \equiv t}_{\det \equiv d}\vert = \ell \left(\ell+\legendre{t^2-4 d}{\ell}\right) .
\]
\end{Lemma}
\begin{proof}
We are counting the set of matrices 
\[
	\left\{\begin{pmatrix} 
	x & y \\
	z & w
	\end{pmatrix}\in M_{2\times 2}(\mbz/\ell\mbz) : xw - yz= d, x+w= t\right\}
\]
This amounts to computing $\vert\{x,y,z\in \mbz/\ell\mbz : yz = -x^2 + tx - d \}\vert$.
If $yz=0$, there are $\legendre{t^2-4d}{\ell}+1$ possibilities 
for $x$. If $yz\in \mathbb{F}_\ell^\times$ then $x$ can run through
$\ell-1-\legendre{t^2-4d}{\ell}$ possibilities. We thus obtain
\[
\begin{split}
\vert M_{2\times 2}(\mbz/\ell\mbz)^{\tr \equiv t}_{\det \equiv d}\vert 
&= \left(1+\legendre{t^2-4d}{\ell}\right)\cdot(2\ell-1) 
		+ \left(\ell-1-\legendre{t^2-4d}{\ell}\right)\cdot (\ell-1) \\
&= \ell\left(\ell+\legendre{t^2-4d}{\ell}\right).
\end{split}
\]
\end{proof}
Moreover, recalling the notation $G=\GL_2\times_{\det}\GL_2$, we record here the counts:
\begin{Lemma}
Let $\ell$ be a prime. Then $\vert G\left(\ell\right)\vert = (\ell-1) \ell^2 (\ell^2-1)^2$.
\end{Lemma}

\begin{Lemma} \label{tracezero}
Let $\ell$ be an odd prime. Then 

$$\vert G\left(\ell\right)\vert = (\ell-1) \ell^2 (\ell^2-1)^2 \; \text{ and } \; \vert G\left(\ell,0\right)\vert = \ell^2(\ell-1)(\ell^3-\ell-1).$$

\end{Lemma} 
\begin{proof}
Note that 
\begin{align*}
	\vert G\left(\ell,0\right) \vert
	&= \sum_{t\in \mbz/\ell\mbz} \sum_{d\in (\mbz/\ell\mbz)^\times} \vert M_{2\times 2}(\mbz/\ell\mbz)\vert^{\tr \equiv t}_{\det \equiv d}\vert M_{2\times 2}(\mbz/\ell\mbz)^{\tr \equiv -t}_{ \det \equiv d}\vert \\
	&= \sum_{t\in \mbz/\ell\mbz} \sum_{d\in (\mbz/\ell\mbz)^\times} \ell^2\left(\ell+\legendre{t^2-4d}{\ell}\right)^2 .
\end{align*}
The $t=0$ term is
\begin{align*}
	\sum_{d\in (\mbz/\ell\mbz)^\times} \ell^2 \left( \ell + \left(\frac{4d}{\ell} \right)\right)^2
	&= (\ell-1)/2\cdot \ell^2((\ell+1)^2+(\ell-1)^2) \\
	&= \ell^2 (\ell-1) (\ell^2+1).
\end{align*}
For fixed $t\in (\mbz/\ell\mbz)^\times$ it is easy to see that the set 
$\{t^2-4d : d\in (\mbz/\ell\mbz)^\times\}$ consists of zero, $\frac{\ell-3}{2}$ squares, and 
$\frac{\ell-1}{2}$ non-squares. Thus, 
\begin{align*}
	\sum_{t\in (\mbz/\ell\mbz)^\times}\sum_{d\in (\mbz/\ell\mbz)^\times} \ell^2 \left( \ell + \legendre{t^2-4d}{\ell}\right)^2
	&= \sum_{t\in (\mbz/\ell\mbz)^\times} \ell^2\left(\ell^2+\frac{(\ell-3)}{2}(\ell+1)^2+\frac{(\ell-1)}{2}(\ell-1)^2\right) \\
	&= \ell^2 (\ell-2) (\ell^2+\ell+1)(\ell-1) .
\end{align*}
We conclude that, as desired,
\begin{align*}
	\vert G\left(\ell,0\right)\vert 
	&= \ell^2 (\ell-1) (\ell^2+1) + \ell^2 (\ell-2) (\ell^2+\ell+1)(\ell-1) \\
	&= \ell^2 (\ell-1) (\ell^3-\ell-1).
\end{align*}
\end{proof}
\begin{Lemma} \label{good primes}
Let $\ell$ be an odd prime. Then 
\[
	\frac{\ell \vert G\left(\ell,T\right)\vert}{\vert G\left(\ell\right)\vert}
	= \begin{cases} 
		\displaystyle \frac{\ell(\ell^3-\ell-1)}{(\ell^2-1)^2} & \text{if }T=0 \\
		\displaystyle \frac{\ell(\ell^4-\ell^3-2\ell^2+\ell+2)}{(\ell^2-1)^2(\ell-1)} & \text{else} 
	\end{cases} .
\]
Moreover if $\ell\nmid T$, then $\vert G\left(\ell^n,T\right)\vert=\ell^{6n-4}\cdot (\ell^4-\ell^3-2\ell^2+\ell+2)$
and
\[
\displaystyle \frac{\ell \vert G\left(\ell,T\right)\vert}{\vert G\left(\ell\right)\vert}=\frac{\ell^n \vert G\left(\ell^n,T\right)\vert}{\vert G\left(\ell^n\right)\vert}.
\]
\end{Lemma}
\begin{proof}
For $t\in (\mbz/\ell\mbz)^\times$, the rule $h\mapsto t h$ gives a bijection 
$G\left(\ell,1\right) \to G\left(\ell, t\right)$. Thus $\vert G\left(\ell,t\right)\vert = \vert G\left(\ell,1\right)\vert$ for all $t\in (\mbz/\ell\mbz)^\times$. 
Using Lemma \ref{tracezero}, it follows that for $t\in (\mbz/\ell\mbz)^\times$ we have 
\[
	\vert G\left(\ell,t\right)\vert = \frac{\vert G\left(\ell\right)\vert - \vert G\left(\ell,0\right)\vert}{\ell-1} 
	= \ell^2 (\ell^4-\ell^3-2 \ell^2+\ell+2) .
\]
The first result follows easily and the second follows from Lemma \ref{stable}.  
\end{proof}
We deduce convergence of the product from these counts:
\begin{proof}[Proof of Proposition \ref{infiniteproduct}]

It suffices to consider the primes $\ell\nmid (T\cdot m_A)$. For such primes, we have that $H(\ell^n)=G(\ell^n)$ and from our previous results since $T\neq 0$ we have:
\begin{align*}\frac{\ell^{v_\ell(T)+1} \vert H(\ell^{v_\ell(T)+1},T)\vert}{\vert H(\ell^{v_\ell(T)+1})\vert}=	\frac{\ell \vert G\left(\ell,T\right)\vert}{\vert G\left(\ell\right)\vert}&=\frac{\ell(\ell^4-\ell^3-2\ell^2+\ell+2)}{(\ell^2-1)^2(\ell-1)}\\
&=1+O\left( \frac{1}{\ell^3} \right).
\end{align*}
This proves convergence.
\end{proof}

In Sections \ref{section:explicit constants} and \ref{section:bad primes} below, we will prove Theorem \ref{thm:mainformulas}, establishing explicit formulas for the universal factors and the exceptional factor (in the case where $(E_1,E_2)$ form a Serre pair), respectively.

\subsection{Positivity of the conjectural constant} \label{missingtracessection}

We now discuss briefly the question of when the conjectural constant $C(E_1 \times E_2,T)$ appearing in Conjecture \ref{conj:main} is positive.  As we have seen in the previous section, the infinite product part of the constant is a convergent Euler product. In particular that product is non-zero, and so we have
\begin{equation} \label{equivalentconditionforCATvanishing}
C(A,T) = 0 \; \Longleftrightarrow \; H_A(m_{A,T},T) = \emptyset \quad\quad \left( A = E_1 \times E_2 \right).
\end{equation}
This is analogous to the situation in the original Lang-Trotter conjecture, wherein, for a single elliptic curve $E$ over $\bQ$ and an integer $t$, we have
\begin{equation*} 
C(E,t) = 0 \; \Longleftrightarrow \; H_E(m_E,T) = \emptyset.
\end{equation*}
For each $m \in \mathbb{N}$, we are motivated to define the following subsets of $\mbz/m\mbz$:
\begin{equation} \label{defofmcT1andmcT2}
\begin{split}
\mathcal{T}_{E_i}(m) :=& \left\{ t \in \mbz/m\mbz : \; \left\vert \{ p \text{ prime } : p \nmid N_{E_i}, \; a_p(E_i) \equiv t \mod m \} \right\vert = \infty \right\} \\
=& \tr \left( H_{E_i}(m) \right) \subseteq \bZ/m\bZ \quad\quad\quad \left( i \in \{ 1, 2 \} \right), \\
\mathcal{T}_{E_1 \times E_2}(m) :=& \left\{ T \in \mbz/m\mbz : \; \left\vert \{ p \text{ prime } : p \nmid N_{E_1}N_{E_2}, \; a_p(E_1 \times E_2) \equiv T \mod m \} \right\vert = \infty \right\} \\
=& \tr \left( H_{E_1 \times E_2}(m) \right)  \subseteq \bZ/m\bZ.
\end{split}
\end{equation}
Evidently, we have
\begin{equation} \label{mcTcontainment}
    \mathcal{T}_{E_1 \times E_2}(m) \subseteq \mathcal{T}_{E_1}(m) + \mathcal{T}_{E_2}(m) \subseteq \mbz/m\mbz
\end{equation}
and, denoting by $\overline{t}_m$ the reduction modulo $m$ of an integer $t$,
\begin{equation} \label{interpofconstantvanishingintermsofTs}
\begin{split}
    C(E,t) = 0 \; &\Longleftrightarrow \; \exists m \mid m_E \text{ with } \overline{t}_m  \notin \mathcal{T}_{E}(m), \\
     C(E_1\times E_2,T) = 0 \; &\Longleftrightarrow \; \exists m \mid m_{E_1 \times E_2} \text{ with } \overline{T}_m \notin \mathcal{T}_{E_1 \times E_2}(m).
\end{split}
\end{equation}

Suppose that $A = E_1 \times E_2$ where $(E_1,E_2)$ is a Serre pair, let $m$ be any positive integer divisible by $m_A$, and write $m =: 2^\ga m'$ with $m'$ odd.  By the subgroup containment
\[
\left( \pi_{2^\ga,2}^{-1}\left( \left\langle \begin{pmatrix} 1 & 1 \\ 1 & 0 \end{pmatrix} \right\rangle \right) \times \SL_2(\mbz/m'\mbz) \right) \times \left( \pi_{2^\ga,2}^{-1}\left( \left\langle \begin{pmatrix} 1 & 1 \\ 1 & 0 \end{pmatrix} \right\rangle \right) \times \SL_2(\mbz/m'\mbz) \right) \subseteq H_A(m),
\]
together with \eqref{equivalentconditionforCATvanishing}, we may see that the constant $C(E_1\times E_2,T)$ is positive for all $T \in \bZ$.  By \cite{MR3071819}, the condition that $(E_1,E_2)$ is a Serre pair represents the generic case, happening for almost all pairs $(E_1,E_2)$.  

To find examples of products and values of $T$ for which $C(E_1\times E_2,T) = 0$, the condition \eqref{interpofconstantvanishingintermsofTs}, together with \eqref{mcTcontainment}, leads us to the observation that
\[
\mathcal{T}_{E_1}(m) + \mathcal{T}_{E_2}(m) \neq \bZ/m\bZ \; \Longrightarrow \; \exists T \in \mbz \text{ with } C(E_1 \times E_2,T) = 0.
\]
In particular, for any $T \in \mbz$ for which $\overline{T}_m \notin \mathcal{T}_{E_1}(m) + \mathcal{T}_{E_2}(m)$, we have $C(E_1 \times E_2,T) = 0$.  A basic example of this occurs when
\begin{equation} \label{2torsionhypothesis}
E_1(\bQ)[2] \neq \{ \mc{O}_{E_1} \} \quad \text{ and } \quad E_2(\bQ)[2] \neq \{ \mc{O}_{E_2} \}.
\end{equation}
Note in general that, if $E(\bQ)$ has a point of order $m$ then for all primes $p \nmid N_{E}$, we have 
\[
m \; \text{ divides } \; p+1-a_p(E)=\vert E(\bF_p)\vert.
\]
Taking $m = 2$, we deduce that, provided \eqref{2torsionhypothesis} holds, for all odd primes $p \nmid N_{E_1}N_{E_2}$, each $a_p(E_i)$ must be even, and therefore $a_p(E_1 \times E_2) = a_p(E_1) + a_p(E_2)$ must be even as well; in this case we have
\[
\mathcal{T}_{E_1}(2) = \{ 0 \} = \mathcal{T}_{E_2}(2), \quad \text{ so } \quad \mathcal{T}_{E_1}(2) + \mathcal{T}_{E_2}(2) = \{ 0\}.
\]
A quick search produces many explicit examples, for instance taking 
\[
\begin{split}
&E_1 : \; y^2=x^3+x^2+x, \\
&E_2 : \;  y^2+xy=x^3+x^2+x
\end{split}
\]
(\cite[\href{http://www.lmfdb.org/EllipticCurve/Q/48.a5}{E.C. 48.a5}]{lmfdb} and \cite[\href{http://www.lmfdb.org/EllipticCurve/Q/39.a4}{E.C. 39.a4}]{lmfdb}, respectively) leads to an abelian surface $A=E_1\times E_2$ satisfying the conditions of Conjecture \ref{conj:main} where at most finitely many primes $p$ have odd $a_p(A)$. 

Our next example shows that it is possible to find pairs $(E_1,E_2)$ with $C(E_1 \times E_2,T) = 0$ for an appropriate $T$ even when $\mathcal{T}_{E_1}(m) + \mathcal{T}_{E_2}(m) = \mbz/m\mbz$ for all $m$.  Given \eqref{interpofconstantvanishingintermsofTs} and \eqref{mcTcontainment}, this is only possible if, for some $m \in \mathbb{N}$, we have
\begin{equation} \label{propercontainmentfromentanglement}
\mc{T}_{E_1\times E_2}(m) \subsetneq \mc{T}_{E_1}(m) + \mc{T}_{E_2}(m) = \mbz/m\mbz.
\end{equation}
We define $E_1$ and $E_2$ by 
\[
\begin{split}
    &E_1 : \; y^2 = x^3 + 6x - 2, \\
    &E_2 : \; y^2 = x^3 - 108x - 918.
\end{split}
\]
As demonstrated in \cite{LangTrotter}, the elliptic curve $E_1$ is a Serre curve, and again using that
\[
\pi_{2^\ga,2}^{-1}\left\langle \begin{pmatrix} 1 & 1 \\ 1 & 0 \end{pmatrix} \right\rangle \times \SL_2(\mbz/m'\mbz) \subseteq \im(\rho_{E,2^\ga m'}),
\]
we may see that, for every $m \in \mathbb{N}$, we have $\mathcal{T}_{E_1}(m) = \mbz/m\mbz$.  In particular, this implies that $\mathcal{T}_{E_1}(m) + \mathcal{T}_{E_2}(m) = \mbz/m\mbz$. On the other hand, \cite[Theorem 1]{RubinSilverberg} implies that $E_1[2]$ and $E_2[2]$ are isomorphic as $\Gal(\overline{\bQ}/\bQ)$-modules, which in turn implies that, for every prime $p \nmid N_{E_1}N_{E_2}$, we have $a_p(E_1) \equiv a_p(E_2) \mod 2$.  From this, it follows that $a_p(E_1 \times E_2)$ must always be even, so
\[
\mc{T}_{E_1 \times E_2}(2) = \{ 0 \} \subsetneq \mbz/2\mbz = \mc{T}_{E_1}(2) + \mc{T}_{E_2}(2),
\]
and thus $C(E_1 \times E_2,T) = 0$ for each odd $T \in \bZ$.  

Another example in this same spirit involves level $4$ structure.  Consider the elliptic curves
\[
\begin{split}
    &E_1 : \; y^2 = x^3 - 4900172/45369x - 19600688/45369, \\
    &E_2 : \; y^2 = x^3 - 186732x - 746928.
\end{split}
\]
The elliptic curve $E_1$ admits a rational $4$-isogeny (equivalently, $\displaystyle \im(\rho_{E_1,4}) \subseteq \left\{ \begin{pmatrix} * & * \\ 0 & * \end{pmatrix} \right\}$), whereas $E_2$ has a square discriminant (equivalently, $\displaystyle \im(\rho_{E_2,2}) \subseteq \left\langle \begin{pmatrix} 1 & 1 \\ 1 & 0 \end{pmatrix} \right\rangle$).  Taking $m = 4$ in \eqref{defofmcT1andmcT2}, we have
\[
\begin{split}
\mathcal{T}_{E_1}(4) &= \{ 0, 2 \}, \\
\mathcal{T}_{E_2}(4) &= \{ 0, 1, 2, 3 \}.
\end{split}
\]
However, because of the cyclotomic entanglement $\bQ(i) \subseteq \bQ(E_1[4]) \cap \bQ(E_2[4])$, and the fact that 
\[
\forall g \in \GL_2(\bZ/4\bZ), \; g \equiv I \mod 2 \; \Longrightarrow \tr g \equiv 1 + \det g \mod 4,
\]
we may see that $2 \notin \mc{T}_{E_1 \times E_2}(4)$.  In fact,
\[
\mc{T}_{E_1 \times E_2}(4) = \{0, 1, 3 \},
\]
and thus $C(E_1 \times E_2,T) = 0$ whenever $T \equiv 2 \mod 4$.  

Finally, one might ask whether there exists pairs  $(E_1,E_2)$ such that for some $m \in \bZ$ we have
\[
\mathcal{T}_{E_1}(m) \neq \bZ/m\bZ \; \text{ and } \; \mathcal{T}_{E_2}(m) \neq \bZ/m\bZ,\;  \text{ but } \; \mathcal{T}_{E_1 \times E_2}(m) = \mbz/m\mbz.
\]
Our final example shows that this can indeed happen.  Let $E_1$ and $E_2$ be given by 
\[
\begin{split}
    &E_1 : \; y^2 = x^3 - 41399424/41971x - 496793088/41971, \\
    &E_2 : \; y^2 = x^3 + 533418040116/625x - 3633643689270192/3125.
\end{split}
\]
The elliptic curve $E_1$ has a rational $2$-torsion point, and so, taking $m = 4$, we have $\mathcal{T}_{E_1}(4) = \{ 0, 2 \}$.  The second elliptic curve $E_2$ has the property that $\bQ(\sqrt{\Delta_{E_2}}) = \bQ(i)$, which forces the sub-extension $\bQ(i,\Delta_{E_2}^{1/4})$, which generically has Galois group $D_4$ over $\bQ$, to be at most a bi-quadratic extension of $\bQ$. For $E_2$, we further have $\bQ(i,\Delta_{E_2}^{1/4}) = \bQ(i)$, and this in turn forces $\rho_{E_2,4}(G_\bQ)$ to be contained in a certain index $4$ subgroup of $\GL_2(\bZ/4\bZ)$. The traces of this subgroup are $\{0,2,3\}$, and we compute that, in fact, $\mathcal{T}_{E_2}(4) = \{0,2,3\}$.  However, unlike the previous example, the cyclotomic entanglement does not manage to link up the traces in this case, and we may see by direct computation that 
\[
\{ a_p(E_1 \times E_2) \mod 4 : p \nmid N_{E_1}N_{E_2} \} = \bZ/4\bZ.
\]

All of the above examples suggest that the question of positivity of the conjectural constants $C(E_1 \times E_2,T)$ is rather subtle and worthy of further study.

\subsection{Theoretical evidence}\label{subsection:theoretic}
Although this is not the main focus of our paper, we conclude this section with some remarks about theoretical results towards Conjecture \ref{conj:main}. Applying the effective \v{C}ebotarev results of J. Lagarias and A. Odlyzko \cite{Lagarias1977EffectiveVO}, we obtain, following Serre's original treatment for single elliptic curves \cite[Th\'eor\`eme 20]{SerreCebo}: 
\begin{theorem} \label{upperboundtheorem}
Let $E_{1/\bQ}$, $E_{2/\bQ}$ be non-CM elliptic curves non-isogenous over $\overline{\bQ}$. Let $A\sim_{\bQ}E_1\times E_2$ and let $T\in \mathbb{Z}\setminus \{0\}$. Set $\epsilon(x)=(\log x)(\log\log x)^{-2}(\log\log\log x)^{-1}$ and $\epsilon_R(x)=\sqrt{x}(\log x)^{-2}$.
\begin{enumerate}
    \item We have an unconditional upper bound as $x\to \infty$ of 
    $$\pi_{A,T}(x)=O(\Li(x)\cdot\epsilon(x)^{-1/7}).$$
    \item Under GRH, we obtain the upper bound as $x\to \infty$: 
    $$\pi_{A,T}(x)=O(\Li(x)\cdot\epsilon_R(x)^{-1/7}).$$
    In particular, for all $\varepsilon>0$ we have under GRH the bound $\pi_{A,T}(x)=O(x^{1-\frac{1}{14}+\varepsilon})$ as $x\to \infty$. 
\end{enumerate}
\end{theorem}
\begin{proof}
It suffices to consider the case $A=E_1\times E_2$. Let $\ell$ denote a large rational prime such that the image of the $\ell$-adic Galois representation $\im(\rho_{A,\ell})$ is all of the compact $\ell$-adic Lie group $G(\bZ_\ell)=\GL_2(\bZ_\ell)\times_{\det}\GL_2(\bZ_\ell)$ and such that $\ell\nmid T$. The latter condition implies smoothness of the $\ell$-adic analytic manifold $X_{T}(\bZ_\ell):=\{(g_1,g_2)\in \GL_2(\mathbb{Z}_\ell)\times_{\det}\GL_2(\mathbb{Z}_\ell)\vert \tr(g_1)+\tr(g_2)=T\}$. Moreover,  $X_{T}(\bZ_\ell)\subset G(\bZ_\ell)$ is closed, stable under conjugation, and the dimensions as $\ell$-adic manifolds (in the sense of Serre-Oesterl\'e) of $X_{T}(\bZ_\ell)$ and $G(\bZ_\ell)$ are $6$ and $7$, respectively. Observing that 
$$\pi_{A,T}(x)\leq \vert\{p\leq x \text{ prime }: p\nmid \ell\cdot N_A\text{ and } \rho_{A,\ell}(\Frob_p)\in X_{T}(\bZ_\ell)\}\vert,$$
it suffices to give an upper bound for the latter. This bound now follows from \cite[Th\'eor\`eme 10]{SerreCebo}, which is an application of effective \v{C}ebotarev results to Galois extensions $K_\ell/\bQ$ with Galois groups $G_\ell = \Gal(K_\ell/\bQ)$ which are compact $\ell$-adic Lie groups. Indeed, in our setting taking $K_\ell$ to be the fixed field of $\ker(\rho_{A,\ell})$ gives a Galois extension $K_\ell = \bQ(A_\ell)$ of $\mathbb{Q}$ with Galois group $\Gal(\bQ(A_\ell)/\bQ) \simeq G(\bZ_\ell)$, and we are precisely counting the primes $p$ outside a finite set of ultrametric places such that $\rho_{A,\ell}(\Frob_p)\in X_{T}(\bZ_\ell)$. 
\end{proof}
Just as for single elliptic curves, these results can likely be refined and the upper bound improved using similar techniques, and we plan to address this in ongoing work \cite{future_upperbounds}. \par
Furthermore, we expect that Conjecture \ref{conj:main} can be shown to hold on average, as was proven for single elliptic curves by C. David and F. Pappalardi \cite{DavidPapp}, extending results by E. Fouvry and R. Murty in the supersingular case \cite{fouvry_murty_1996}. More precisely, writing $E(a,b)$ with $a,b\in \bZ$ for the elliptic curve with Weierstrass equation $y^2=x^3+ax+b$, David and Pappalardi show that:
\begin{theorem}[Corollary 1.3. of \cite{DavidPapp}]
Fix $t\in\bZ$. For $\varepsilon>0$ and $A,B>x^{1+\varepsilon}$ we have as $x\to \infty$:
$$\frac{1}{4AB}\sum_{{\begin{substack} {a\leq A\\ b\leq B} \end{substack}}}\pi_{E(a,b),t}(x)\sim C_t\cdot \sqrt{x}/\log(x),$$
where the constant $C_t$ is the universal part of the constant $C(E,t)$ given by:
$$C_t=\frac{2}{\pi}\prod_{\ell\mid t}\left(1-\frac{1}{\ell^2}\right)^{-1}\prod_{\ell\nmid t}\frac{\ell(\ell^2-\ell-1)}{(\ell-1)(\ell^2-1)}.$$
\end{theorem}
They also prove that the normal order of $\pi_{E(a,b),t}(x)$ is $C_t\cdot \sqrt{x}/\log(x)$ (see \cite[Corollary 1.5.]{DavidPapp}). Moreover, it is known for Serre curves $E(a,b)$ and conditionally in general (see \cite{MR2534114}) that $$\frac{1}{4AB}\sum_{a\leq A,b\leq B}C(E(a,b),t)\sim C_t$$
as $A, B \rightarrow \infty$. \par 
We anticipate that in our setting for large enough $A,B$ in relation to $x$ one has similarly the asymptotic 
$$\frac{1}{16A^2B^2}\sum_{{\begin{substack} {a_1,a_2\leq A\\ b_1,b_2\leq B} \end{substack}}}\pi_{E(a_1,b_1)\times E(a_2,b_2),T}(x)\sim C_T\cdot (\sqrt{x}/\log(x)),$$
where 
$$
C_T:=2\Phi_{\SU(2)^{\times 2}}(0)\cdot \prod_{\ell\text{ prime }} \frac{\ell^{v_\ell(T)+1} \vert G(\ell^{v_\ell(T)+1},T)\vert}{\vert G(\ell^{v_\ell(T)+1})\vert}.
$$
Moreover, the normal order should also be the average order and it seems reasonable to expect that
$C_T=\lim_{A,B\to\infty}\left(\frac{1}{16A^2B^2}\sum_{{\begin{substack} {a_1,a_2\leq A\\ b_1,b_2\leq B} \end{substack}}}C(E(a_1,b_1)\times E(a_2,b_2),T)\right)$ is the average constant.

\section{Explicit constants}\label{section:explicit constants}
Let $\ell\geq 3$ be a fixed prime. The goal of this section is to compute the stable ratios given by 
$$
\frac{\ell^{v_\ell(T)+1}\vert H(\ell^{v_\ell(T)+1},T)\vert }{\vert H(\ell^{v_\ell(T)+1})\vert }
=
\frac{\ell^{v_\ell(T)+1}\vert G(\ell^{v_\ell(T)+1},T)\vert }{\vert G(\ell^{v_\ell(T)+1})\vert }
$$ 
for primes $\ell\vert T$ and coprime to $m_A$, where $T$ is some non-zero integer. As soon as $\ell\vert T$, one has to take into consideration that the number of matrices of a given trace and determinant above a class modulo $\ell$ scales differently depending on the $\ell$-adic distance to scalar matrices. Therefore a more detailed analysis is necessary. 
\subsection{Fixed trace and determinant counts}
We first consider counts for $M_{2\times 2}(\mathbb{Z}/\ell^{n}\mathbb{Z})$. Denote by $S_0$ the matrices in $M_{2\times 2}(\mathbb{Z})$ which are not congruent to scalar matrices modulo $\ell$ and by $S_0(j)$ the reduction of $S_0$ modulo $\ell^j$:
\[
\begin{split}
S_0 :=& \left\{ \begin{pmatrix} a & b \\ c & d \end{pmatrix} \in M_{2\times 2}(\mbz) : \; \forall \lambda \in \mbz/\ell\mbz, \, \begin{pmatrix} a & b \\ c & d \end{pmatrix} \not\equiv \begin{pmatrix} \lambda & 0 \\ 0 & \lambda \end{pmatrix} \mod \ell \right\}, \\
S_0(j) =& \left\{ \begin{pmatrix} a & b \\ c & d \end{pmatrix} \in M_{2\times 2}(\mbz/\ell^j\mbz) : \; \forall \lambda \in \mbz/\ell\mbz, \, \begin{pmatrix} a & b \\ c & d \end{pmatrix} \not\equiv \begin{pmatrix} \lambda & 0 \\ 0 & \lambda \end{pmatrix} \mod \ell \right\}.  
\end{split}
\]
For fixed integers $t, s \in \mbz$, we let $\mathcal{H}_{i,n}(t,s)$ denote the set of matrices in $M_{2\times 2}(\mathbb{Z}/\ell^{n}\mathbb{Z})$ with trace (resp. determinant) congruent to t (resp. to $s$) modulo $\ell^n$ and which are congruent to a scalar matrix modulo $\ell^i$ but not $\ell^{i+1}$ and denote by $h_{i,n}(s,t)$ the cardinality of that set: 
\[
\begin{split}
\mathcal{H}_{i,n}(t,s) &:= \left\{ A \in M_{2 \times 2}(\mbz/\ell^n\mbz) : \; \begin{matrix} \exists \lambda \in \mbz, \\ \exists A' \in S_0 \end{matrix} \quad \text{ with } \quad \begin{matrix} A \equiv \lambda I_2 + \ell^i A' \mod \ell^n, \\ \tr A \equiv t \mod \ell^n \\ \det A \equiv s \mod \ell^n \end{matrix} \right\}, \\
h_{i,n}(t,s) &:= \left| \mathcal{H}_{i,n}(t,s) \right|.
\end{split}
\]
\begin{Lemma}\label{singlematrix}
Fix $(t,s)\in\mathbb{Z}$. We have that:
$$
h_{i,n}(t,s)= \begin{cases}\ell^{2n-i}\frac{\left(\ell^2+\ell\left(\frac{\ell^{-2i}\Delta(t,s)}{\ell}\right)+\left(\frac{\ell^{-2i}\Delta(t,s)}{\ell}\right)^2-1 \right)}{\ell^2}&\text{ if }0\leq i<n/2 \text{ and }\ell^{2i}\vert \Delta(t,s)\\
\ell^{2n-(3i-n)}\frac{(\ell^3-1)}{\ell^3}&\text{ if }n/2\leq i<n \text{ and }\ell^{n}\vert \Delta(t,s)\\
1&\text{ if }i=n\text{ and }\ell^{n}\vert \Delta(t,s)\\
0&\text{ else,}
\end{cases}
$$
where $\Delta(t,s)=t^2-4s$.
\end{Lemma}
\begin{proof}
We are counting matrices of  trace and determinant $(t,s)$ of the form 
$$
aI+\ell^i\begin{pmatrix}b&c\\
d&e\\
\end{pmatrix}
\mod \ell^n
$$
where $\begin{pmatrix}b&c\\
d&e\\
\end{pmatrix}\in S_0$ and $a\in \mathbb{Z}/\ell^i\mathbb{Z}$. 
We first consider the case when $i=0$ and compute $h_{0,n}(t,s)$. We know from \cite[p.127]{LangTrotter} that the count modulo $\ell$ is given by
$$h_{0,1}(t,s)=\ell^2+\ell\left(\frac{\Delta(t,s)}{\ell}\right)+\left(\frac{\Delta(t,s)}{\ell}\right)^2-1.$$ 
Moreover, once an $\ell$-adic distance to scalar matrices is fixed, the count scales uniformly and only depends on the classes of $(t,s)$ in $\mathbb{F}_\ell$.
This yields:
$$h_{0,n}(t,s)=\ell^{2n-2}\left(\ell^2+\ell\left(\frac{\Delta(t,s)}{\ell}\right)+\left(\frac{\Delta(t,s)}{\ell}\right)^2-1\right).$$ 
Let now $i>0$, the condition on the determinant amounts to the congruences
\begin{align*}\ell^{2i}(be-dc)&\equiv -a^2-\ell^ia(b+e)+s \mod \ell^n\\
&\equiv a^2-at+s\mod \ell^n\\
&\equiv 1/4((t-2a)^2-\Delta(t,s)) \mod \ell^n.
\end{align*}
We observe that for our choice of $a$ we must have that $t-2a\equiv 0 \mod \ell^i$ and therefore $a^2-at+s$ is divisible by $\ell^{\min(2i,n)}$ if and only if the discriminant $\Delta(t,s)$ is divisible by $\ell^{\min(2i,n)}$ (here $4$ is a unit). This proves that $h_{i,n}(t,s)=0$ whenever $\Delta(t,s)$ is not divisible by $\ell^{\min(2i,n)}$.
The condition on the trace amounts to the congruence
$$\ell^i(b+e)\equiv t-2a \mod \ell^n$$
which yields exactly one possibility for the sum $b+e$ modulo $\ell^{n-i}$.\\
If $i<n/2$ we must have $\ell^{2i}\vert \Delta(t,s)$ and in that case $be-dc$ has to project onto a fixed residue class modulo $\ell^{n-2i}$, yielding $\ell^i$ possibilities modulo $\ell^{n-1}$. Moreover, $\Delta(b+e, be-dc)\equiv \ell^{-2i}\Delta(t,s) \mod \ell$. Since the computation of $h_{0,n-i}(t,s)$ only depends on the discriminant modulo $\ell$, the count is $\ell^i$ times the number of matrices in $S_0(n-i)$ of a fixed discriminant $\ell^{-2i}\Delta(t,s)$:
$$h_{i,n}(t,s)=\ell^i\cdot \ell^{2(n-1)-2}\left(\ell^2+\ell\left(\frac{\ell^{-2i}\Delta(t,s)}{\ell}\right)+\left(\frac{\ell^{-2i}\Delta(t,s)}{\ell}\right)^2-1\right),$$
as claimed. \par

Finally assume $n/2\leq i<n$, the case of $i=n$ being trivial. We must then have $a^2-at+s\equiv 0 \mod \ell^n$ or equivalently $\ell^n\vert \Delta(t,s)$ since $\ell^i\vert (t-2a)$. We are reduced to counting matrices in $S_0(n-i)$ of a fixed trace $t_0=(t-2a)/\ell^i$, where now the determinant $be-dc$ is allowed to run through all of $\mathbb{Z}/\ell^{n-i}\mathbb{Z}$. We obtain:
\begin{align*}h_{i,n}(t,s)&=\ell^{2(n-i)-2}\sum_{s\in \mathbb{Z}/\ell^{n-i}\mathbb{Z}}\left( \ell^2+\ell\left(\frac{\Delta(t_0,s)}{\ell}\right)+\left(\frac{\Delta(t_0,s)}{\ell}\right)^2-1 \right)\\
&=\ell^{2(n-i)-2}(\ell^{n-i+2}+0-\ell^{n-i-1})\\
&=\ell^{3n-3i}\frac{\ell^3-1}{\ell^3}.
\end{align*}

\end{proof}

\subsection{Computation of $\vert H(\ell^n,T)\vert $}
We now fix a nonzero trace $T\in \mathbb{Z}$ and assume that $v_\ell(T)>0$. Let $S_{i,i',n}(T)$ denote the number of $\mathbb{Z}/\ell^{n}\mathbb{Z}$-points on $\GL_2\times_{\det}\GL_2$ consisting of pairs of two-by-two matrices whose traces sum to $T$ and which are congruent to scalars exactly modulo $\ell^i$ and $\ell^{i'}$, respectively:
\[
\begin{split}
S_{i,i',n}(T) :=& \left| \bigsqcup_{ t \in \mbz/\ell^n\mbz} \bigsqcup_{ s \in (\mbz/\ell^n\mbz)^\times}  \mc{H}_{i,n}(t,s) \times \mc{H}_{i',n}(T-t,s) \right| \\
=& \sum_{t \in \mbz/\ell^n\mbz} \sum_{ s \in (\mbz/\ell^n\mbz)^\times} h_{i,n}(t,s) h_{i',n}(T-t,s).
\end{split}
\]
We may as well assume that $i'\geq i$. To simplify notation, we define the function $h_i : \ell^{2i} \mbz \rightarrow \mbz$ by 
$$
h_i(x):=\ell^2+\ell\left(\frac{\ell^{-2i}x}{\ell}\right)+\left(\frac{\ell^{-2i}x}{\ell}\right)^2-1 \quad\quad\quad \left( \ell^{-2i}x \in \mbz \right).
$$
\begin{Lemma}\label{Hcounts}
Assume $i'>0$ and $v_\ell(T)>0$. We have that 
$$
S_{i,i',n}(T)=\begin{cases}\sum\limits_{t\in (\mathbb{Z}/\ell^{n})^\times}\sum\limits_{s\equiv (t/2)^2(\ell^{2i'})}\ell^{4n-(i+i')-4}h_{i'}(\Delta(t,s))h_i(\Delta(T-t,s))& \text{ if }n>2i'\text{ and } \ell^{2i}\vert T\\
\left(\ell-1\right)\ell^{6n-(i+3i')-6}\left(\ell^3-1\right)\left(\ell^2+\left(\frac{\ell^{-2i}T}{\ell}\right)^2-1\right) & \text{ if }2n>2i'\geq n>2i\text{ and } \ell^{2i}\vert T \\
\left(\ell-1\right)\ell^{3n-i-3}\left(\ell^2+\left(\frac{\ell^{-2i}T}{\ell}\right)^2-1\right) & \text{ if }i'=n>2i\text{ and }\ell^{2i}\vert T\\
(\ell-1)\ell^{7n-(3i+3i')-7}(\ell^3-1)^2 & \text{ if }2n>2i'\geq 2i\geq n \text{ and } \ell^{n}\vert T\\
(\ell-1)\ell^{4n-3i-3} (\ell^3-1)&\text{ if }2n=2i'\geq 2i\geq n \text{ and } \ell^{n}\vert T\\
0&\text{ if }\ell^{\min(2i,n)}\nmid T\\
\end{cases}
$$ 

\end{Lemma}
\begin{proof}
This follows from summing the contributions $h_{i,n}(t,s)\cdot h_{i',n}(T-t,s)$ in Lemma \ref{singlematrix}. Since $i'>0$ we get from Lemma \ref{singlematrix} positive contributions only when $v_\ell((T-t)^2-4s)>0$. But here $s$ is only running through units modulo $l\neq 2$ so that $v_\ell(4s)=0$, and thus to get positive contributions $T-t$ must be a unit, and similarly $v_\ell(t)=0$ since $v_\ell(T)>0$. Furthermore, writing $\Delta(T-t,s)-\Delta(t,s)=T(T-2t)$, by Lemma \ref{singlematrix} we only get positive contributions when $\ell^{\min(2i,n)}\vert T$. Moreover, 
$$\ell^{\min(2i',n)}\vert \Delta(t,s) \Leftrightarrow s\equiv (t/2)^2 \mod \ell^{\min(2i',n)}$$
and this also guarantees that $\ell^{2i}\vert \Delta(T-t,s)$ provided $\ell^{\min(2i,n)}\vert T$. If $n>2i'$ the conclusion follows from Lemma \ref{singlematrix}. When $2n>2i'\geq n>2i$ we obtain from Lemma \ref{singlematrix} the count: 
\begin{align*}
S_{i,i',n}(T)&=\sum\limits_{t\in (\mathbb{Z}/\ell^{n})^\times}\sum\limits_{s\equiv (t/2)^2(\ell^{n})}\ell^{3n-3i'-3}(\ell^3-1)\ell^{2n-i-2}h_i(\Delta(T-t,s))\\
&= \sum\limits_{t\in (\mathbb{Z}/\ell^{n})^\times}\ell^{5n-(i+3i')-5}(\ell^3-1)h_i(\Delta(T-t, t^2/4))\\
&=\ell^{5n-(i+3i')-5}(\ell^3-1)\sum\limits_{t\in (\mathbb{Z}/\ell^{n})^\times} h_i(T(T-2t)).
\end{align*}
Similarly, when $2n=2i'>2i$ we get $S_{i,i',n}(T)=\ell^{2n-i-2}\sum\limits_{t\in (\mathbb{Z}/\ell^{n})^\times} h_i(T(T-2t))$.
The result follows from:
$$\sum_{u\in (\mathbb{Z}/\ell^{n}\mathbb{Z})^\times} h_i(Tu)=\begin{cases} (\ell-1)\ell^{n-1}(\ell^2-1)&\text{ if } \ell^{2i+1}\vert T\\
(\ell-1)\ell^{n-1}\ell^2&\text{ if }\ell^{2i}\parallel T\\
\end{cases}$$
The first case is trivial since then all the residue symbols vanish. When $\ell^{2i}\parallel T$, we obtain that
\begin{align*}
\sum_{u\in (\mathbb{Z}/\ell^{n})^\times} h_i(Tu)&=\sum_{x\in (\mathbb{Z}/\ell^{n})^\times} \left( \ell^2+\ell\left(\frac{x}{\ell}\right)+\left(\frac{x}{\ell}\right)^2-1 \right)\\
&=\ell^{n-1}(\ell-1)\ell^2,
\end{align*}
as claimed. All of the cases now follow by an easy calculation based on Lemma \ref{singlematrix}.
\end{proof}

We now provide an exact count also in the first case when $n>2i'$. We first recall a useful standard result for the readers convenience:
\begin{Lemma}\label{lemma:legendresum}
For integers $a,b,c\in\bZ$ with $(a,b)\neq (0,0)$ and for any odd prime $\ell$ we have the formula
$$\sum_{x=0}^{\ell-1}\left(\frac{ax^2+bx+c}{\ell}\right)=
\begin{cases}-\left(\frac{a}{\ell}\right)&\text{ if }\ell\nmid b^2-4ac\\
(\ell-1)\cdot \left(\frac{a}{\ell}\right)&\text{ else.}
\end{cases}
$$
\end{Lemma}
\begin{proof}
This is a straightforward computation: when $l\mid a$ the result easily holds, and when $l\nmid a$ the sum can be rewritten as 
$$\left(\frac{a}{\ell}\right)\cdot\sum_{x=0}^{\ell-1}\left(\frac{(2ax+b)^2+(4ac-b^2)}{\ell}\right).$$
One immediately concludes when $l\mid b^2-4ac$. When $l\nmid b^2-4ac$, one concludes since $\sum_{x=0}^{\ell-1}\left(\frac{x^2+k}{\ell}\right)=-1$ when $l\nmid k$.
\end{proof}

Subsequently we shall assume that $n>v_\ell(T)$. Therefore we only get contributions when $\ell^{2i}\mid T$ and hence only when $n>2i$, so we only list these cases. 

\begin{proposition}\label{exactcount}
Assume $i'>0$ and $v_\ell(T)>0$. We have that 
$$
S_{i,i',n}(T)=
\begin{cases}(\ell-1)\ell^{6n-(i+3i')-6}(\ell^3-1)\left(\ell^2+\left(\frac{\ell^{-2i}T}{\ell}\right)^2-1 \right)& \text{ if }n>i'>i,n>2i\text{ and } \ell^{2i}\vert T \\
(\ell-1)\ell^{6n-(i+3i')-6}(\ell^3-3)\ell^2& \text{ if }i'=i,n>2i\text{ and } \ell^{2i}\parallel T\\
(\ell-1)\ell^{6n-(i+3i')-6}(\ell-1)(\ell^4+\ell^3+2\ell^2-\ell-1)& \text{ if }i'=i,n>2i\text{ and } \ell^{2i+1}\vert T\\
(\ell-1)\ell^{6n-(i+3i')-6}\ell^3\left( \ell^2+\left(\frac{\ell^{-2i}T}{\ell}\right)^2-1 \right) & \text{ if }i'=n, n>2i\text{ and } \ell^{2i}\vert T\\
0&\text{ if }\ell^{\min(2i,n)}\nmid T\\
\end{cases}
$$ 
\end{proposition}
\begin{proof}
Using Lemma \ref{Hcounts} it remains to deal with the case of $n>2i'$ and compute the sum
$$C_{i,i',n}(T):=\sum_{t\in (\mathbb{Z}/\ell^{n}\mathbb{Z})^\times}\sum_{s\equiv (t/2)^2\mod \ell^{2i'}}h_{i'}(\Delta(t,s))h_i(\Delta(T-t,s))$$
under the assumption that $\ell^{2i}\vert T$. First observe that since discriminants run through all residue classes
$$\sum_{s\equiv (t/2)^2(\ell^{2i'})}\left(\frac{\ell^{-2i'}\Delta(t,s)}{\ell}\right)=\sum_{s\equiv (t/2)^2(\ell^{2i'})}\left(\frac{\ell^{-2i}\Delta(T-t,s)}{\ell}\right)=0.$$ We can therefore ignore the corresponding terms in the expansion of $h_{i'}(\Delta(t,s))\cdot h_i(\Delta(T-t,s))$. In what follows, we use crucially that $\Delta(T-t,s)-\Delta(t,s)=T\cdot (\text{unit})$ to establish the formulas.
\begin{itemize}
\item[Case 1.]
Assume $i'>i$ and $\ell^{2i+1}\vert T$.
Since $\ell^{2i+1}\vert \Delta(t,s)$ it follows that  $\ell^{2i+1}\vert \Delta(T-t,s)$ and thus 
$$\ell^2+\ell\left(\frac{\ell^{-2i}\Delta(T-t,s)}{\ell}\right)+\left(\frac{\ell^{-2i}\Delta(T-t,s)}{\ell}\right)^2-1=\ell^2-1.$$
Therefore it follows that 
\begin{align*}C_{i,i',n}(T)&=\sum_{t\in (\mathbb{Z}/\ell^{n})^\times}\sum_{y\in \mathbb{Z}/\ell^{n-2i'}}(\ell^2-1)\left( \ell^2+\left(\frac{y}{\ell}\right)+\left(\frac{y}{\ell}\right)^2-1 \right)\\
&=\sum_{t\in (\mathbb{Z}/\ell^{n})^\times}(\ell^2-1)(\ell^{n-2i'-1}(\ell^2-1)+\ell^{n-2i'-1}(\ell-1)\ell^2)\\
&=\sum_{t\in (\mathbb{Z}/\ell^{n})^\times}(\ell^2-1)\ell^{n-2i'-1}(\ell^3-1)\\
&=(\ell-1)(\ell^2-1)\ell^{2n-2i'-2}(\ell^3-1),\\
\end{align*}
establishing the claim in this case.\\

\item[Case 2.]
Assume $i=i'$ and $\ell^{2i+1}\vert T$. Since $\ell^{-2i}(\Delta(t,s)-\Delta(T-t,s))\equiv \ell^{-2i}T\equiv 0\mod \ell$, we may write 
\begin{align*}
C_{i',i',n}(T)&=\sum_{t\in (\mathbb{Z}/\ell^{n})^\times}\sum_{s\equiv(t/2)^2(\ell^{2i'})}\left( \ell^2+\ell\left(\frac{\ell^{-2i'}\Delta(t,s)}{\ell}\right)+\left(\frac{\ell^{-2i'}\Delta(t,s)}{\ell}\right)^2-1 \right)^2\\
&=\sum_{t\in (\mathbb{Z}/\ell^{n}\mathbb{Z})^\times}\sum_{y\in\mathbb{Z}/\ell^{n-2i'}}\left( \ell^2+\ell\left(\frac{y}{\ell}\right)+\left(\frac{y}{\ell}\right)^2-1 \right)^2\\
&=\sum_{t\in (\mathbb{Z}/\ell^{n}\mathbb{Z})^\times}\ell^{n-2i'-1}(\ell^2-1)^2+\ell^{n-2i'-1}(\ell-1)(\ell^4+\ell^2)\\
&=(\ell-1)\ell^{2n-2i'-2}(\ell-1)(\ell^4+\ell^3+2\ell^2-\ell-1).
\end{align*}

\item[Case 3.]
We assume that $\ell^{2i} \parallel T$. For any unit $t$ we set $u_t:=\ell^{-2i}T(T-2t)$, a unit, and examine the contributions to $C_{i,i',n}(T)$ of the different terms of $h_{i'}(\Delta(t,s))\cdot h_i(\Delta(T-t,s))$. We first evaluate the expression
\small{$$\sum_{t\in (\mathbb{Z}/\ell^{n}\mathbb{Z})^\times}\sum_{s\equiv (t/2)^2(\ell^{2i'})}\left(\frac{\ell^{-2i'}\Delta(t,s)}{\ell}\right)\left(\frac{\ell^{-2i}\Delta(T-t,s)}{\ell}\right)=\sum_{t\in (\mathbb{Z}/\ell^{n}\mathbb{Z})^\times}\sum_{y\in\mathbb{Z}/\ell^{n-2i'}}\left(\frac{y}{\ell}\right)\left(\frac{u_t+\ell^{2(i'-i)}y}{\ell}\right),$$}
writing $y$ for $\ell^{-2i'}\Delta(t,s)$.
If $i'>i$ this equals $\sum_{y\in \mathbb{Z}/\ell^{n-2i'}}\left(\frac{yu_t}{\ell}\right)=0$.
However, if $i'=i$ this becomes:
\begin{align*}
\sum_{y\in \mathbb{Z}/\ell^{n-2i'}}\sum_{t\in (\mathbb{Z}/\ell^{n})^\times}\left(\frac{yu_t+y^2}{\ell}\right)&=\sum_{y\in \mathbb{Z}/\ell^{n-2i'}}\sum_{u_t\in (\mathbb{Z}/\ell^{n})^\times}\left(\frac{yu_t+y^2}{\ell}\right)\\
&=\sum_{y\in (\mathbb{Z}/\ell^{n-2i'})^\times} -\ell^{n-1}\\
&=-\ell^{2n-2i'-2}(\ell-1).
\end{align*}
Moreover, observe that the term
$$\sum_{t\in (\mathbb{Z}/\ell^{n})^\times}\sum_{s\equiv (t/2)^2(\ell^{2i'})}\left(\left(\frac{\ell^{-2i'}\Delta(t,s)}{\ell}\right)^2-1\right)\left(\frac{\ell^{-2i}\Delta(T-t,s)}{\ell}\right)=0$$
vanishes, and this also holds after exchanging the roles of $T-t$ and $t$ when $i=i'$. Indeed, for a fixed trace $t$ we have with notations as above: 
\begin{align*}\sum_{s\equiv (t/2)^2(\ell^{2i'})}\left(\left(\frac{\ell^{-2i'}\Delta(t,s)}{\ell}\right)^2-1\right)\left(\frac{\ell^{-2i}\Delta(T-t,s)}{\ell}\right)&=\sum_{y\in \mathbb{Z}/\ell^{n-2i'}}\left( \left(\frac{y}{\ell}\right)^2-1 \right)\left(\frac{u_t+\ell^{2(i'-i)}y}{\ell}\right)\\
&=-\ell^{n-2i'-1}\left(\frac{u_t}{\ell}\right)\\
\end{align*}
so that summing over $t$ we get:
$$\sum_{t\in (\mathbb{Z}/\ell^{n})^\times}-\ell^{n-2i'-1}\left(\frac{u_t}{\ell}\right)=\sum_{t\in (\mathbb{Z}/\ell^{n})^\times}-\ell^{n-2i'-1}\left(\frac{(\ell^{-2i}T)2t}{\ell}\right)=0.$$
The remaining terms need to be treated separately depending on whether $i=i'$ or not:
\begin{itemize}
\item[Case 3a)] 
First consider the situation when $i'>i$. Since $\Delta(T-t,s)-\Delta(t,s)=T(T-2t)$ we see that $\ell^{2i} \parallel \Delta(T-t,s)$ and thus 
$$
\ell^2+\ell\left(\frac{\ell^{-2i}\Delta(T-t,s)}{\ell}\right)+\left(\frac{\ell^{-2i}\Delta(T-t,s)}{\ell}\right)^2-1=\ell\left( \ell+\left(\frac{\ell^{-2i}\Delta(T-t,s)}{\ell}\right) \right).
$$
Putting everything together we obtain when $i'>i$ the count
\begin{align*}
C_{i,i',n}(T)&=\sum_{t\in (\mathbb{Z}/\ell^{n})^\times}\sum_{s\equiv (t/2)^2(\ell^{2i'})}\ell^4+\ell^2\left( \left(\frac{\ell^{-2i'}\Delta(t,s)}{\ell}\right)^2-1 \right)\\
&=(\ell-1)\ell^{2n-2i'}(\ell^3-1).
\end{align*}
\item[Case 3b)]
We now assume $i'=i$. In this case we also remark that for fixed $t$ with notations as above
$$
\sum_{y\in \mathbb{Z}/\ell^{n-2i'}}\left( \left(\frac{y}{\ell}\right)^2-1 \right) \left( \left(\frac{y+u_t}{\ell}\right)^2-1 \right) = 0
$$
since $y\equiv 0\equiv y+u_t \mod \ell$ is impossible, so we get no positive contribution from these terms. We also use:
$$
\sum_{y\in \mathbb{Z}/\ell^{n-2i'}}\left( \left(\frac{y}{\ell}\right)^2-1 \right)+\left( \left(\frac{y+u_t}{\ell}\right)^2-1 \right) = -2\ell^{n-2i'-1}.
$$
Eliminating the terms in $h_{i'}(\Delta(t,s))\cdot h_{i}(\Delta(T-t,s))$ that cancel out, we thereby obtain that
\begin{align*}
C_{i',i',n}(T)&=\sum_{t\in (\mathbb{Z}/\ell^{n})^\times}\sum_{y\in\mathbb{Z}/(\ell^{n-2i'})} \left( \ell^4+\ell^2\left(\frac{y}{\ell}\right)\left(\frac{u_t+y}{\ell}\right)+\ell^2 \left( \left(\frac{y}{\ell}\right)^2+\left(\frac{u_t+y}{\ell}\right)^2-2 \right) \right)\\
&=\ell^4(\ell-1)\ell^{n-1}\ell^{n-2i'}+\ell^2(\ell-1)(-\ell^{2n-2i'-2})+\ell^2(\ell-1)\ell^{n-1}(-2\ell^{n-2i'-1})\\
&=(\ell-1)\ell^{2n-2i'}(\ell^3-3).
\end{align*}
\end{itemize}
\end{itemize}
The result now follows easily using Lemma \ref{Hcounts}.
\end{proof}

It remains to count pairs whose sum of traces is $T$ and which are not congruent to scalar matrices. 
\begin{Lemma}\label{notcongruent}
Assume $\ell\vert T$, then the number of pairs in $G(\ell^n,T)$ both not congruent to scalar matrices is
$$S_{0,0,n}(T):=\ell^{6n-6}(\ell-1)(\ell^5-\ell^3-3\ell^2+1)$$
\end{Lemma}
\begin{proof}
The count $S_{0,0,n}(T)$ by Lemma \ref{singlematrix} equals
$$ 
\sum\limits_{t\in \mathbb{Z}/\ell^{n}}\sum\limits_{s\in (\mathbb{Z}/\ell^{n})^\times}\ell^{4n-4} \left( \ell^2+\ell\left(\frac{\Delta(t,s)}{\ell}\right)+\left(\frac{\Delta(t,s)}{\ell}\right)^2-1 \right) \left( \ell^2+\ell\left(\frac{\Delta(T-t,s)}{\ell}\right)+\left(\frac{\Delta(T-t,s)}{\ell}\right)^2-1 \right).
$$ 
Since $\ell\vert T$, we have that $\Delta(t,s)\equiv \Delta (T-t,s) \mod \ell$ and we arrive at the count:

\begin{align*}S_{0,n}(T)&=\ell^{4n-4}\sum\limits_{t\in \mathbb{Z}/\ell^{n}}\sum\limits_{s\in (\mathbb{Z}/\ell^{n})^\times}  \left( \ell^2+\ell\left(\frac{\Delta(t,s)}{\ell}\right)+\left(\frac{\Delta(t,s)}{\ell}\right)^2-1 \right)^2\\
&=\ell^{4n-4}\sum\limits_{s\in (\mathbb{Z}/\ell^{n})^\times} \left(\sum\limits_{t^2\equiv 4s (\ell)}(\ell^2-1)^2+\sum\limits_{t^2\not\equiv 4s (\ell)}\left( \ell^4+2\ell^3\left(\frac{\Delta(t,s)}{\ell}\right)+\ell^2 \right) \right)\\
&=\ell^{4n-4}\sum\limits_{s\in (\mathbb{Z}/\ell^{n})^\times}\left(\ell^{n-1}(\ell^2-1)^2+(\ell-1)\ell^{n-1}(\ell^4+\ell^2)+2\ell^3(-\ell^{n-1})\right)\\
&=\ell^{6n-6}(\ell-1)(\ell^5-\ell^3-3\ell^2+1).
\end{align*}

\end{proof}

We can now establish explicit formulas for these local densities:
\begin{theorem}\label{primesdividingtrace}
Let $T\in\mathbb{Z}$ and let $\ell\neq 2$ be a prime so that $v_\ell(T)>0$. 
\begin{enumerate}
\item Assume $n\geq v_\ell(T)+1$.
Then the number of $\mathbb{Z}/\ell^{n}\mathbb{Z}$-points on $\GL_2\times_{\det}\GL_2$ consisting of pairs of matrices with sum of traces $T$ equals
$$\vert H(\ell^n,T)\vert=(\ell-1)\ell^{6n-4e_\ell(T)-9}(\ell^{4e_\ell(T)+3}(\ell^5-\ell^3-3\ell^2+1)+E(\ell,T))+\delta_2(v_\ell(T))(\ell^3-1)).$$
\item We have that 
$$
\lim_{n\to \infty}\frac{\ell^n\vert H(\ell^n,T)\vert}{\vert H(\ell^n)\vert}=
\frac{\ell^6-\ell^4-3\ell^3-\ell}{\ell^6-2\ell^4+\ell^2}+\frac{E(\ell,T)+\delta_2(v_\ell(T))(\ell^3-1)}{\ell^{4e_\ell(T)+4}(\ell^2-1)^2},
$$
\end{enumerate}
where $e_\ell(T)=\lfloor (v_\ell(T)-1)/2\rfloor$, where $\delta_2(v_\ell(T))=1$ if $2\vert v_\ell(T)$ and $\delta_2(v_\ell(T))=0$ otherwise, and where $E(\ell,T)\in \mathbb{Z}(\ell)$ is given by:

$$E(\ell,T)=\frac{\ell^3(\ell+1)(\ell^{4e_\ell(T)+4}-1)+(\ell^4+\ell^3+2\ell^2-\ell-1)(\ell^{4e_\ell(T)}-1)}{(\ell^2+1)(\ell+1)}.$$
\end{theorem}

For instance, 
we deduce the following counts:
$$\vert H(\ell^n,T)\vert=\begin{cases}(\ell-1)\ell^{6n-6}(\ell^5-\ell^3-\ell^2-1) &\text{ when }v_\ell(T)=1,\\
(\ell-1)\ell^{6n-6}(\ell^7-\ell^5-\ell^4+\ell^3-\ell^2-1) &\text{ when }v_\ell(T)=2,\\
(\ell-1)\ell^{6n-13}(\ell^{12}-\ell^{10}-\ell^7+3\ell^5-\ell^3-3\ell^2+2\ell-1) &\text{ when }v_\ell(T)=3,
\end{cases}$$

\begin{proof}
The second assertion follows from the first since 
$$\vert H(\ell^n)\vert=\ell^{7(n-1)}\vert G(\ell)\vert=\ell^{7(n-1)}(\ell-1)\ell^2(\ell^2-1)^2.$$
To prove the formula for $H(\ell^n,T)$ we observe that 
$$\vert H(\ell^n,T)\vert:=2\sum_{i=0}^{\lfloor v_\ell(T)/2 \rfloor}\sum_{i'=\max{(i+1,1)}}^nS_{i,i',n}(T)+\sum_{i=0}^{\lfloor v_\ell(T)/2 \rfloor}S_{i,i,n}(T).$$
Recall $e_\ell(T)=\lfloor (v_\ell(T)-1)/2\rfloor$, so that in particular $n\geq e_\ell(T)+2$. We evaluate the sum using Proposition \ref{exactcount}. We compute the term:
\begin{align*}
\sum\limits_{i=0}^{e_\ell(T)}\sum\limits_{i'=i+1}^{n-1}S_{i,i',n}(T)&=\sum\limits_{i=0}^{e_\ell(T)}\sum\limits_{i'=i+1}^{n-1}(\ell-1)\ell^{6n-6-i}(\ell^3-1)(\ell^2-1)\ell^{-3i'}\\
&=(\ell-1)\ell^{6n-6}(\ell^3-1)(\ell^2-1)\sum\limits_{i=0}^{e_\ell(T)} \ell^{-i} \frac{\ell^{-3i-3}-\ell^{-3n}}{1-\ell^{-3}}\\
&=(\ell-1)\ell^{6n-6}(\ell^2-1)\sum\limits_{i=0}^{e_\ell(T)} (\ell^{-4i}-\ell^{3-3n}\ell^{-i})\\
&=(\ell-1)\ell^{6n-6}(\ell^2-1)\left(\frac{\ell^4-\ell^{-4e_\ell(T)}}{\ell^4-1}-\ell^{-3n+3}\frac{\ell-\ell^{-e_\ell(T)}}{\ell-1}\right)\\
&=(\ell-1)\ell^{6n-6}\frac{\ell^4-\ell^{-4e_\ell(T)}}{\ell^2+1}-\ell^{3n-3}(\ell^2-1)(\ell-\ell^{-e_\ell(T)})
\end{align*}
We also compute the remaining terms:
\begin{align*}
\sum\limits_{i=0}^{e_\ell(T)} S_{i,n,n}(T)&=\sum\limits_{i=0}^{e_\ell(T)}(\ell-1)\ell^{3n-i-6}\ell^3(\ell^2-1)\\
&=\ell^{3n-3}(\ell^2-1)(\ell-\ell^{-e_\ell(T)})
\end{align*}
and also provided $e_\ell(T)\geq 1$:
\begin{align*}
\sum\limits_{i=1}^{e_\ell(T)} S_{i,i,n}(T)&=(\ell-1)\ell^{6n-6}(\ell-1)(\ell^4+\ell^3+2\ell^2-\ell-1)\sum\limits_{i=1}^{e_\ell(T)}\ell^{-4i}\\
&=(\ell-1)\ell^{6n-9}(\ell^4+\ell^3+2\ell^2-\ell-1)\left( \frac{1-\ell^{-4e_\ell(T)}}{\ell^3+\ell^2+\ell+1} \right).
\end{align*}
Finally, when $v_\ell(T)$ is even we have that $\ell^{2(e_\ell(T)+1)}\parallel T$ and there are the additional terms:
\begin{align*}
&S_{e_\ell(T)+1, e_\ell(T)+1, n}(T)+2\sum\limits_{i'=e_\ell(T)+2}^{n-1}S_{e_\ell(T)+1,i',n}(T)+2S_{e_\ell(T)+1,n,n}(T)\\
&=(\ell-1)\ell^{6n-e_\ell(T)-7}\left(\ell^2(\ell^3-3)\ell^{-3(e_\ell(T)+1)}+2(\ell^3-1)\ell^2 \sum\limits_{i'=e_\ell(T)+2}^{n-1}\ell^{-3i'}+2\ell^2\ell^{-3n+3}\right)\\
&=(\ell-1)\ell^{6n-4e_\ell(T)-8}(\ell^3-1).
\end{align*}
to account for. The result follows by adding all the terms and $S_{0,0,n}(T)$ computed in Lemma \ref{notcongruent}.

\end{proof}

\subsection{The prime $\ell=2$}
Throughout this subsection $\ell=2$ as we deal with this remaining case. We will use both notations $\ell$ and $2$ interchangeably for comparison with the odd prime case. This prime requires special treatment, for instance the discriminant plays a crucial role in all of our calculations, but involves a factor of $4$. This explains why we have to look modulo $8$ in the following result: 
\begin{Lemma}\label{notscalar}
Provided $n\geq 1$, the number of matrices in $M_{2\times 2}(\mathbb{Z}/\ell^n\mathbb{Z})$ not congruent to scalar of trace and determinant $(t,s)$ is 
$$h_{0,n}(t,s)=\ell^{2n-2}\cdot \begin{cases}
\ell^2-1 &\text{ if }\Delta(t,s)\equiv 0 \mod 4\\
\ell^2-\ell &\text{ if }\Delta(t,s)\equiv 5 \mod 8\\
\ell^2+\ell &\text{ if }\Delta(t,s)\equiv 1\mod 8\\
0& \text{ otherwise. }\\
\end{cases}
$$
\end{Lemma}
\begin{proof}
This is easily checked for $n=1$ and it suffices to show that there are $\ell^2$ lifts from $\mathbb{Z}/\ell^n\mathbb{Z}$ to $\mathbb{Z}/\ell^{n+1}\mathbb{Z}$ of  trace and determinant congruent to $(t,s)$ modulo $\ell^{n+1}$. The lifts of a fixed matrix may be written as 
$$\begin{pmatrix}b_0&c_0\\
d_0&e_0\\
\end{pmatrix}+2^n\begin{pmatrix}b&c\\
d&e\\
\end{pmatrix}$$
and we count the classes of $b,c,d,e$ modulo $2$. The trace condition yields $e\equiv -b\mod \ell$ 
and the determinant condition yields $b_0e+be_0-cd_0-dc_0\equiv 0 \mod \ell$ yielding exactly $\ell^2$ solutions if and only if $\begin{pmatrix}b_0&c_0\\
d_0&e_0\\
\end{pmatrix}$ isn't a scalar modulo $\ell$.
\end{proof}
We deduce the rest of the counts from Lemma \ref{notscalar}. We summarize the results here:
\begin{proposition}\label{primetwocount}
The number of matrices congruent to scalar matrices exactly modulo $2^i$ and of trace and determinant $(t,s)$ is, provided $n\geq 3$:
\begin{enumerate}
\item If $i<n/2$
$$h_{i,n}(t,s)=\ell^{2n-2-i}\cdot \begin{cases}
\ell^2-1 &\text{ if }\ell^{-2i}\Delta(t,s)\equiv 0 \mod 4\\
\ell^2-\ell &\text{ if }\ell^{-2i}\Delta(t,s)\equiv 5 \mod 8\\
\ell^2+\ell &\text{ if }\ell^{-2i}\Delta(t,s)\equiv 1\mod 8\\
0& \text{ otherwise. }\\
\end{cases}
$$\\
\item
If $n>i\geq n/2+1$ then 
$$h_{i,n}(t,s)=\ell^{3n-3i-2}\cdot \begin{cases}\ell^3-1&\text{ if }\ell^{n+2} \vert \Delta(t,s)\\
0& \text{ otherwise.}\\
\end{cases}
$$

\item and the remaining cases: 
$$h_{i,n}(t,s)=\ell^{3n-3i-2}\cdot \begin{cases}
\ell^2-1 &\text{ if }\ell^{-2i}\Delta(t,s)\equiv 0 \mod 4\\
\ell^2 &\text{ if }\ell^{-2i}\Delta(t,s)\equiv 1 \mod 4\\
0& \text{ otherwise. }\\
\end{cases}$$
if $i=n/2$ as well as if $i=(n+1)/2$
$$
h_{i,n}(t,s)=\ell^{3n-3i-2}\cdot \begin{cases}
\ell^2-1 &\text{ if }\ell^{-2i}\Delta(t,s)\equiv 0 \mod 2\\
\ell^2 &\text{ if }\ell^{-2i}\Delta(t,s)\equiv 1 \mod 2\\
0& \text{ if }\ell^{2i}\nmid\Delta(t,s).\\
\end{cases}
$$
\end{enumerate}
\end{proposition}
 \begin{proof}
We count matrices congruent to scalars exacly modulo $2^i$ of trace/det $(t,s)$, so of the form 
$$aI+2^i\begin{pmatrix}b&c\\
d&e\\
\end{pmatrix}$$
where $\begin{pmatrix}b&c\\
d&e\\
\end{pmatrix}\in S_0(n-i)$ are matrices not congruent to scalar in $\mathbb{Z}/2^{n-i}\mathbb{Z}$ and $a\in \mathbb{Z}/2^i\mathbb{Z}$.
The trace $t$ must be even and there are now two choices for $a\mod 2^i$ such that $2a\equiv t\mod 2^i$ which are congruent modulo $2^{i-1}$. We must have for any choice of $a$:
\begin{alignat*}{2}2^{2i}(be-dc)&\equiv a^2-at+s &\mod \ell^n &\text{ (det. condition) }\\
2^i(b+e)&\equiv t-2a &\mod \ell^n&\text{ (trace condition). }
\end{alignat*}

First, observe that solving the determinant condition yields two choices of $a$ congruent modulo $2^i$, whereas the trace condition yields two choices congruent modulo $2^{i-1}$. Thus when $i<n/2$ there is a unique $a \mod 2^i$ such that both conditions are satisfied, and this happens if and only if $2^{2i}\vert \Delta(t,s)$. 
Relating the discriminant of trace $\tilde{t}=b+e$ and determinant $\tilde{s}=be-dc$ of this new matrix to $\Delta(t,s)$ one also has that:
$$2^{2i}(\tilde{t}^2-4 \tilde{s})\equiv (t-2a)^2-4(a^2-at+s)\equiv \Delta(t,s) \mod 2^{n}.$$
When $i<n/2$ we deduce the following count: 

$$\ell^{2n-2-i}\cdot \begin{cases}
\ell^2-1 &\text{ if }2^{-2i}\Delta(t,s)\equiv 0 \mod 4\\
\ell^2-\ell &\text{ if }2^{-2i}\Delta(t,s)\equiv 5 \mod 8\\
\ell^2+\ell &\text{ if }2^{-2i}\Delta(t,s)\equiv 1\mod 8\\
0& \text{ otherwise. }\\
\end{cases}
$$
This follows from Lemma \ref{notscalar} as $\tilde{t}$ is a fixed class mod $2^{n-i}$ and $\tilde{s}$ has to be a fixed class mod $2^{n-2i}$, which yields $2^i$ choices for $\tilde{s}$. So the count is $2^i h_{0,n-i}(\tilde{t},\tilde{s})$ as claimed. \\

Now let us consider the case of $i\geq n/2$. The determinant condition yields
\begin{align*}
2^{2i}\tilde{s}&\equiv a^2-at+s\\
&\equiv 1/4(t-2^{i}\tilde{t})^2-t/2 (t-2^{i}\tilde{t})+s\\
&\equiv -t^2/4-2^{i-1}t\tilde{t}+2^{i-1}t\tilde{t}+s+2^{2i-2}\tilde{t}^2\\
&\equiv -\Delta(t,s)/4+2^{2i-2}\tilde{t}^2 \mod 2^n
\end{align*}
If $i\geq n/2+1$, this is equivalent to $2^{n+2} \vert \Delta(t,s)$. In this case for each of the two solutions for $a \mod 2^i$ we get that $\tilde{t}\mod 2^{n-i}$ is $0$ or $1$, respectively and $\tilde{s}$ runs through all the classes modulo $2^{n-i}$. Summing over all possibilities and keeping track of $\Delta(\tilde{t},\tilde{s})\mod 8$ as in Lemma \ref{notscalar} we obtain:
$$\ell^{2(n-i)-2}\cdot \begin{cases}\ell^{n-i}(\ell^2-1)+\ell^{n-i}(\ell^2)&\text{ if }\ell^{n+2} \vert \Delta(t,s)\\
0& \text{ otherwise,}\\
\end{cases}
$$
as claimed. \par
It remains to treat the case of $i=\lfloor (n+1)/2\rfloor $ when we still need to work modulo $\ell^{n+2}$. Since $\ell^{2i}\Delta(\tilde{t}, \tilde{s}) \equiv \Delta(t,s)\mod \ell^{n+2}$ we get a trivial count if $\ell^{2i}\nmid \Delta(t,s)$. 
 If $n$ is even then $i=n/2$, and we have $\Delta(\tilde{t}, \tilde{s}) \equiv \ell^{-n}\Delta(t,s) \mod 4$. Since $\Delta(t,s)$ doesn't see if $\Delta(\tilde{t}, \tilde{s})$ is $1$ or $5$ modulo $8$ we just get the average of those two counts for $h_{0,n-i}(\tilde{t},\tilde{s})$, as there is no condition on $\tilde{s}\in \mathbb{Z}/\ell^{n-i}\mathbb{Z}$:
$$\ell^{3n-3i-2}\cdot \begin{cases}
\ell^2-1 &\text{ if }\ell^{-2i}\Delta(t,s)\equiv 0 \mod 4\\
\ell^2 &\text{ if }\ell^{-2i}\Delta(t,s)\equiv 1 \mod 4\\
0& \text{ otherwise. }\\
\end{cases}
$$
Finally if $n$ is odd, then $i=(n+1)/2$  and $\Delta(\tilde{t}, \tilde{s}) \equiv \ell^{-n}\Delta(t,s) \mod 2$.
 We get the count:
$$\ell^{3n-3i-2}\cdot \begin{cases}
\ell^2-1 &\text{ if }2^{-2i}\Delta(t,s)\equiv 0 \mod 2\\
\ell^2 &\text{ if }2^{-2i}\Delta(t,s)\equiv 1 \mod 2\\
0& \text{ if }2^{2i}\not \vert\Delta(t,s)\\
\end{cases}
$$
where now again $\tilde{s}$ can vary in $\mathbb{Z}/\ell^{n-i}\mathbb{Z}$ and we get a similar average.
\end{proof}

\section{The exceptional factor}\label{section:bad primes}

This section is devoted to the study of the factor at the \emph{conductor} $m_A$ of the open subgroup $\im (\rho_A) \subseteq G(\hat{\mbz})$ where again we just take throughout this section $A=E_1\times E_2$. Recall that $m_A$ is defined as the least integer $m$ such that, in the diagram
\[
\begin{tikzcd}
    \Gal(\overline{\mathbb{Q}}/\mathbb{Q}) \ar[r, "\rho_A"] \ar[dr, "\bar\rho_{A,m}"']
        & G(\hat\bZ) \ar[d, "\rred"] \\
    & G(m),
\end{tikzcd}
\]
we have $\im (\rho_A) = \rred^{-1} H_A(m)$, keeping with our notations of $G=\GL_2\times_{\det}\GL_2$ and $H_A(m)=\im(\overline{\rho}_{A,m})$.  \\
In particular, we would like to compute the stable factor 
$$ \frac{m_{A,T} \vert H_A(m_{A,T},T) \vert}{\vert H_A(m_{A,T})\vert}$$
that appears in Conjecture \ref{conj:main}. In general, it is difficult to compute this factor systematically and we have seen it can even vanish depending on the trace $T$. However, we will be able to compute it in the case when $\rho_A$ has image as large as possible in $G(\widehat{\bZ})$: 

\begin{Definition}
A pair $(E_1,E_2)$ of elliptic curves over $\bQ$ is a \emph{Serre pair} if the image of $\rho_{E_1\times E_2}$ has index $4$ in 
$G(\hat\bZ)$. 
\end{Definition}

In particular, since the image of $\rho_{E_i}$ in $\GL_2(\widehat{\bZ})$ has index at least $2$, each individual elliptic curve $E_i$ must have maximal possible image, ergo must be a \emph{Serre curve}. It has been shown by the second author \cite{MR3071819,MR2563740} that almost all (ordered by height) elliptic curves and pairs of curves over $\bQ$ are Serre curves or Serre pairs, respectively. \par

This image of $\rho_{E_i}$ is going to be by definition the full preimage under $\rred :\GL_2(\widehat{\bZ})\to \GL_2(\bZ/m_E\bZ)$ of an index two subgroup of $\GL_2(\bZ/m_E\bZ)$.
We briefly characterize this subgroup, following \cite[Section 3]{MR2534114}:  \\
Since the discriminant of the elliptic curve has a concrete expression as a symmetric polynomial in the coordinates of the $2$-torsion of $E$, it is easy to see that any element of the Galois group $\sigma \in \Gal(E[2])$ has to act on $\sqrt{\Delta_E}$ via the unique non-trivial quadratic character of $\GL_2(\bZ/2\bZ)$, which we denote by $\varepsilon:\GL_2(\bZ/2\bZ)\cong S_3\to \{\pm 1\} $, so we have that:
$$\sigma(\sqrt{\Delta_E})= \epsilon(\sigma)\sqrt{\Delta_E}$$
and $\bQ(\sqrt{\Delta_E})\subseteq\bQ(E[2])$. By class field theory, we also have the containment  $\bQ(\sqrt{\Delta_E})\subseteq\bQ(\zeta_{D_E})$ for some primitive root of unity $\zeta_{D_E}$ of order $D_E$. Since the conductor of the quadratic field $\bQ(\sqrt{\Delta_E})$ is the absolute value of its discriminant (for instance by the F\"uhrerdiskriminantenproduktformel), we have that 
$$D_E=\begin{cases}\vert \Delta_{sf}\vert& \text{if }  \Delta_{sf}\equiv 1\mod 4\\
4\vert \Delta_{sf}\vert& \text{otherwise}
\end{cases}$$
with $\Delta_{sf}$ denoting the squarefree part of the discriminant $\Delta_E$ of the elliptic curve. 
Letting now $m_E:= \lcm(D_E,2)$ so that the field $\bQ(E[m_E])$ is the compositum of $\bQ(E[D_E])$ and $\bQ(E[2])$, we must have that for $\sigma\in \Gal(E[m_E])$: 
$$\sigma(\sqrt{\Delta_E})= \left(\frac{\Delta_{sf}}{\det\sigma}\right)\cdot\sqrt{\Delta_E},$$
using the Kronecker symbol. 
Therefore, for any elliptic curve $E$, the image of $\bar{\rho}_{E,m_E}\subseteq \GL_2(\bZ/m_E\bZ)$ is contained in the index two subgroup $\ker(\varepsilon(\cdot)\left(\frac{\Delta_{sf}}{\det(\cdot)}\right))$. Moreover, $E$ is a Serre curve if and only if this subgroup is exactly the image of $\bar{\rho}_{E,m_E}$ and $m_E$ is the conductor of the open subgroup $\im (\rho_E) \subseteq \GL_2(\hat{\mbz})$. \par
Considering now pairs of elliptic curves and the abelian surface $A=E_1\times E_2$, we get therefore an obstruction to the Galois representation on torsion having full image in the group $G(\widehat{\bZ})$ for $G=\GL_2\times_{\det}\GL_2$. Namely, considering the compositum $L=\bQ(E_1[D_1])\cdot\bQ(E_2[D_2])\cdot\bQ((E_1\times E_2)[2])$, we must have for $\sigma\in\Gal(L)$: 
$$\sigma(\sqrt{\Delta_{E_i}})= \left(\frac{\Delta_{E_i,sf}}{\det\sigma}\right)\cdot\sqrt{\Delta_{E_i}}$$
for $i=1,2$. We remark that even though each individual elliptic curve is a Serre curve, there could be additional obstructions coming from entanglements between the division fields of the $E_i$. However, since this obstruction forces the image of $\rho_A$ to lie in an index $4$ subgroup of $G(\widehat{\bZ})$, for Serre pairs this has to be the only obstruction. Therefore, setting
$$m_A:=\lcm(m_{E_1},m_{E_2})=\lcm(2,D_{E_1},D_{E_2}),$$
$m_A$ is precisely the integer we called the conductor of the open subgroup $\im (\rho_A) \subseteq G(\hat{\mbz})$ so that the image of $\bar{\rho}_{A,m_A}\subseteq G(m_A)$ is the index $4$ subgroup $\ker(\boldsymbol{\psi}:G(m_A)\to \{\pm 1\}^2)$ for 
\[
\boldsymbol{\psi}(g)=(\varepsilon(g^{(1)})\left(\frac{\Delta_{E_1,sf}}{\det(g^{(1)})}\right),\varepsilon(g^{(2)})\left(\frac{\Delta_{E_2,sf}}{\det(g^{(2)})}\right))
\]
where $g=(g^{(1)},g^{(2)})$ denote the projections onto each $\GL_2$-factor of $G$. The preimage of this subgroup under projection is the full image of $\rho_A$ inside $G(\widehat{\bZ})$. 
The next subsection is devoted to computing the relevant counts for Serre curves. We conclude this discussion by providing pairs of Serre curves which are not Serre pairs, showing that entanglements between division fields are indeed a prevalent issue. 
\begin{proposition}\label{prop:Serrenopair}
There are infinitely many pairs $(E_1,E_2)$ of elliptic curves over $\bQ$ for which each $E_i$ is a Serre curve but which, nevertheless, do not constitute a Serre pair. 
\end{proposition}
\begin{proof}
We prove this by exhibiting an explicit family. Such examples may be found via the following strategy: fix an explicit Serre curve $E_1/\bQ$. One may then construct explicit families of elliptic curves with the same mod $n$ Galois representation for some $n$. Clearly any Serre curve $E_2$ in that family will then produce an abelian surface $A=E_1\times E_2$ with $\overline{\rho}_{A,n}(G_\bQ)$ of index at least $\vert\SL_2(\bZ/n\bZ)\vert$ in $G(\bZ/n\bZ)$. It remains to show that such a Serre curve can be found in the family, but one certainly expects most elliptic curves in such a family to be Serre curves. Concretely, Cojocaru, Grant and Jones show in the main result of \cite{MR2837018} that for a non-isotrivial elliptic curve $E/\bQ(t)$, a positive proportion of specializations are Serre curves, provided the Galois representation associated to $E$ is sufficiently large. Namely, they require that 
$$[\Gal(\bQ(t)(E[n])/\bQ(t)):\GL_2(\bZ/n\bZ)]=1\text{ or }2.$$ 
Moreover, they show \cite[Proposition 4]{MR2837018} that to prove full Galois image it suffices to show that $\bQ(t)(\sqrt{\Delta_E})\cap\overline{\bQ}=\bQ$ and that 
$$ \Gal(\bQ(t)(E[n])/\bQ(t))=\GL_2(\bZ/n\bZ) \text{ for }n\in\{36, 5,7,11,13\}.$$
The latter condition will be automatically satisfied if we start with a Serre curve $E_1$ where no primes $l\geq 5$ are exceptional and where $n=36$ is not exceptional (meaning that the Galois representation modulo $n$ has image $\GL_2(\bZ/n\bZ)$). There are infinitely many such examples in \cite{MR3349445}. Finally, checking non-isotriviality and that adjoining the square-root of the discriminant is geometric is straightforward for explicit families. \\
Concretely, one may for instance take the Serre curve $E_1: y^2+xy=x^3+3$ and consider the family of elliptic curves with isomorphic Galois module $E[4]$ as given explicitly in \cite[Theorem 4.1.]{MR1638488} to obtain such an infinite family of Serre curves $E_2$ such that $E_1, E_2$ is not a Serre pair. 
\end{proof}

\subsection{Computing the factor at $m_A$}
The main result of this subsection, Proposition \ref{badprimesmainprop}, can be in particular used to compute the quantity $\vert H_A(m_{A,T},T) \vert$ appearing in Conjecture \ref{conj:main} in the case of Serre pairs, via our previous considerations. \\

It will be convenient to fix some additional notations here: we will be calculating the size of a certain subset of $G(m)$, where $m$ is a composite number, and this will utilize the Chinese Remainder Theorem.  As such, we choose to reserve subscripts for the Chinese Remainder Theorem decomposition and superscripts for the decomposition inherent in $G=\GL_2 \times_{\det} \GL_2$, i.e. we will use the following notation for ${\mathbf{g}} \in G(m)$:
\[
\begin{array}{ccccc}
G(m) & := & \GL_2(\bZ/m\bZ) \times_{\det} \GL_2(\bZ/m\bZ) & \simeq & \ds \prod_{\ell^\alpha \parallel m}  \GL_2(\bZ/\ell^\alpha\bZ) \times_{\det} \GL_2(\bZ/\ell^\alpha\bZ) \\
{\mathbf{g}} & \leftrightarrow & (g^{(1)},g^{(2)}) & \leftrightarrow & \prod_{l\mid m}\left( g^{(1)}_\ell, g^{(2)}_\ell \right).
\end{array}
\]
Let $D_1$ and $D_2$ be the discriminants of the quadratic fields $\mathbb{Q}(\sqrt{\Delta_{E_i}})$ for $i=1,2$. Since $(E_1,E_2)$ is a Serre pair, it follows that for $A = E_1 \times E_2$, we have $m_A = \lcm(2,|D_1|,|D_2|)$. Let $m \in \bN$ be any positive integer satisfying 
\[
m_A \mid m \mid m_A^\infty.
\]
Let
$
\ds \chi^{(i)}(\cdot) := \left( \frac{D_{i}}{\cdot} \right)
$
denote the Kronecker symbol, and regard each character $\chi^{(i)}$ as a function on $(\bZ/m\bZ)^\times$, which as before we further extend to a character on the group $\GL_2(\bZ/m\bZ)$ by pre-composing with the determinant map\footnote{We apologize for the abuse of notation here; in general when $\psi$ is a Dirichlet character of conductor dividing $m$, its domain will sometimes be regarded as $(\bZ/m\bZ)^\times$ and sometimes as $\GL_2(\bZ/m\bZ)$, via pre-composing with determinant.}, i.e. we set
$
\chi^{(i)}(g) := \chi^{(i)}(\det g).
$
We recall that $\ve : \GL_2(\bZ/2\bZ) \longrightarrow \{ \pm 1 \}$ denotes the unique non-trivial character of order $2$, viewed as a function on $\GL_2(\bZ/m\bZ)$ by first reducing modulo $2$.  Finally, we consider for general $m_A \mid m \mid m_A^\infty$ the map
\[
{\boldsymbol{\psi}} : G(m) \longrightarrow \{ \pm 1 \}^2
\]
defined by ${\boldsymbol{\psi}}({\mathbf{g}}) := (\psi^{(1)}(g^{(1)}),\psi^{(2)}(g^{(2)})) \in \{ \pm 1 \}^2$ with 
$\psi^{(i)}(g) := \ve (g) \cdot \chi^{(i)}(g) \in \{ \pm 1 \}$.
We will use the abbreviation
\begin{equation} \label{defofnotation} 
G(m,T)_{{\boldsymbol{\psi}} = {\boldsymbol{1}}} := \{ {\mathbf{g}} \in G(m,T) : {\boldsymbol{\psi}}({\mathbf{g}}) = {\boldsymbol{1}} \},
\end{equation}
and by our previous considerations it is precisely the size of this set when $m$ is divisible enough by any prime $\ell \mid m_A$ that we ultimately wish to evaluate. \par
Write $m = 2^\ga m'$ where $\ga \geq 1$ and $m'$ is odd and define the odd part of the discriminant $D_i$ to be
\begin{equation} \label{defofDsubi}
D_i' := \gcd(m',D_i).
\end{equation}
We may write a decomposition ${\boldsymbol{\psi}} = {\boldsymbol{\psi}}_2 \cdot {\boldsymbol{\psi}}_{m'}$, where $\psi^{(i)}_{m'}:=\left( \frac{\cdot}{\vert D_i'\vert}\right)$ and $\psi^{(i)}_{2}$ is a character modulo $8$. The size  $\ds | G(m,T)_{{\boldsymbol{\psi}} = {\boldsymbol{1}}} |$ depends somewhat on whether or not the characters $\psi^{(1)}_{m'}$ and $\psi^{(2)}_{m'}$ are equal, or, equivalently, whether or not $|D_1'| = |D_2'|$. 
In the following proposition, we denote: 
\begin{equation} \label{lotsofdefinitions}
\begin{split}
e_\ell(T) :=& \lfloor (v_\ell(T)-1)/2 \rfloor, \quad\quad \gd_n(m) :=
\begin{cases}
1 & \text{ if } n \text{ divides } m \\
0 & \text{ otherwise,}
\end{cases} \\
E(\ell,T) :=& \frac{\ell^3(\ell+1)(\ell^{4e_\ell(T)+4}-1) + (\ell^4 + \ell^3 + 2\ell^2 + \ell - 1)(\ell^{4e_\ell(T)}-1)}{(\ell^2+1)(\ell+1)}, \\
A_\ell(T) :=&
\begin{cases}
 \ell^5 - \ell^3 - 3\ell^2 + 1 + \frac{2E(\ell,T)}{\ell^{4e_\ell(T)+3}} + \gd_2(v_\ell(T)) \frac{(\ell^3-1)}{\ell^{4e_\ell(T)+2}} & \text{ if $\ell \mid T$} \\ 
(\ell^6 - \ell^5 - 2\ell^4 + \ell^3 + 2\ell^2)/(\ell-1) & \text{ if $\ell \nmid T$,}
\end{cases} \\
B_\ell(T) :=& 
\begin{cases}
-3\ell^2 + 1 + \frac{2E(\ell,T)}{\ell^{4e_\ell(T)+3}} + \gd_2(v_\ell(T)) \frac{(\ell^3-1)}{\ell^{4e_\ell(T)+2}} & \text{ if $\ell \mid T$} \\
\ell^2/(\ell-1) & \text{ if $\ell \nmid T$}.
\end{cases} \\
\mathfrak{S}_{{\boldsymbol{\psi}}_2}(T) &:=
\left| G\left(2^{v_2(m)},T\right)_{\im {\boldsymbol{\psi}}_2 \subseteq \im {\boldsymbol{\psi}}_{m'}} \right|, \\
\mathfrak{S}_{{\boldsymbol{\psi}}_2}^{(i)}(T) &:=
\left| G\left(2^{v_2(m)},T\right)_{\im {\boldsymbol{\psi}}_2 \subseteq \im {\boldsymbol{\psi}}_{m'}}^{\psi_2^{(i)} = 1} \right| - \left| G\left(2^{v_2(m)},T\right)_{\im {\boldsymbol{\psi}}_2 \subseteq \im {\boldsymbol{\psi}}_{m'}}^{\psi_2^{(i)} = -1} \right|.
\end{split}
\end{equation}
\begin{proposition} \label{badprimesmainprop}
Assume the notation just outlined and let $T \in \mathbb{Z} - \{ 0 \}$ be arbitrary. Assume that $m$ is large enough so that for all primes $\ell\vert m_A$ we have $v_\ell(m)>v_\ell(T)$. In case $|D_1'| \neq |D_2'|$, the quantity $\ds \left| G(m,T)_{{\boldsymbol{\psi}} = {\boldsymbol{1}}} \right|$ is equal to
\[
\frac{(m')^5\varphi(m')}{4\left( \rad(m') \right)^5} \left( \mathfrak{S}_{{\boldsymbol{\psi}}_2}(T) \prod_{\ell \mid m'} \left( A_\ell(T) + 1\right) 
+ 2^{\omega(\frac{m'}{|D_1'|})} \mathfrak{S}_{{\boldsymbol{\psi}}_2}^{(1)}(T) \prod_{\ell \mid D_1'} B_\ell(T) \\
+ 2^{\omega(\frac{m'}{|D_2'|})} \mathfrak{S}_{{\boldsymbol{\psi}}_2}^{(2)}(T) \prod_{\ell \mid D_2'} B_\ell(T) \right).
\]
In case $|D_1'| = |D_2'| =: D'$, $\ds \left| G(m,T)_{{\boldsymbol{\psi}} = {\boldsymbol{1}}} \right|$ is equal to
\[
\frac{(m')^5\varphi(m')}{2\left( \rad(m') \right)^5} \left( \mathfrak{S}_{{\boldsymbol{\psi}}_2}(T) \prod_{\ell \mid m'} \left( A_\ell(T) + 1 \right) \\
+ 2^{\omega(\frac{m'}{D'})} \mathfrak{S}_{{\boldsymbol{\psi}}_2}^{(1)}(T) \prod_{\ell \mid D'} B_\ell(T) \right).
\]
\end{proposition}
\begin{proof}
By \eqref{defofnotation}, we find that
\begin{equation} \label{nnewmaingoal}
\begin{split}
\left| G(m,T)_{{\boldsymbol{\psi}} = {\boldsymbol{1}}} \right| &= \sum_{{\mathbf{t}} \in (\bZ/m\bZ)^2 \atop t^{(1)} + t^{(2)} = T} \left| \GL_2(\bZ/m\bZ)_{\psi^{(1)} = 1}^{\tr \equiv t^{(1)}} \times_{\det} \GL_2(\bZ/m\bZ)_{\psi^{(2)} = 1}^{\tr \equiv t^{(2)}} \right| \\
&= 
\sum_{{\mathbf{t}} \in (\bZ/m\bZ)^2 \atop t^{(1)} + t^{(2)} = T} \sum_{d \in (\bZ/m\bZ)^\times} \left| \GL_2(\bZ/m\bZ)_{\psi^{(1)} = 1 \atop \det \equiv d}^{\tr \equiv t^{(1)}} \right| \cdot \left| \GL_2(\bZ/m\bZ)_{\psi^{(2)} = 1 \atop \det \equiv d}^{\tr \equiv t^{(2)}} \right|.
\end{split}
\end{equation}
Recall that $m = 2^\ga m'$ where $\ga \geq 1$ and $m'$ is odd, and likewise that
$D_i'$ is odd part of $D_i$.
As noted above, we may write $\psi^{(i)}$ as a product $\psi^{(i)}_2 \cdot \psi^{(i)}_{m'}$, where $\psi^{(i)}_{m'}$ factors through the determinant map.  Thus, the condition $\psi^{(i)}(g^{(i)}) = 1$ is equivalent to $\psi^{(i)}_2(g^{(i)}) = \psi^{(i)}_{m'}(\det(g^{(i)}))$.  We will split the sum in \eqref{nnewmaingoal} according to the common value of $\psi^{(i)}_2$ and $\psi^{(i)}_{m'}$, i.e. we consider the map 
\[
{\boldsymbol{\psi}}_{m'} : (\bZ/m\bZ)^\times \longrightarrow \{ \pm 1 \}^2,
\]
and, for fixed ${\mathbf{t}} \in (\bZ/m\bZ)^2$, we write the inner sum in \eqref{nnewmaingoal} as
\begin{equation} \label{nnnewmaingoal}
\sum_{{\mathbf{v}} \in \im {\boldsymbol{\psi}}_{m'}} 
\sum_{d \in (\bZ/m\bZ)^\times \atop {\boldsymbol{\psi}}_{m'}(d) = {\mathbf{v}}} 
\left| \GL_2(\bZ/m\bZ)_{\psi^{(1)}_2 = \psi^{(1)}_{m'} = v^{(1)} \atop \det \equiv d}^{\tr \equiv t^{(1)}} \right| \cdot \left| \GL_2(\bZ/m\bZ)_{\psi^{(2)}_2 = \psi^{(2)}_{m'} = v^{(2)} \atop \det \equiv d}^{\tr \equiv t^{(2)}} \right|.
\end{equation}
Since
$
\psi^{(i)}_{m'}(\cdot) = \left( \frac{\cdot}{|D_i'|} \right)
$
is simply the Jacobi symbol, we see that the image $\im {\boldsymbol{\psi}}_{m'} \subseteq \{ \pm 1 \}^2$ satisfies
\[
\im {\boldsymbol{\psi}}_{m'} = 
\begin{cases}
\{ (1,1), (-1,-1) \} & \text{ if } |D_1'| = |D_2'| \\
\{ \pm 1 \}^2 & \text{ otherwise.}
\end{cases}
\]
We now wish to further analyze each factor $(i = 1, 2)$ in each summand in \eqref{nnewmaingoal}. We may write $\psi^{(i)}_{m'}(\cdot)=\prod_{\ell\mid m'}\psi^{(i)}_{\ell}(\cdot)$ where, for an odd prime $\ell$, we set
\begin{equation} \label{identifyingpsisubell}
\psi^{(i)}_{\ell}(\cdot) := 
\begin{cases}
\left( \frac{\cdot}{\ell} \right) & \text{ if } \ell \mid D_i' \\
1 & \text{ if } \ell \nmid D_i'.
\end{cases}
\end{equation}
Consider now the Chinese Remainder Isomorphism
\begin{equation} \label{CRT}
\GL_2(\bZ/m\bZ)_{\det \equiv d}^{\tr \equiv t^{(i)}} \; \longleftrightarrow \; \GL_2(\bZ/2^\ga\bZ)_{\det \equiv d_2}^{\tr \equiv t^{(i)}_2} \times \prod_{\ell^{\ga_{\ell}} \parallel m'} \GL_2(\bZ/\ell^{\ga_{\ell}}\bZ)_{\det \equiv d_\ell}^{\tr \equiv t^{(i)}_{\ell}},
\end{equation}
where $t^{(i)}_{\ell}$, $d_{\ell}$ (resp. $t^{(i)}_2$,$d_{2}$) denote the reductions of $t^{(i)}$, $d$ modulo $\ell^{\ga_{\ell}}$ (resp. modulo $2^\ga$). Since for any prime $\ell$, the characters $\psi^{(i)}_\ell$ factor through the $\ell$-primary part in the decomposition \eqref{CRT}, we obtain with notations as above the factorization
\[
\psi^{(i)}(g^{(i)}) = \psi^{(i)}_2(g^{(i)}_2) \prod_{\ell \mid m'} \psi^{(i)}_\ell(\det(g^{(i)}_\ell)).
\]

(We have isolated the prime $2$ in \eqref{CRT} because it is the only prime $\ell$ for which $\psi^{(i)}_\ell$ does not factor through the determinant map.)  Thus, for a fixed ${\mathbf{v}} = (v^{(1)},v^{(2)}) \in \im {\boldsymbol{\psi}}_{m'}$, the inner sum of \eqref{nnnewmaingoal} becomes
\begin{equation} \label{nnnnewmaingoal}
\left( \sum_{d_2 \in (\bZ/2^{\ga}\bZ)^\times} f_{{\boldsymbol{\psi}}_2,{\mathbf{v}}}({\mathbf{t}}_2,d_2) \right)
\left( 
\sum_{{\begin{substack} { (v_\ell^{(1)}), (v_\ell^{(2)}) \in \{ \pm 1 \}^{\{ \ell \mid m' \}} \\ \prod_{\ell \mid D_i'} v_\ell^{(i)} = v^{(i)}} \end{substack}}}
\sum_{{\begin{substack} { (d_\ell) \in \prod_{\ell \mid m'} (\bZ/\ell^{\ga_{\ell}}\bZ)^\times \\ \psi^{(i)}_{\ell}(d_{\ell}) = v^{(i)}_{\ell}} \end{substack}}}
\prod_{\ell \mid m'} f_{\ell}({\mathbf{t}}_\ell,d_\ell)
\right),
\end{equation}
where we have introduced the notation
\begin{equation} \label{defoffsub2}
f_{{\boldsymbol{\psi}}_2,{\mathbf{v}}}({\mathbf{t}}_2,d_2) := \prod_{i = 1}^2 \left| \GL_2(\bZ/2^\ga \bZ)_{\psi^{(i)}_2 = v^{(i)} \atop \det \equiv d_2}^{\tr \equiv t^{(i)}_2} \right| \quad \text{ and } \quad 
f_{\ell}({\mathbf{t}}_\ell,d_\ell) := \prod_{i = 1}^2 \left| \GL_2(\bZ/\ell^{\ga_{\ell}}\bZ)_{\det \equiv d_{\ell}}^{\tr \equiv t_{\ell}^{(i)}} \right|.
\end{equation}
Noting the condition 
\begin{equation} \label{conditiononi}
\psi^{(i)}_{\ell}(d_{\ell}) = v_{\ell}^{(i)} \quad\quad (i = 1, 2)
\end{equation}
in the inner sum above and considering \eqref{identifyingpsisubell}, we see that this sum becomes an empty sum exactly when either some prime $\ell$ does not divide $D_i'$ and $v^{(i)}_{\ell} = -1$, or when some prime $\ell$ divides both $D_1$ and $D_2$ and $v^{(1)}_\ell \neq v^{(2)}_\ell$.  Thus, we may impose the conditions
\begin{equation} \label{vcondition}
\begin{split}
v^{(i)}_{\ell} &= 1 \; \text{ whenever } \; \ell \nmid D_i' \\
v^{(1)}_\ell &= v^{(2)}_\ell \; \text{ whenever } \; \ell \mid \gcd(D_1',D_2')
\end{split}
\end{equation}
in the outer sum without affecting the expression.
Further noting that
\[
\begin{cases}
\psi^{(1)}_{\ell}(\cdot) = \psi^{(2)}_{\ell}(\cdot) = \left( \frac{\cdot}{\ell} \right) & \text{ if } \ell \mid D_1' \text{ and } \ell \mid D_2' \\
\psi^{(1)}_{\ell}(\cdot) = 1 \text{ and } \psi^{(2)}_{\ell}(\cdot) = \left( \frac{\cdot}{\ell} \right) & \text{ if } \ell \nmid D_1' \text{ and } \ell \mid D_2' \\
\psi^{(1)}_{\ell}(\cdot) = \left( \frac{\cdot}{\ell} \right) \text{ and } \psi^{(2)}_{\ell}(\cdot) = 1 & \text{ if } \ell \mid D_1' \text{ and } \ell \nmid D_2',
\end{cases}
\]
we see that, for each $\ell \mid m'$, the conditions \eqref{conditiononi} and \eqref{vcondition} are equivalent to 
\begin{equation*} 
\psi^{(i_{\ell})}_{\ell}(d_{\ell}) = v_{\ell}^{(i_{\ell})},
\end{equation*}
where $i_{\ell} \in \{1, 2\}$ is any number for which $\ell \mid D_{i_\ell}'$ (in the ambiguous case that $\ell \mid D_i'$ for each $i$, one may take either $1$ or $2$ for $i_\ell$ since the value of $\psi^{(i_{\ell})}_{\ell}$ is well-defined).  We may thus replace the outer sum in \eqref{nnnnewmaingoal} with a sum over $(v_\ell) \in \{ \pm 1 \}^{\{ \ell \mid m' \}}$, obtaining that \eqref{nnnnewmaingoal} is equal to
\begin{equation} \label{nnnnnewmaingoal}
\left( \sum_{d_2 \in (\bZ/2^{\ga}\bZ)^\times} f_{{\boldsymbol{\psi}}_2,{\mathbf{v}}}({\mathbf{t}}_2,d_2) \right)
\sum_{{\begin{substack} {(v_\ell) \in \{ \pm 1 \}^{\{ \ell \mid m' \}} \\ \prod_{\ell \mid D_i'} v_\ell = v^{(i)}} \end{substack}}}
\sum_{(d_\ell) \in \prod_{\ell \mid m'} (\bZ/\ell^{\ga_{\ell}}\bZ)^\times \atop \psi_\ell(d_{\ell}) = v_{\ell}} 
\prod_{\ell \mid m'} f_{\ell}({\mathbf{t}}_\ell,d_\ell),
\end{equation}
where now $\psi_\ell$ simply denotes the Legendre symbol.  We insert this back into \eqref{nnnewmaingoal} and then further into \eqref{nnewmaingoal}, obtaining that \eqref{nnewmaingoal} is equal to
\begin{equation} \label{bigequation}
\begin{split}
&\sum_{{\begin{substack} {{\mathbf{t}} \in (\bZ/m\bZ)^2 \\ t^{(1)} + t^{(2)} = T }\end{substack}}} 
\sum_{{\mathbf{v}} \in \im {\boldsymbol{\psi}}_{m'}}
\left( \sum_{d_2 \in (\bZ/2^{\ga}\bZ)^\times} f_{{\boldsymbol{\psi}}_2,{\mathbf{v}}}({\mathbf{t}}_2,d_2) \right)
\sum_{{\begin{substack} {(v_\ell) \in \{ \pm 1 \}^{\{ \ell \mid m' \}} \\ \prod_{\ell \mid D_i'} v_\ell = v^{(i)}} \end{substack}}} 
\sum_{(d_\ell) \in \prod_{\ell \mid m'}(\bZ/\ell^{\ga_{\ell}}\bZ)^\times \atop \psi_\ell(d_{\ell}) = v_{\ell}} 
\prod_{\ell \mid m'} f_\ell({\mathbf{t}}_\ell,d_\ell) \\
=  
&\sum_{{\mathbf{v}} \in \im {\boldsymbol{\psi}}_{m'}} 
\left( \sum_{{\begin{substack} { {\mathbf{t}}_2 \in (\bZ/2^{\ga}\bZ)^2 \\ t^{(1)}_2 + t^{(2)}_2 = T_2 \\ d_2 \in (\bZ/2^{\ga}\bZ)^\times } \end{substack}}} f_{{\boldsymbol{\psi}}_2,{\mathbf{v}}}({\mathbf{t}}_2,d_2) \right)
\sum_{{\begin{substack} {(v_\ell) \in \{ \pm 1 \}^{\{ \ell \mid m' \}} \\ \prod_{\ell \mid D_i'} v_\ell = v^{(i)}} \end{substack}}} 
\prod_{\ell \mid m'} 
\sum_{{\begin{substack} {{\mathbf{t}}_\ell \in (\bZ/\ell^{\ga_{\ell}}\bZ)^2 \\ t^{(1)}_\ell + t^{(2)}_\ell = T_\ell \\ d_\ell \in (\bZ/\ell^{\ga_{\ell}}\bZ)^\times \\ \psi_\ell(d_\ell) = v_\ell} \end{substack}}}  f_\ell({\mathbf{t}}_\ell,d_\ell).
\end{split}
\end{equation}
The following lemma aids in evaluating the above summand.  
We recall the following notation:
\begin{equation} \label{defofAandB}
\begin{split}
e_\ell(T_\ell) &:= \lfloor (v_\ell(T_\ell)-1)/2 \rfloor, \quad\quad \gd_n(m) :=
\begin{cases}
1 & \text{ if } n \text{ divides } m \\
0 & \text{ otherwise,}
\end{cases} \\
E(\ell,T_\ell) &:= \frac{\ell^3(\ell+1)(\ell^{4e_\ell(T_\ell)+4}-1) + (\ell^4 + \ell^3 + 2\ell^2 + \ell - 1)(\ell^{4e_\ell(T_\ell)}-1)}{(\ell^2+1)(\ell+1)}, \\
A_\ell(T_\ell) &:= 
\begin{cases}
 \ell^5 - \ell^3 - 3\ell^2 + 1 + \frac{2E(\ell,T_\ell)}{\ell^{4e_\ell(T_\ell)+3}} + \gd_2(v_\ell(T_\ell)) \frac{(\ell^3-1)}{\ell^{4e_\ell(T_\ell)+2}} & \text{ if $\ell \mid T_\ell$} \\ 
(\ell^6 - \ell^5 - 2\ell^4 + \ell^3 + 2\ell^2)/(\ell-1) & \text{ if $\ell \nmid T_\ell$,}
\end{cases} \\
B_\ell(T_\ell) &:= 
\begin{cases}
-3\ell^2 + 1 + \frac{2E(\ell,T_\ell)}{\ell^{4e_\ell(T_\ell)+3}} + \gd_2(v_\ell(T_\ell)) \frac{(\ell^3-1)}{\ell^{4e_\ell(T_\ell)+2}} & \text{ if $\ell \mid T_\ell$} \\
(\ell^2)/(\ell-1) & \text{ if $\ell \nmid T_\ell$}.
\end{cases}
\end{split}
\end{equation}
\begin{Lemma}
Let $\ell$ be an odd prime, let ${\mathbf{t}}_\ell \in (\bZ/\ell^{\ga_{\ell}}\bZ)^2$ and $d_\ell \in (\bZ/\ell^{\ga_{\ell}}\bZ)^\times$, and let $f_\ell({\mathbf{t}}_\ell,d_\ell)$ be as defined in \eqref{defoffsub2}.  Assume that $\alpha_\ell>v_\ell(T_\ell)$. We have 
\[
\sum_{{\begin{substack} {{\mathbf{t}}_\ell \in (\bZ/\ell^{\ga_{\ell}}\bZ)^2 \\ t^{(1)}_\ell + t^{(2)}_\ell = T_\ell \\ d_\ell \in (\bZ/\ell^{\ga_{\ell}}\bZ)^\times \\ \psi_\ell(d_\ell) = \nu_\ell} \end{substack}}}  f_\ell({\mathbf{t}}_\ell,d_\ell) = \frac{\ell^{5\ga_{\ell}-5} \varphi(\ell^{\ga_\ell})}{2} \left[ A_\ell(T_\ell) + \nu_\ell B_\ell(T_\ell) \right],
\]
where $\varphi$ denotes the Euler phi function and $A_\ell(T_\ell)$ and $B_\ell(T_\ell)$ are as in \ref{defofAandB}. 
\end{Lemma} 
\begin{proof}
We let $F(-1,T_\ell)$ denote the sum above when $\nu_\ell=-1$ and put $n=\alpha_\ell$. It suffices to compute $F(-1,T_\ell)$, since the result may then be deduced by setting $A_\ell(T_\ell)=\frac{\vert G(\ell^n,T_\ell)\vert}{2}\cdot \frac{2}{(\ell-1)\ell^{6n-6}}$ and $B_\ell(T_\ell)=A_\ell(T_\ell)-\frac{2 F(-1,T_\ell)}{(\ell-1)\ell^{6n-6}}$. The computation of $\vert G(\ell^n,T_\ell)\vert$ can then be read off from Theorem \ref{primesdividingtrace} when $\ell\vert T_\ell$ and Lemma \ref{good primes} otherwise. \\

To compute $F(-1,T_\ell)$, observe that if $d_\ell \mod \ell$ is a non-square the discriminant $t_\ell^2-4d_\ell \mod \ell$ never vanishes and therefore we are counting only matrices which are not congruent to scalars modulo $\ell$. Hence applying Lemma \ref{singlematrix} with $i=0$ we obtain:  
\begin{equation}\label{eqgenT}
F(-1,T_\ell)=\sum_{t_\ell^{(1)}+t_\ell^{(2)}= T_\ell}\sum_{{\begin{substack}{d_\ell\in(\mathbb{Z}/\ell^{n}\mathbb{Z})^\times\\
\psi(d_\ell)=-1}\end{substack}}}\ell^{4n-2}\left(\ell+\left(\frac{{(t_\ell^{(1)})}^2-4d_\ell}{\ell}\right)\right)\cdot\left(\ell+\left(\frac{{(t_\ell^{(2)})}^2-4d_\ell}{\ell}\right)\right).
\end{equation}
We have for any $T_\ell$ the equality 
\begin{equation}\label{eqsum}
\sum_{t_\ell^{(1)}\in(\mathbb{Z}/\ell\mathbb{Z})}\sum_{{\begin{substack}{d_\ell\in(\mathbb{Z}/\ell^{n}\mathbb{Z})^\times\\
\psi(d_\ell)=-1}\end{substack}}}\left(\frac{(t_\ell^{(1)})^2-4d_\ell}{\ell}\right)+\left(\frac{(T_\ell-t_\ell^{(1)})^2-4d_\ell}{\ell}\right)=-(\ell-1)
\end{equation}

by exchanging the summation order and using from Lemma \ref{lemma:legendresum} that for any $T\in\bZ$ and prime  $\ell\nmid 2d_\ell$:
$$\sum_{t_\ell\in(\mathbb{Z}/\ell\mathbb{Z})}\left(\frac{(T-t_\ell)^2-4d_\ell}{\ell}\right)=-1.
$$
To compute the quantity in equation \ref{eqgenT}, it remains to evaluate the expression 
$$P(T_\ell,\ell):=\sum_{t_\ell^{(1)}\in(\mathbb{Z}/\ell\mathbb{Z})}\sum_{{\begin{substack}{d_\ell\in(\mathbb{Z}/\ell^{n}\mathbb{Z})^\times\\
\psi(d_\ell)=-1}\end{substack}}}\left(\frac{(t_\ell^{(1)})^2-4d_\ell}{\ell}\right)\cdot\left(\frac{{(T_\ell-t_\ell^{(1)})}^2-4d_\ell}{\ell}\right).
$$
Observe that for any unit $u\in(\mathbb{Z}/\ell\mathbb{Z})^\times$, we have
$P(u\cdot T_\ell,\ell)=P(T_\ell,\ell)$ provided $\ell\nmid T_\ell$, whereas we trivially obtain $P(0,\ell)=\ell(\ell-1)/2$. Moreover, we may write 
\begin{align*}
\sum_{T_\ell=0}^{\ell-1}P(T_\ell,\ell)&=\sum_{{\begin{substack}{d_\ell\in(\mathbb{Z}/\ell^{n}\mathbb{Z})^\times\\
\psi(d_\ell)=-1}\end{substack}}}\left(\sum_{t_\ell\in(\mathbb{Z}/\ell\mathbb{Z})}\left(\frac{t_\ell^2-4d_\ell}{\ell}\right)\right)^2\\
&=\sum_{{\begin{substack}{d_\ell\in(\mathbb{Z}/\ell^{n}\mathbb{Z})^\times\\
\psi(d_\ell)=-1}\end{substack}}}\left(-1\right)^2\\
&=(\ell-1)/2.
    \end{align*}
We deduce therefore that 
\begin{equation}\label{eqpl}
    P(T_\ell,\ell)=\begin{cases}
    \ell(\ell-1)/2&\text{ if }\ell\vert T_\ell\\
    -(\ell-1)/2&\text{ else. }
        \end{cases}
\end{equation}

Putting everything together, we obtain from equations \ref{eqgenT}, \ref{eqsum} and \ref{eqpl} the count: 
$$F(-1,T_\ell)=\ell^{6n-4}\cdot(\ell-1)/2\cdot
\begin{cases}
\ell^3-\ell&\text{ if }\ell\vert T_\ell\\
\ell^3-2\ell-1&\text{ else.}
\end{cases}$$
The result now follows by simple calculation from our previous considerations. 
\end{proof}

Inserting the result of the lemma into \eqref{bigequation}, we obtain
\begin{equation} \label{nbigequation}
\left| G(m,T)_{{\boldsymbol{\psi}} = 1} \right| = 
\frac{(m')^5\varphi(m')}{2^{\om(m')}\left( \rad(m') \right)^5} \sum_{{\mathbf{v}} \in \im {\boldsymbol{\psi}}_{m'}} 
S_{{\boldsymbol{\psi}}_2}(T_2,{\mathbf{v}})
\sum_{{\begin{substack} {(v_\ell) \in \{ \pm 1 \}^{\{ \ell \mid m' \}} \\ \prod_{\ell \mid D_i'} v_\ell = v^{(i)}} \end{substack}}} 
\prod_{\ell \mid m'} \left( A_\ell(T_\ell) + v_\ell B_\ell(T_\ell) \right),
\end{equation}
where
\[
S_{{\boldsymbol{\psi}}_2}(T_2,{\mathbf{v}}) :=  \sum_{{\begin{substack} { {\mathbf{t}}_2 \in (\bZ/2^\ga\bZ)^2 \\ t^{(1)}_2 + t^{(2)}_2 = T_2 \\ d_2 \in (\bZ/2^\ga\bZ)^\times } \end{substack}}} f_{{\boldsymbol{\psi}}_2,{\mathbf{v}}}({\mathbf{t}}_2,d_2).
\]
Formally expanding out the innermost product of binomials and reversing summation, we observe that
\[
\sum_{{\begin{substack} {(v_\ell) \in \{ \pm 1 \}^{\{ \ell \mid m' \}} \\ \prod_{\ell \mid D_i'} v_\ell = v^{(i)}} \end{substack}}} 
\prod_{\ell \mid m'} 
(A_\ell + v_\ell B_\ell)
=
\sum_{S \subset \{ \ell \mid m' \}}
\sum_{{\begin{substack} {(v_\ell) \in \{ \pm 1 \}^{\{ \ell \mid m' \}} \\ \prod_{\ell \mid D_i'} v_\ell = v^{(i)}} \end{substack}}}
\left( \prod_{\ell \in S^c} A_\ell + \prod_{\ell \in S} v_\ell \prod_{\ell \in S} B_\ell \right).
\]
In the sum over $(v_\ell)$ on the second term $\ds \prod_{\ell \in S} v_\ell \prod_{\ell \in S} B_\ell$, we have perfect cancellation unless $S = \{ \ell \mid D_1' \}$ or $S = \{ \ell \mid D_2' \}$.  Thus we have
\[
\sum_{{\begin{substack} {(v_\ell) \in \{ \pm 1 \}^{\{ \ell \mid m' \}} \\ \prod_{\ell \mid D_i} v_\ell = v^{(i)}} \end{substack}}} 
\prod_{\ell \mid m'} 
(A_\ell + v_\ell B_\ell)
=
\begin{cases}
\ds 2^{\omega(m') - 2} \left( \prod_{\ell \mid m'} \left( A_\ell + 1 \right) +
v^{(1)} \prod_{\ell \mid D_1'} B_\ell + 
v^{(2)} \prod_{\ell \mid D_2'} B_\ell \right) & \text{ if } |D_1'| \neq |D_2'| \\
\ds 2^{\omega(m') - 1} \left( \prod_{\ell \mid m'} \left( A_\ell + 1 \right) +
v^{(1)} \prod_{\ell \mid D_1'} B_\ell \right) & \text{ if } |D_1'| = |D_2'|.
\end{cases}
\]
Inserting this back into \eqref{nbigequation}, we find that our original expression \eqref{nnewmaingoal} is equal to
\begin{equation} \label{nnnnnnnewmaingoal}
\begin{cases}
\ds \frac{(m')^5\varphi(m')}{4 \left( \rad(m') \right)^5} 
\left[
\mathfrak{S}_{{\boldsymbol{\psi}}_2}(T_2) \prod_{\ell \mid m'} (A_\ell + 1) 
+  
\mathfrak{S}_{{\boldsymbol{\psi}}_2}^{(1)}(T_2) \prod_{\ell \mid D_1'} B_\ell 
 + 
\mathfrak{S}_{{\boldsymbol{\psi}}_2}^{(2)}(T_2) \prod_{\ell \mid D_2'} B_\ell 
\right] & \text{ if } |D_1'| \neq |D_2'| \\
\ds \frac{(m')^5\varphi(m')}{2 \left( \rad(m') \right)^5}
\left[
\mathfrak{S}_{{\boldsymbol{\psi}}_2}(T_2) \prod_{\ell \mid m'} (A_\ell + 1) 
+  
\mathfrak{S}_{{\boldsymbol{\psi}}_2}^{(1)}(T_2) \prod_{\ell \mid D_1'} B_\ell 
\right] & \text{ if } |D_1'| = |D_2'|,
\end{cases}
\end{equation}
where we have used the fact that
\[
\begin{split}
\mathfrak{S}_{{\boldsymbol{\psi}}_2}(T_2) &:=  \sum_{{\mathbf{v}} \in \im {\boldsymbol{\psi}}_{m'}} S_{{\boldsymbol{\psi}}_2}(T_2,{\mathbf{v}}), \\
\mathfrak{S}_{{\boldsymbol{\psi}}_2}^{(i)}(T_2) &:= \sum_{{\mathbf{v}} \in \im {\boldsymbol{\psi}}_{m'}} v^{(i)} S_{{\boldsymbol{\psi}}_2}(T_2,{\mathbf{v}}).
\end{split}
\]
This concludes the proof of the proposition.
\end{proof}
We now establish formulas which lead to an explicit evaluation of the terms $\mathfrak{S}_{{\boldsymbol{\psi}}_2}(T_2)$ and $\mathfrak{S}_{{\boldsymbol{\psi}}_2}^{(i)}(T_2)$ above. To that end, recall that we have set $$\psi_2^{(i)}(\cdot)=\varepsilon(\cdot)\left(\frac{D_i}{\det(\cdot)}\right)\cdot \left(\frac{\det(\cdot)}{\vert D_i'\vert}\right).$$ 
By quadratic reciprocity, this yields the following cases for $\psi_2^{(i)}$ when $i=1,2$: 
\begin{equation}\label{psiprimetwo}
\psi_2^{(i)}(g)=
\begin{cases}
\varepsilon(g) & \text{ if }D_i\equiv 1 \mod 4\\
\varepsilon(g) \cdot \chi_4(\det(g)) & \text{ if }D_i=4D_i'\text{ and }D_i'\equiv 3\mod 4\\
\varepsilon(g) \cdot\chi_8(\det(g)) & \text{ if }D_i=4D_i'\text{ and }D_i'\equiv 2 \mod 8\\
\varepsilon(g) \cdot\chi_4\chi_8(\det(g)) & \text{ if }D_i=4D_i'\text{ and }D_i'\equiv 6 \mod 8,\\
\end{cases}
\end{equation}
where the characters $\chi_4$ and $\chi_8$ are the characters modulo $4$ and $8$ respectively defined by $\chi_4(x)=\left(\frac{-1}{x}\right)$ and $\chi_8(x)=\left(\frac{2}{x}\right)$ for odd integers $x$. In the following lemma we also set
\begin{equation} \label{defofentd}
e(n,t,d) := \min \left\{ \left\lceil \frac{n}{2} \right\rceil, \left\lfloor \frac{v_2(t^2-4d)}{2} \right\rfloor \right\}.
\end{equation}
Furthermore, for any set of formulas $S$, we abbreviate by $\gd(S)$ the truth indicator of $S$:
\[
\gd(S) :=
\begin{cases}
1 & \text{ if every formula in $S$ is true} \\
0 & \text{ otherwise}.
\end{cases}
\]
Finally, we will employ the abbreviations
\begin{equation} \label{defofgDandv}
\gD = \gD(t,d) := t^2-4d, \quad\quad v = v(t,d,n) := 
\begin{cases}
v_2(\gD(t,d)) & \text{ if } \gD(t,d) \neq 0 \\
n + 2 & \text{ if } \gD(t,d) = 0.
\end{cases}
\end{equation}
\begin{Lemma} \label{totalfixedtraceanddetcountlemma}
Let $t, d \in \bZ$ with $d$ odd, let $\gD \in \bZ$ and $v \in \bZ_{\geq 0}$ be defined by \eqref{defofgDandv} and let $e(n,t,d)$ be defined by \eqref{defofentd}.  The quantity $\ds | \GL_2(\bZ/2^n\bZ)^{\tr \equiv t}_{\det \equiv d} |$ is equal to
\[
\begin{split}
&2^{2n-1} 3 \left(1 - 2^{-e(n,t,d)} \right) \; + \; \gd\left( \begin{matrix} v < n, \, v \in 2\bZ, \\ \gD/2^v \equiv 0,1 \hspace{-.05in} \mod 4 \end{matrix} \right) 2^{2n - 2 - v/2} \left(4 + 2 \chi_8\left( \gD/2^v \right) \right) \\ 
+ \; &\gd \left( \begin{matrix} v \geq n, \, n \in 2\bZ, \\ \gD/2^n \equiv 0, 1 \hspace{-.05in} \mod 4 \end{matrix} \right) 2^{3n/2-2}  \left(3 + \chi_4\left( \gD/2^n \right) \right) \; + \; 
\gd \left( \begin{matrix} v \geq n+1, \\ n \notin 2\bZ \end{matrix} \right) 2^{3(n-3)/2} \left( 7 - (-1)^{\gD/2^{n+1}}\right) \\
+ \; &\gd \left( v \geq n + 2 \right) 2^{3n - 3 \lceil n/2 \rceil - 2}.
\end{split}
\]
\end{Lemma}
\begin{proof}
This follows from Proposition \ref{primetwocount}.
\end{proof}
The next lemma is the result of a tedious computation, making repeated use of Proposition \ref{primetwocount}.
\begin{Lemma} \label{tediouslemma}
Distinguishing cases as in \ref{psiprimetwo} we obtain the following counts at the prime $2$:\\
If $\psi_2 = \ve$ and $n \geq 1$, then we have
\[
\left| \GL_2(\bZ/2^n\bZ)^{\tr \equiv t}_{{\begin{substack} {\det \equiv d \\ \psi_2 = -1} \end{substack}}} \right| = 
\begin{cases}
3 \cdot 4^{n-1} & \text{ if $t$ is even } \\
0 & \text{ if $t$ is odd.}
\end{cases}
\]
If $\psi_2 = \ve \chi_4$ and $n \geq 3$, then we have
\[
\left| \GL_2(\bZ/2^n\bZ)^{\tr \equiv t}_{{\begin{substack} {\det \equiv d \\ \psi_2 = -1} \end{substack}}} \right| = 
\begin{cases}
3 \cdot 4^{n-1} & \text{ if $t$ is even and $\chi_4(d) = 1$ } \\
2 \cdot 4^{n-1} & \text{ if $t$ is odd and $\chi_4(d) =  -1$ } \\
3 \cdot 4^{n-1} & \text{ if $t \equiv d+1 \mod 8$ and $\chi_4(d) = -1$ } \\
4^{n-1} & \text{ if $t \equiv d + 5 \mod 8$ and $\chi_4(d) = -1$ } \\
0 & \text{ otherwise.}
\end{cases}
\]
If $\psi_2 = \ve \chi_8$ and $n \geq 5$, then we have
\[
\left| \GL_2(\bZ/2^n\bZ)^{\tr \equiv t}_{{\begin{substack} {\det \equiv d \\ \psi_2 = -1} \end{substack}}} \right| = 
\begin{cases}
3 \cdot 4^{n-1} & \text{ if $t$ is even and $\chi_8(d) = 1$ } \\
2 \cdot 4^{n-1} & \text{ if $t$ is odd and $\chi_8(d) = -1$ } \\
3 \cdot 4^{n-1} & \text{ if $t \equiv 4 \mod 8$ and $d \equiv 3 \mod 8$ } \\
1 \cdot 4^{n-1} & \text{ if $t \equiv 0 \mod 8$ and $d \equiv 3 \mod 8$ } \\
6 \cdot 4^{n-2} & \text{ if $t \equiv 2 \mod 8$ and $d \equiv t + 3 \mod 16$ } \\
6 \cdot 4^{n-2} & \text{ if $t \equiv 6 \mod 8$ and $d \equiv t + 7 \mod 16$ } \\
\left[ 10 + 2 \chi_8\left( \frac{\gD}{16} \right) \right] \cdot 4^{n-2} & \text{ if $t \equiv 2 \mod 8$ and $d \equiv t - 5 \mod 16$ } \\
\left[ 10 + 2 \chi_8\left( \frac{\gD}{16} \right) \right] \cdot 4^{n-2} & \text{ if $t \equiv 6 \mod 8$ and $d \equiv t - 1 \mod 16$ } \\
0 & \text{ otherwise.}
\end{cases}
\]
Finally, if If $\psi_2 = \ve \chi_8 \chi_4$ and $n \geq 5$, then we have
\[
\left| \GL_2(\bZ/2^n\bZ)^{\tr \equiv t}_{{\begin{substack} {\det \equiv d \\ \psi_2 = -1} \end{substack}}} \right| = 
\begin{cases}
3 \cdot 4^{n-1} & \text{ if $t$ is even and $\chi_8(d)\chi_4(d) = 1$ } \\
2 \cdot 4^{n-1} & \text{ if $t$ is odd and $\chi_8(d)\chi_4(d) = -1$ } \\
3 \cdot 4^{n-1} & \text{ if $t \equiv 0 \mod 8$ and $d \equiv 7 \mod 8$ } \\
1 \cdot 4^{n-1} & \text{ if $t \equiv 4 \mod 8$ and $d \equiv 7 \mod 8$ } \\
6 \cdot 4^{n-2} & \text{ if $t \equiv 2 \mod 8$ and $d \equiv t + 3 \mod 16$ } \\
6 \cdot 4^{n-2} & \text{ if $t \equiv 6 \mod 8$ and $d \equiv t + 7 \mod 16$ } \\
\left[ 10 + 2 \chi_8\left( \frac{\gD}{16} \right) \right] \cdot 4^{n-2} & \text{ if $t \equiv 2 \mod 8$ and $d \equiv t - 5 \mod 16$ } \\
\left[ 10 + 2 \chi_8\left( \frac{\gD}{16} \right) \right] \cdot 4^{n-2} & \text{ if $t \equiv 6 \mod 8$ and $d \equiv t - 1 \mod 16$ } \\
0 & \text{ otherwise.}
\end{cases}
\]

\end{Lemma}

We remark that Lemmas \ref{totalfixedtraceanddetcountlemma} and \ref{tediouslemma} result in explicit evaluations of 
\[
\left| \GL_2(\bZ/2^n\bZ)^{\tr \equiv t}_{{\begin{substack} {\det \equiv d \\ \psi_2 = \pm 1} \end{substack}}} \right|.
\]
These evaluations, in turn, lead to an evaluation of the quantities $\mathfrak{S}_{{\boldsymbol{\psi}_2}}(T)$ and $\mathfrak{S}_{{\boldsymbol{\psi}_2}}^{(i)}(T)$, via the formulas
\[
\begin{split}
\mathfrak{S}_{{\boldsymbol{\psi}_2}}(T) &= \sum_{v \in \im \psi_{m'}} \sum_{{\begin{substack} { \boldsymbol{t} \in (\bZ/2^\ga \bZ)^2 \\ t^{(1)} + t^{(2)} \equiv T \mod 2^\ga \\ d \in (\bZ/2^\ga \bZ)^\times } \end{substack}}} \prod_{i = 1}^2 \left| \GL_2(\bZ/2^\ga\bZ)^{\tr \equiv t^{(i)}}_{{\begin{substack} {\det \equiv d \\ \psi_2^{(i)} = v^{(i)} } \end{substack}}} \right|, \\
\mathfrak{S}_{{\boldsymbol{\psi}_2}}(T) &= \sum_{v \in \im \psi_{m'}} v^{(i)} \sum_{{\begin{substack} { \boldsymbol{t} \in (\bZ/2^\ga \bZ)^2 \\ t^{(1)} + t^{(2)} \equiv T \mod 2^\ga \\ d \in (\bZ/2^\ga \bZ)^\times } \end{substack}}} \prod_{i = 1}^2 \left| \GL_2(\bZ/2^\ga\bZ)^{\tr \equiv t^{(i)}}_{{\begin{substack} {\det \equiv d \\ \psi_2^{(i)} = v^{(i)} } \end{substack}}} \right|.
\end{split} 
\]
Finally, by virtue of Proposition \ref{badprimesmainprop}, the evaluation of these quantities leads to an explicit evaluation of the factor at $m_{E_1\times E_2}$ (the conductor of the open subgroup $\im (\rho_{E_1\times E_2}) \subseteq G(\hat{\mbz})$) of the constant appearing in Conjecture \ref{conj:main}.  The other factors, as presented in Theorem \ref{thm:mainformulas}, have already been previously computed in Section \ref{section:explicit constants}.

\section{Numerical evidence}\label{section:numerical}

We now present the results of some computations geared towards gathering evidence for Conjecture \ref{conj:main}. The results of the previous sections provide explicit formulas for the constants appearing therein when $E_1 \times E_2$ is a Serre pair, summarized in Theorem \ref{thm:mainformulas}. Using the computational packages MAGMA and SageMath, we worked with the explicit family of Serre pairs found in \cite{MR3557121}, wherein it is shown that, for primes $\ell_1, \ell_2$ satisfying the conditions
 \begin{equation} \label{danielsetalserrepairscondition}
  \ell_1, \ell_2 \notin \{ 2, 7 \}, \quad\quad \gcd(432\ell_1^2+\ell_1, 432\ell_2^2+\ell_2)=1,
  \end{equation}
 the elliptic curves  
 \begin{equation} \label{defofEell1andEell2}
 \begin{split}
 E_{\ell_1}: &y^2+xy = x^3 + \ell_1, \\
 E_{\ell_2}: &y^2 + xy = x^3 + \ell_2
 \end{split}
 \end{equation}
form a Serre pair. We note that each $E_{\ell_i}$ has discriminant $\Delta_{\ell_i}=-432\ell_i^2-\ell_i$. 
With all of this in hand, we computed a whole range of examples for distinct primes $(\ell_1,\ell_2)$ satisfying \eqref{danielsetalserrepairscondition}. We computed the actual count
\begin{equation} \label{piactualdef}
\pi_{E_{\ell_1} \times E_{\ell_2}, T}^{\actual}(x) := \# \{ p \leq x : p \nmid N_{E_{\ell_1}} N_{E_{\ell_2}}, \, a_p(E_{\ell_1}) + a_p(E_{\ell_2}) = T \} 
\end{equation}
for $x = 10^7$ and all non-zero integers $T$ in the Hasse interval $|T| \leq 4\sqrt{x}$.  Furthermore, for all such non-zero $T$, we computed a  value closely related to the one predicted by Conjecture \ref{conj:main}.  In the spirit of the computations carried out in \cite{baierjones}, we calculated the value
\begin{equation} \label{pipreddef}
\pi_{E_{\ell_1} \times E_{\ell_2}, T}^{\pred}(x) := \frac{C(E_1 \times E_2,T)}{2 \Phi(0)} \int_{\max \left\{ 2, (T/2g)^2 \right\}}^x \Phi \left( \frac{T}{\sqrt{t}} \right) \frac{1}{2\sqrt{t} \log t} \, dt,
\end{equation}
which is asymptotic to what is predicted by Conjecture \ref{conj:main} for fixed $T$ as $x \rightarrow \infty$ but which is more accurate when $T$ is on the scale of $\sqrt{x}$ (see Remark \ref{integralexpressionremark}).  Furthermore, we computed the absolute error $\ds \mc{E}_{E_{\ell_1} \times E_{\ell_2}, T}^{\abs}(x)$, relative error $\ds \mc{E}_{E_{\ell_1} \times E_{\ell_2}, T}^{\rel}(x)$ and the ``Central Limit Theorem error'' $\ds \mc{E}_{E_{\ell_1} \times E_{\ell_2}, T}^{\CLT}(x)$, defined respectively by
\begin{equation} \label{formsoferror}
\begin{split}
\mc{E}_{E_{\ell_1} \times E_{\ell_2}, T}^{\abs}(x) &:= \pi_{E_{\ell_1} \times E_{\ell_2}, T}^{\actual}(x) - \pi_{E_{\ell_1} \times E_{\ell_2}, T}^{\pred}(x) \\
\mc{E}_{E_{\ell_1} \times E_{\ell_2}, T}^{\rel}(x) &:= \frac{\pi_{E_{\ell_1} \times E_{\ell_2}, T}^{\actual}(x) - \pi_{E_{\ell_1} \times E_{\ell_2}, T}^{\pred}(x)}{\pi_{E_{\ell_1} \times E_{\ell_2}, T}^{\pred}(x)} \\
\mc{E}_{E_{\ell_1} \times E_{\ell_2}, T}^{\CLT}(x) &:= \frac{\pi_{E_{\ell_1} \times E_{\ell_2}, T}^{\actual}(x) - \pi_{E_{\ell_1} \times E_{\ell_2}, T}^{\pred}(x)}{\sqrt{\pi_{E_{\ell_1} \times E_{\ell_2}, T}^{\pred}(x)}}.
\end{split}
\end{equation}
With all of the results in this paper, such computations should be easily reproducible and we present here a few representative examples for the Serre pairs $E_3 \times E_{11}$, $E_{31} \times E_{107}$, and $E_{79} \times E_{107}$.  For the first pair, we display data plots for the sets
\[
\begin{split}
\left\{ \left( T, \pi_{E_{\ell_1} \times E_{\ell_2}, T}^{\pred}(x) \right) : |T| \leq 4\sqrt{x} \right\}, \\ 
\left\{ \left( T, \pi_{E_{\ell_1} \times E_{\ell_2}, T}^{\actual}(x) \right) : |T| \leq 4\sqrt{x} \right\},
\end{split}
\]
and similarly for each of the error functions in \eqref{formsoferror}; for the latter two pairs, we only display such data for $\pi_{E_{\ell_1} \times E_{\ell_2}, T}^{\pred}(x)$, for $\pi_{E_{\ell_1} \times E_{\ell_2}, T}^{\actual}(x)$ and for absolute error $\mc{E}_{E_{\ell_1} \times E_{\ell_2}, T}^{\abs}(x)$ (see Figures \ref{fig:Pred311p10M} through \ref{fig:Diff79107p10M}).

\begin{figure}[H]
\centering
\includegraphics[width=12cm]{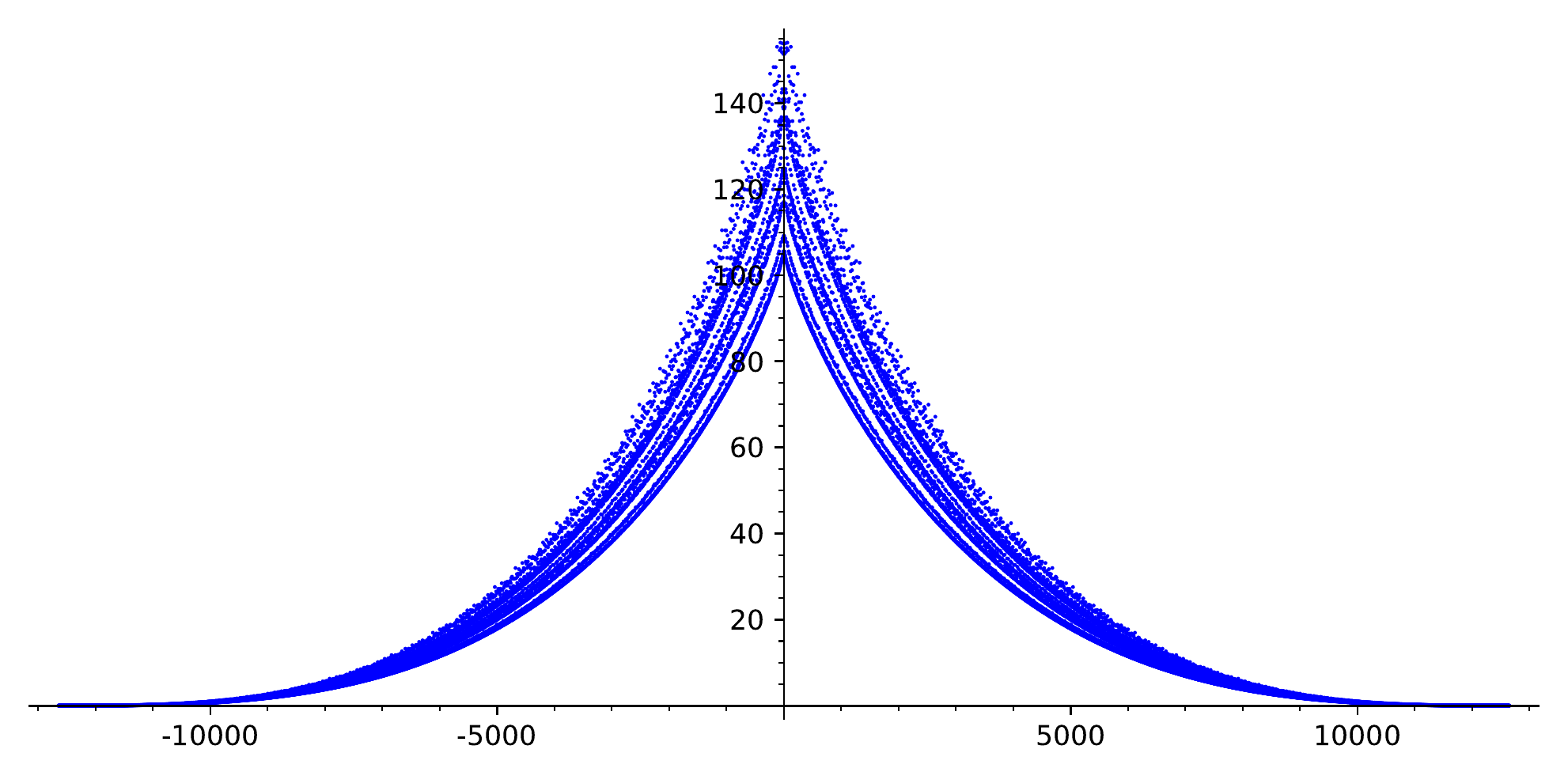}
\caption{$\pi_{E_{3} \times E_{11}, T}^{\pred}(10^7)$}
\label{fig:Pred311p10M}
\end{figure}

\begin{figure}[H]
\centering
\includegraphics[width=12cm]{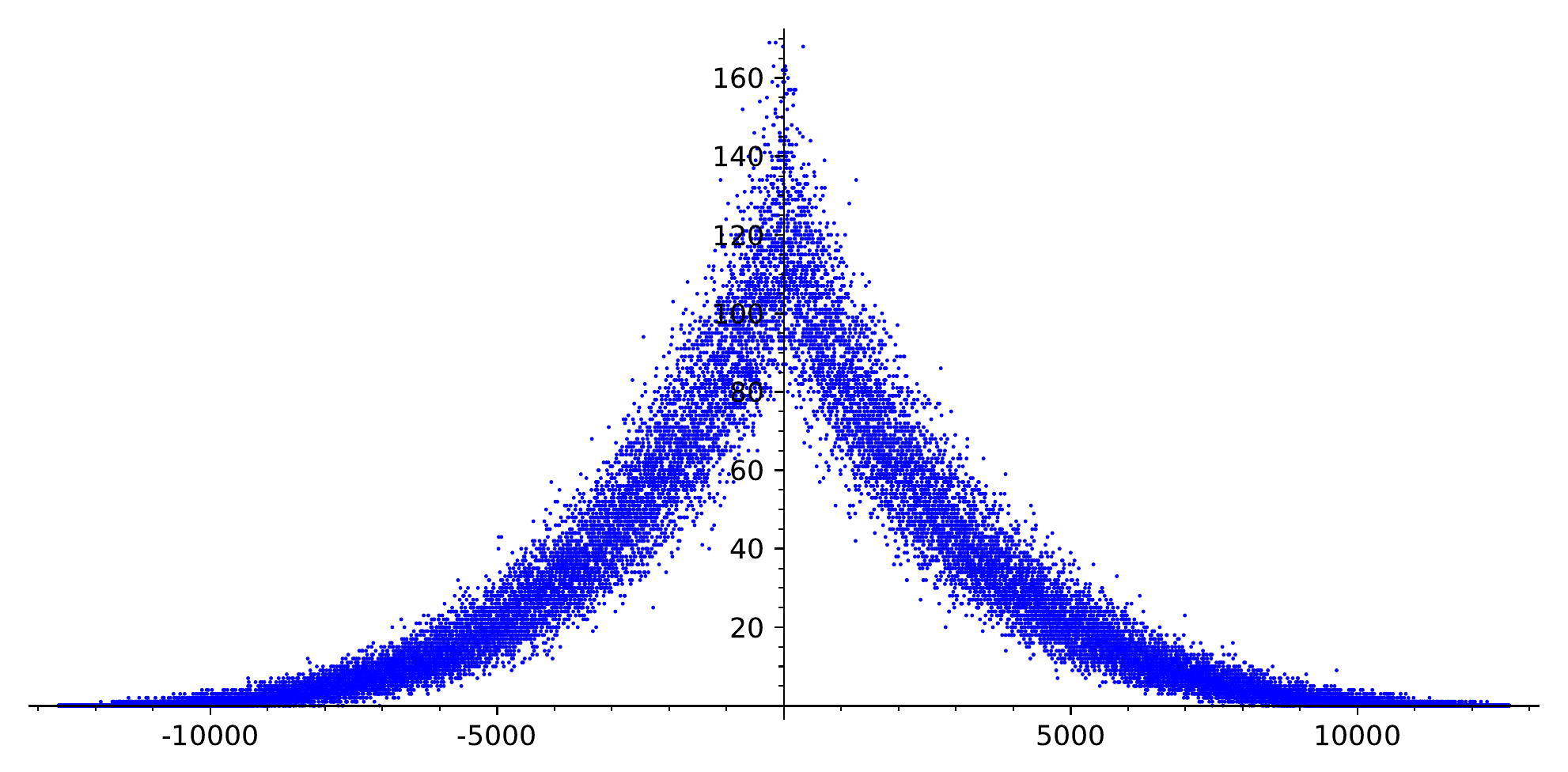}
\caption{$\pi_{E_{3} \times E_{11}, T}^{\actual}(10^7)$}
\label{fig:Actual311p10M}
\end{figure}

\begin{figure}[H]
\centering
\includegraphics[width=12cm]{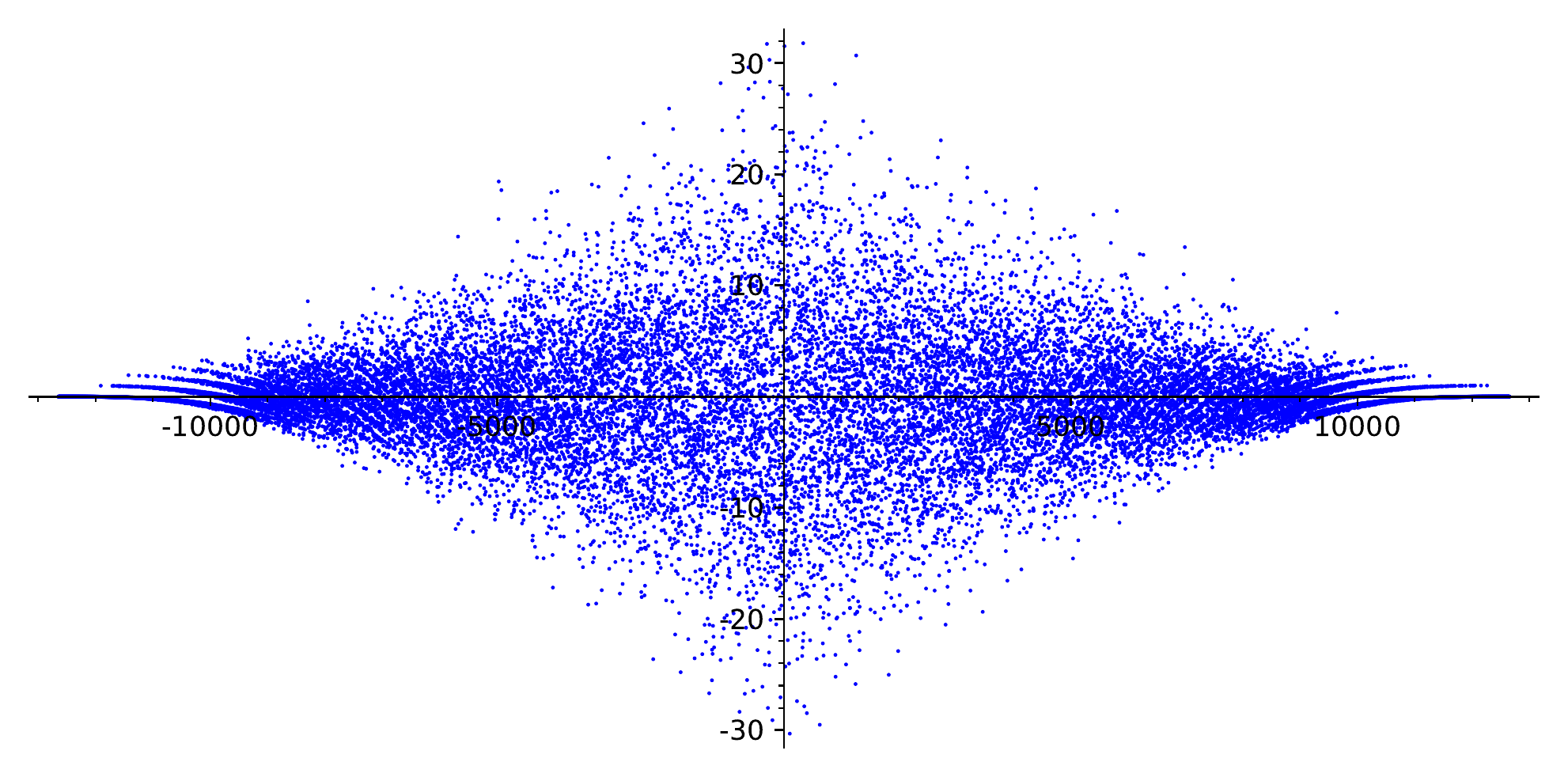}
\caption{$\mc{E}_{E_{3} \times E_{11}, T}^{\abs}(10^7)$}
\label{fig:Diff311p10M}
\end{figure}

\begin{figure}[H]
\centering
\includegraphics[width=12cm]{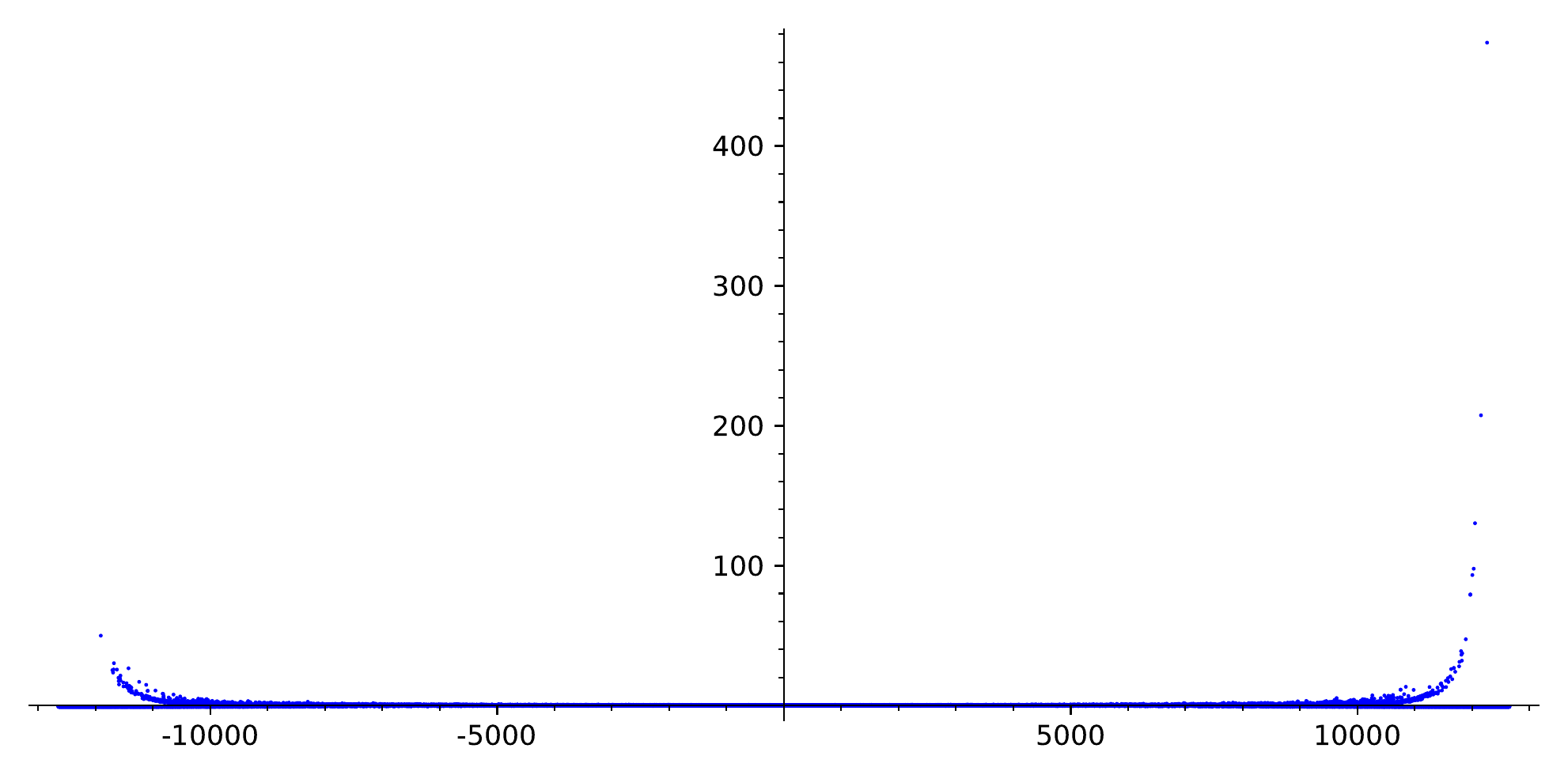}
\caption{$\mc{E}_{E_{3} \times E_{11}, T}^{\rel}(10^7)$}
\label{fig:RelDiff311p10M}
\end{figure}

\begin{figure}[H]
\centering
\includegraphics[width=12cm]{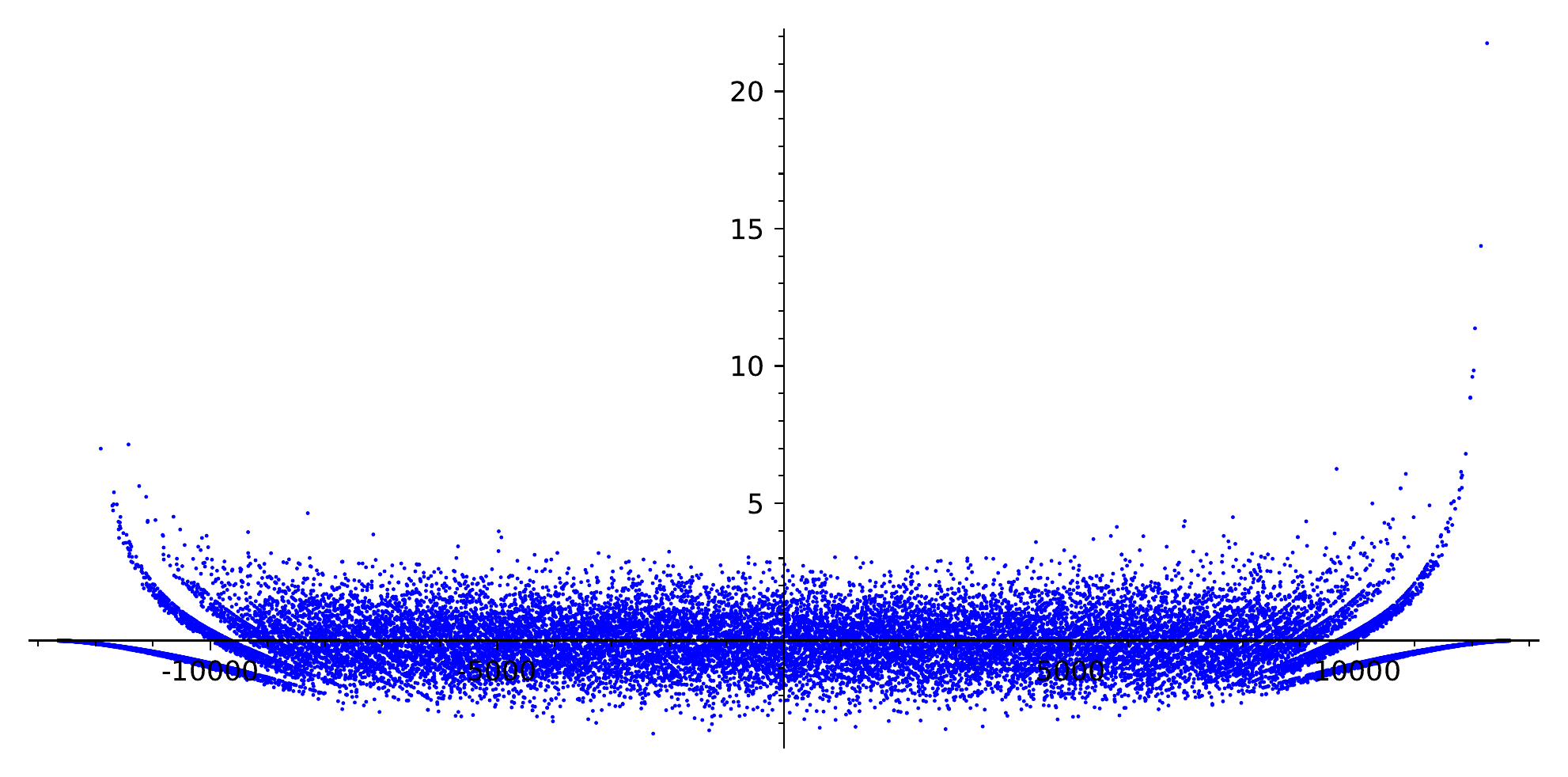}
\caption{$\mc{E}_{E_{3} \times E_{11}, T}^{\CLT}(10^7)$}
\label{fig:CLTDiff311p10M}
\end{figure}

\begin{figure}[H]
\centering
\includegraphics[width=12cm]{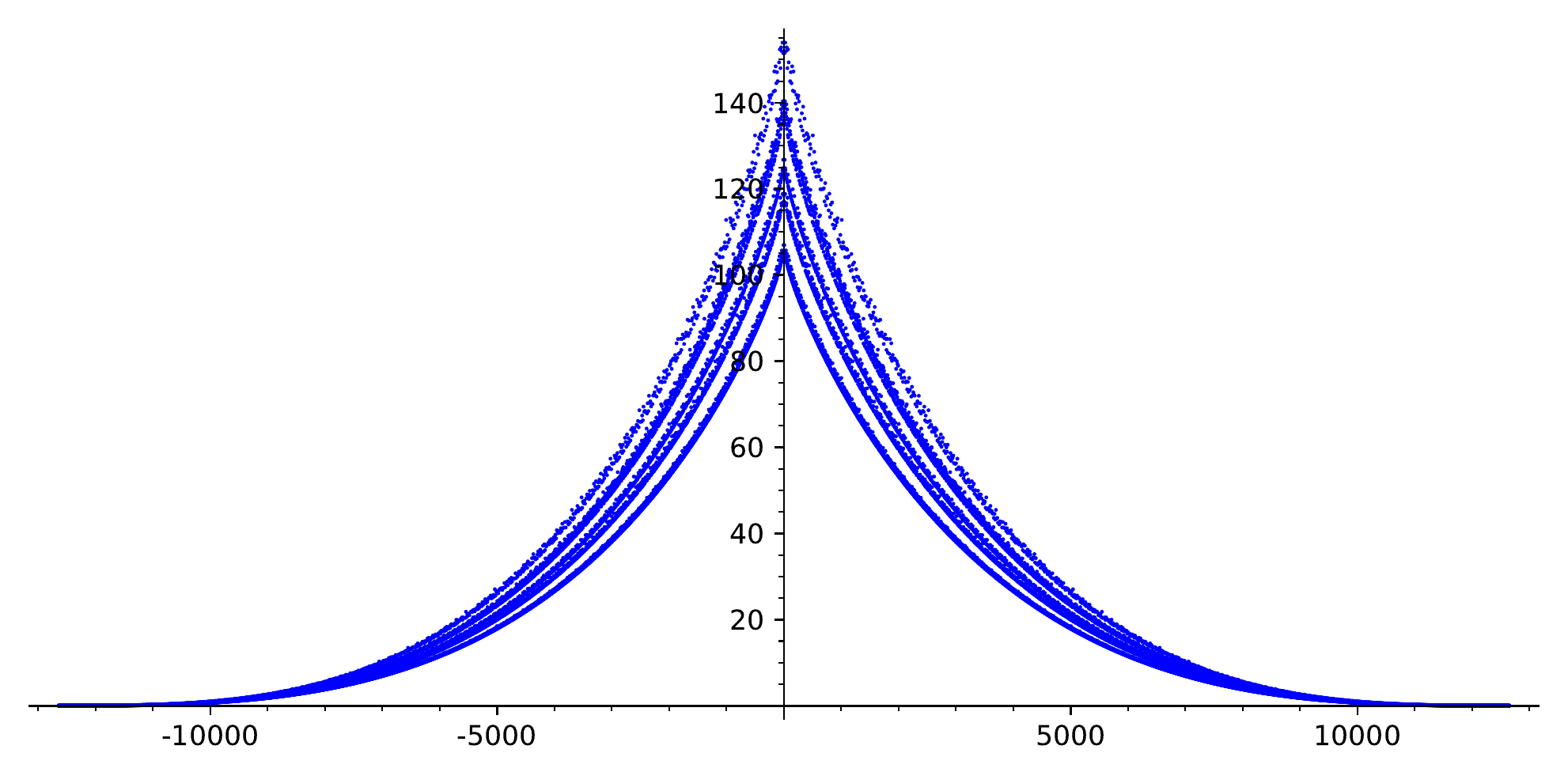}
\caption{$\pi_{E_{31} \times E_{107}, T}^{\pred}(10^7)$}
\label{fig:Pred31107p10M}
\end{figure}

\begin{figure}[H]
\centering
\includegraphics[width=12cm]{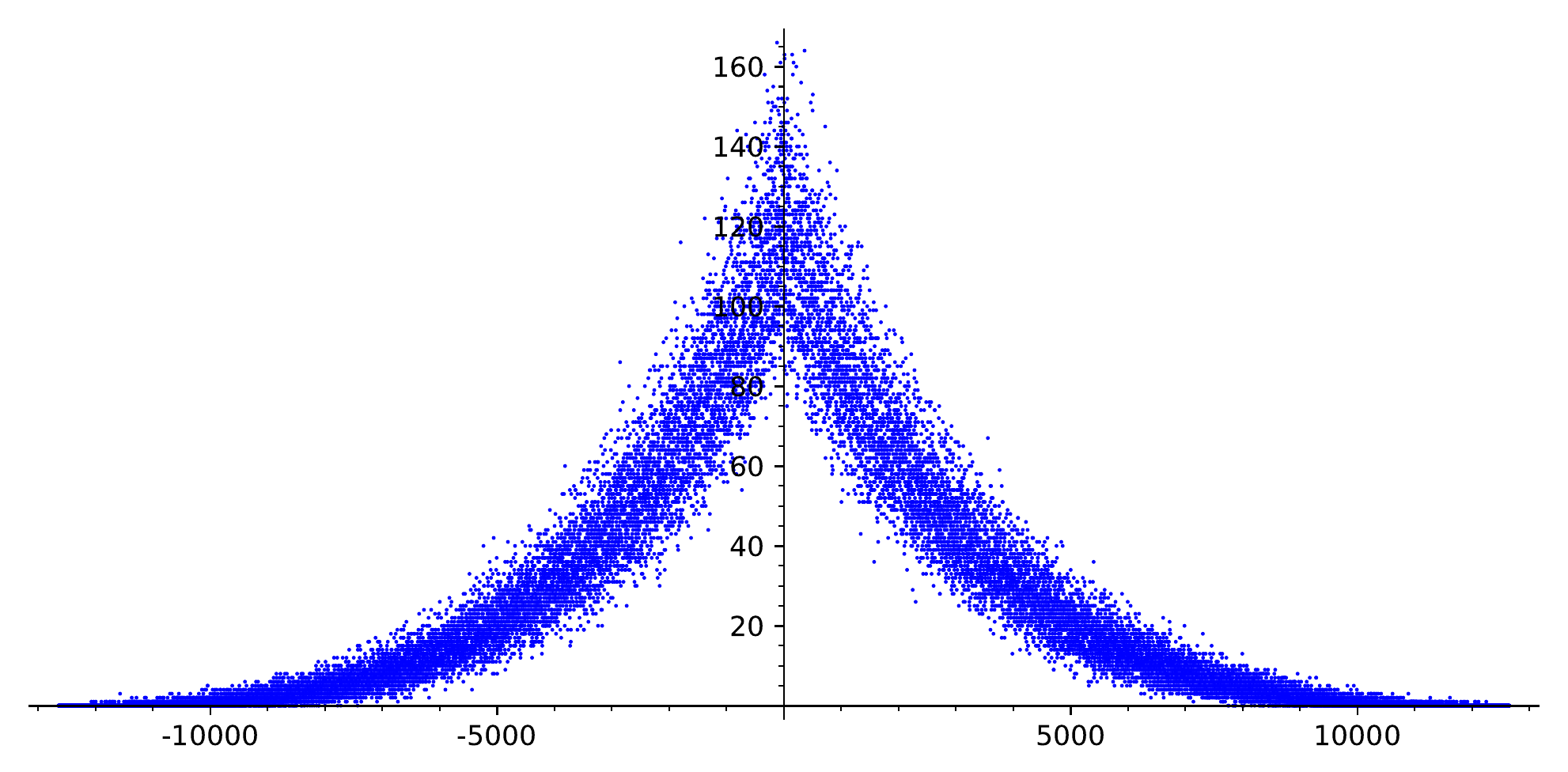}
\caption{$\pi_{E_{31} \times E_{107}, T}^{\actual}(10^7)$}
\label{fig:Actual31107p10M}
\end{figure}

\begin{figure}[H]
\centering
\includegraphics[width=12cm]{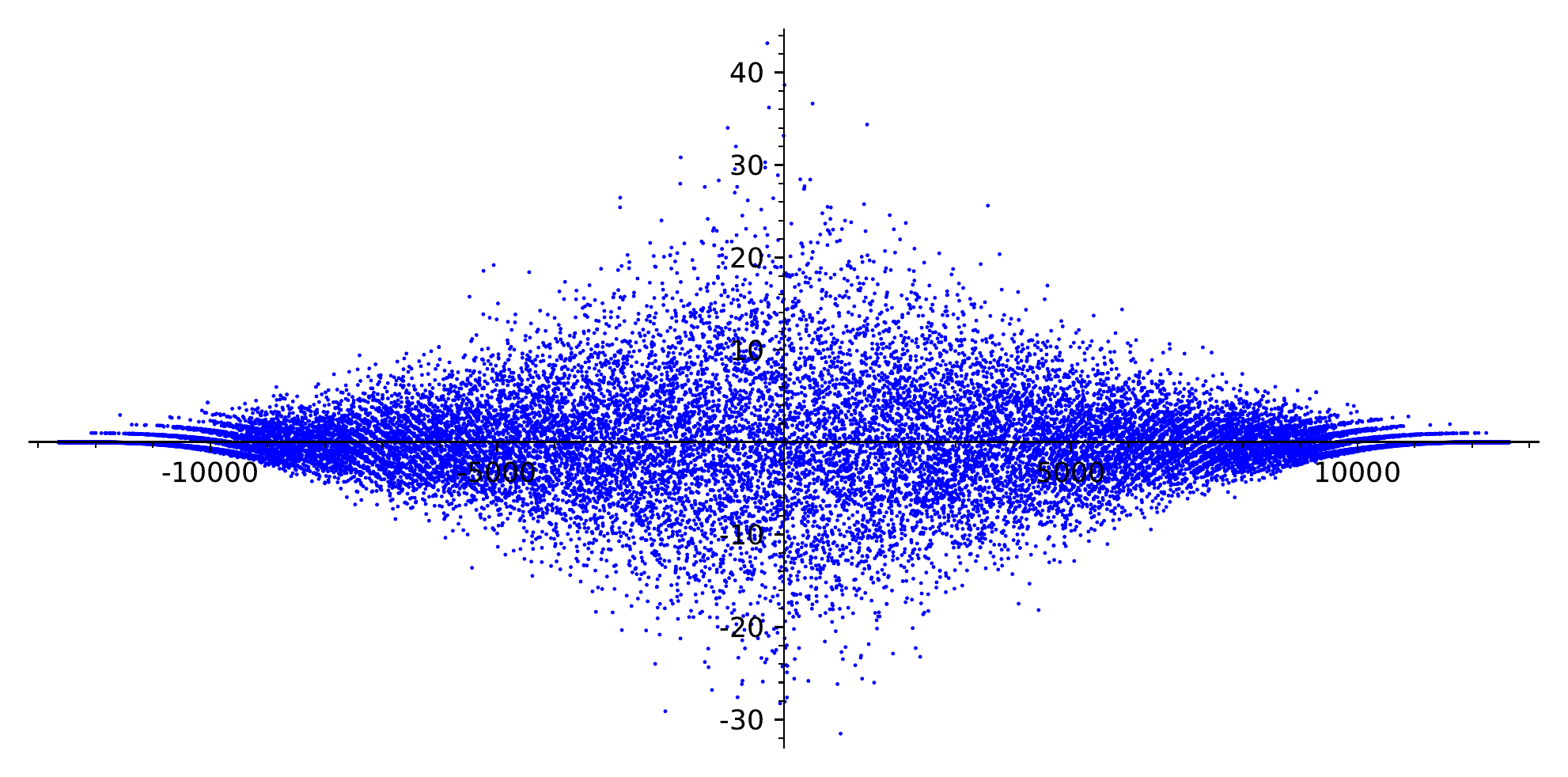}
\caption{$\mc{E}_{E_{31} \times E_{107}, T}^{\abs}(10^7)$}
\label{fig:Diff31107p10M}
\end{figure}

\begin{figure}[H]
\centering
\includegraphics[width=12cm]{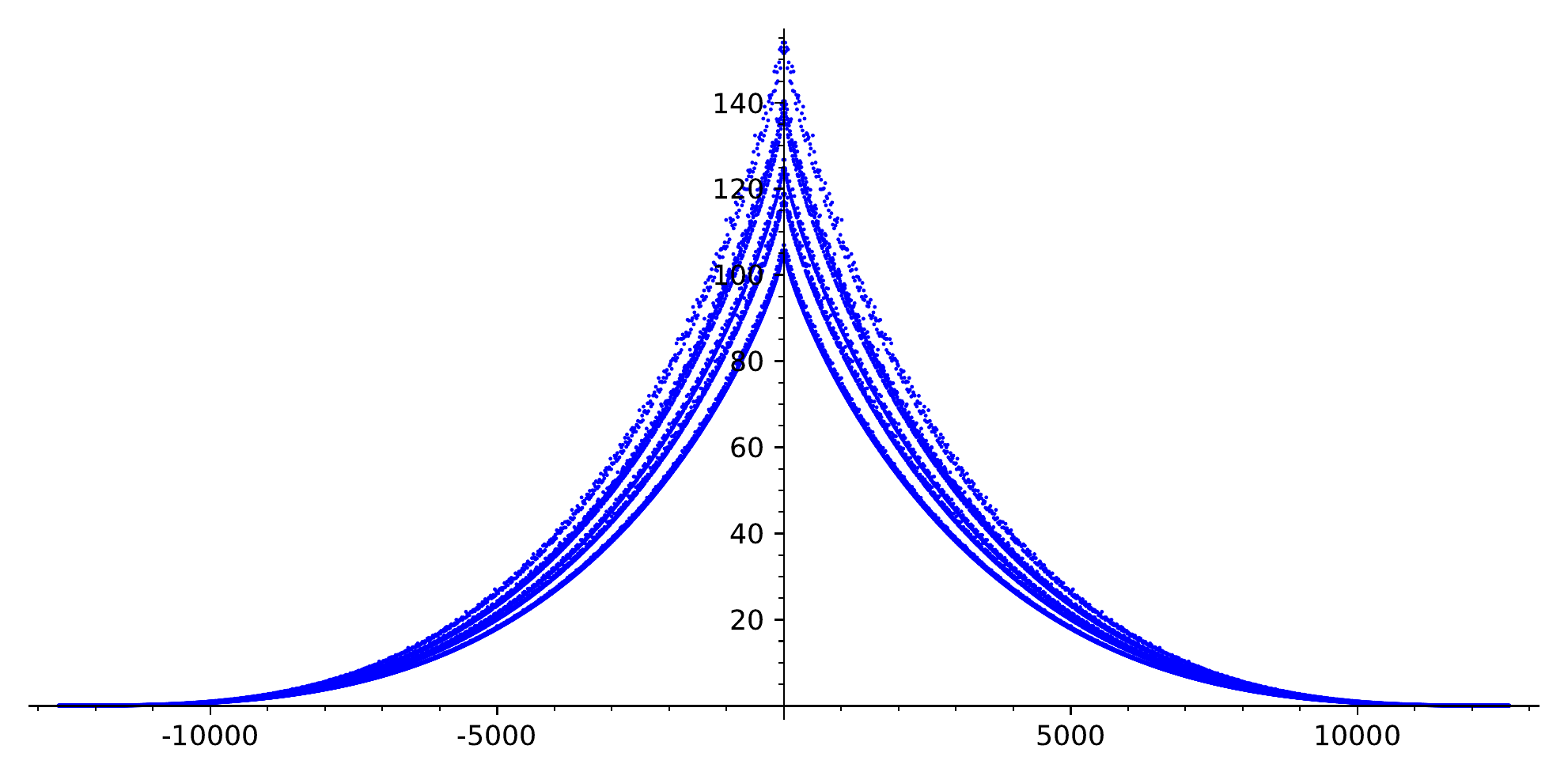}
\caption{$\pi_{E_{79} \times E_{107}, T}^{\pred}(10^7)$}
\label{fig:Pred79107p10M}
\end{figure}

\begin{figure}[H]
\centering
\includegraphics[width=12cm]{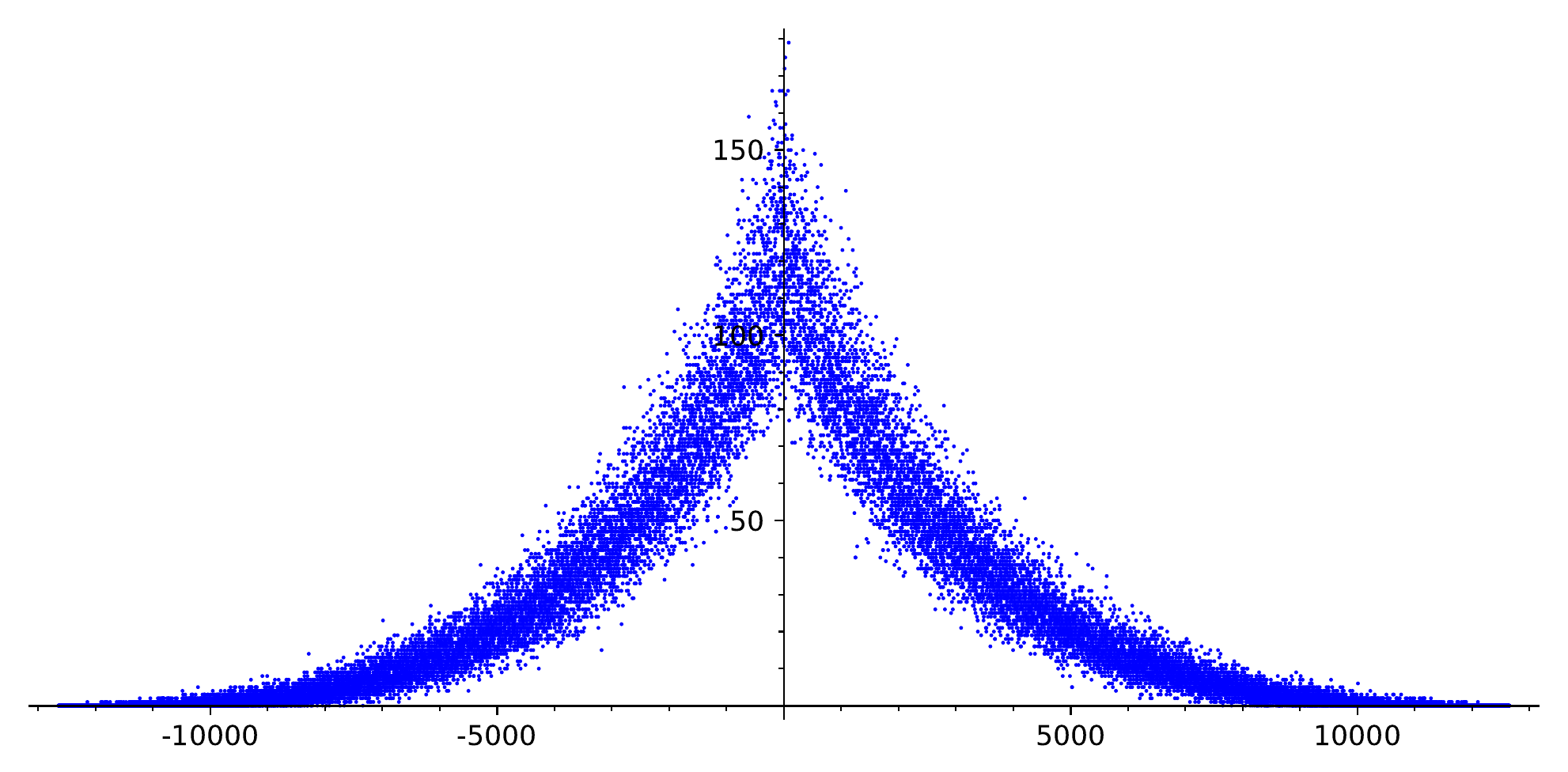}
\caption{$\pi_{E_{79} \times E_{107}, T}^{\actual}(10^7)$}
\label{fig:Actual79107p10M}
\end{figure}

\begin{figure}[H]
\centering
\includegraphics[width=12cm]{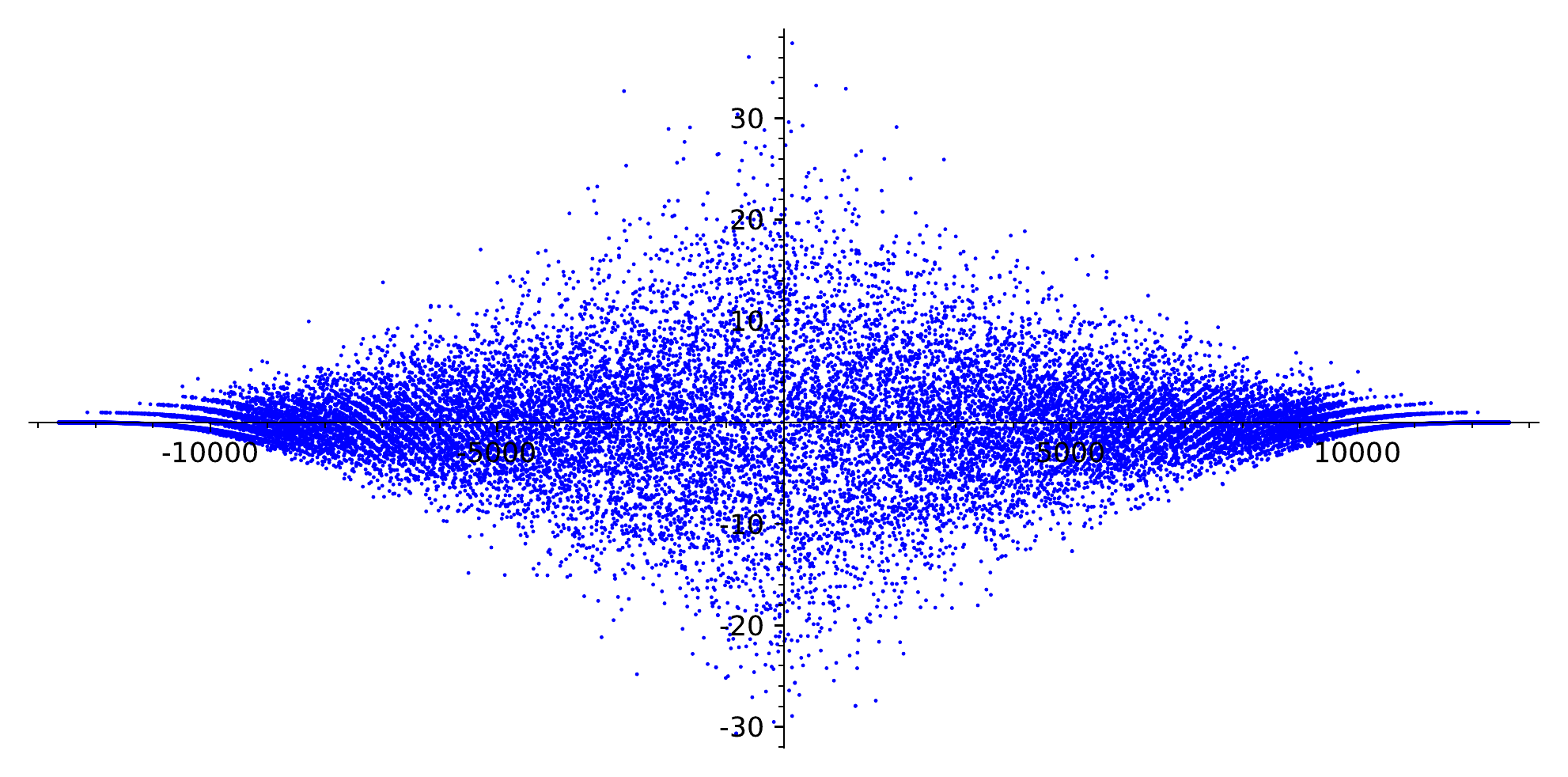}
\caption{$\mc{E}_{E_{79} \times E_{107}, T}^{\abs}(10^7)$}
\label{fig:Diff79107p10M}
\end{figure}

In each case, there are various ``stripes'' appearing in our prediction $\pi_{E_{\ell_1} \times E_{\ell_2}, T}^{\pred}(10^7)$, corresponding to values of $T$ lying in certain residue classes modulo $12$.  As may be seen in the numerical evidence presented in \cite{baierjones}, similar stripes are also present in the analogous plot of the points $\left\{ \left(T,\pi_{E,T}^{\pred}(x) \right) : |T| \leq 2\sqrt{x} \right\}$ associated to the original Lang-Trotter conjecture for fixed trace $T$ involving a single elliptic curve $E$. In that case stripes also emerge in the actual data $\left\{ \left(T,\pi_{E,T}^{\actual}(x) \right) : |T| \leq 2\sqrt{x} \right\}$.  By contrast, in the plots associated to pairs of elliptic curves we computed, the stripes appearing in the predictions are situated closer to one another, and the statistical noise present in the actual data causes them to ``blur together'' and no longer be visible in the plots of $\pi_{E_{\ell_1} \times E_{\ell_2}, T}^{\actual}(10^7)$.  To remedy this and further validate our conjectural constants $c(E_{\ell_1} \times E_{\ell_2},T)$, we additionally exhibit rainbow-colored graphs for the Serre pair $E_{79} \times E_{107}$ (see Figure \ref{fig:PredExtreme311} and Figure \ref{fig:ActualExtreme311}), in which data points are colored according to their greatest common divisor with $12$, as detailed in the following table.
\begin{center}
\begin{tabular}{|L|L|}
\hline \gcd(T,12) &\text{color of $\left( T,\pi_{E_{\ell_1} \times E_{\ell_2}, T}(10^7) \right)$}\\
\hline\hline
1 & \text{ red } \\
3 & \text{ orange } \\
2 & \text{ yellow }\\
4 \text{ or } 6 & \text{ green } \\
12 & \text{ blue }\\
\hline
\end{tabular}
\end{center}
The reason for the appearance of $\gcd(T,12)$ here is that the local factor of our constant at a prime $\ell$ is larger when $\ell$ divides $T$ than when $\ell$ does not divide $T$ (see Lemma \ref{good primes}, Theorem \ref{primesdividingtrace} and Proposition \ref{badprimesmainprop}).  Furthermore, the difference in size between the factor when $\ell \mid T$ versus when $\ell \nmid T$ is most dramatic when $\ell$ is small, and decreases very quickly as $\ell$ grows.  Because of this, the primes $2$ and $3$ account for the stripes in our prediction. \\
The color coding of the stripes in the rainbow graphs allows us to observe a more precise agreement between the predicted and actual data in this regard, giving yet more credence to the conjectural constants $c(E_{\ell_1}\times E_{\ell_2},T)$.

\begin{figure}[H]
\centering
\includegraphics[width=12cm]{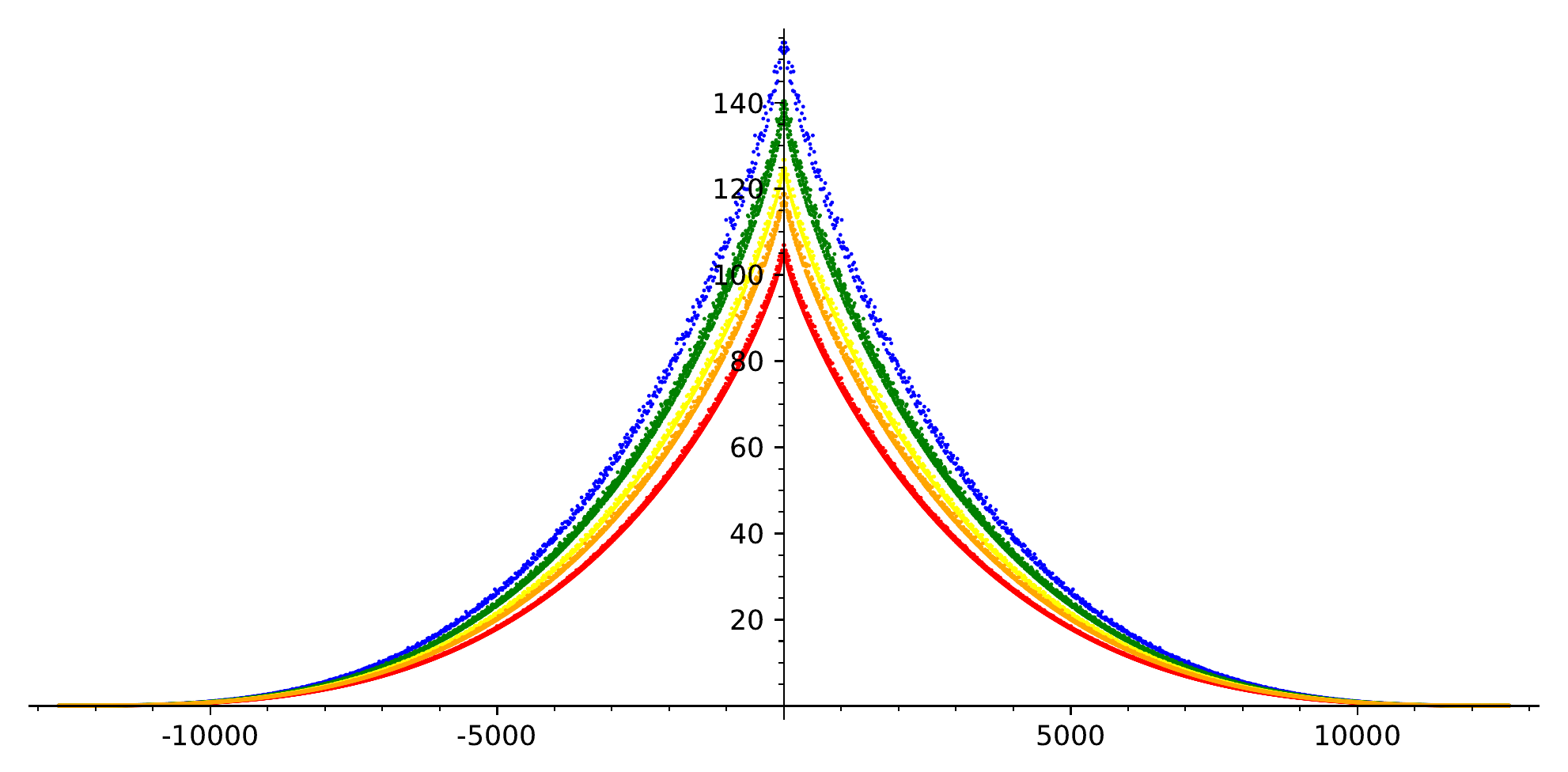}
\caption{Colored values of $\pi_{E_{79} \times E_{107}, T}^{\pred}(10^7)$}
\label{fig:PredExtreme311}
\end{figure}

\begin{figure}[H]
\centering
\includegraphics[width=12cm]{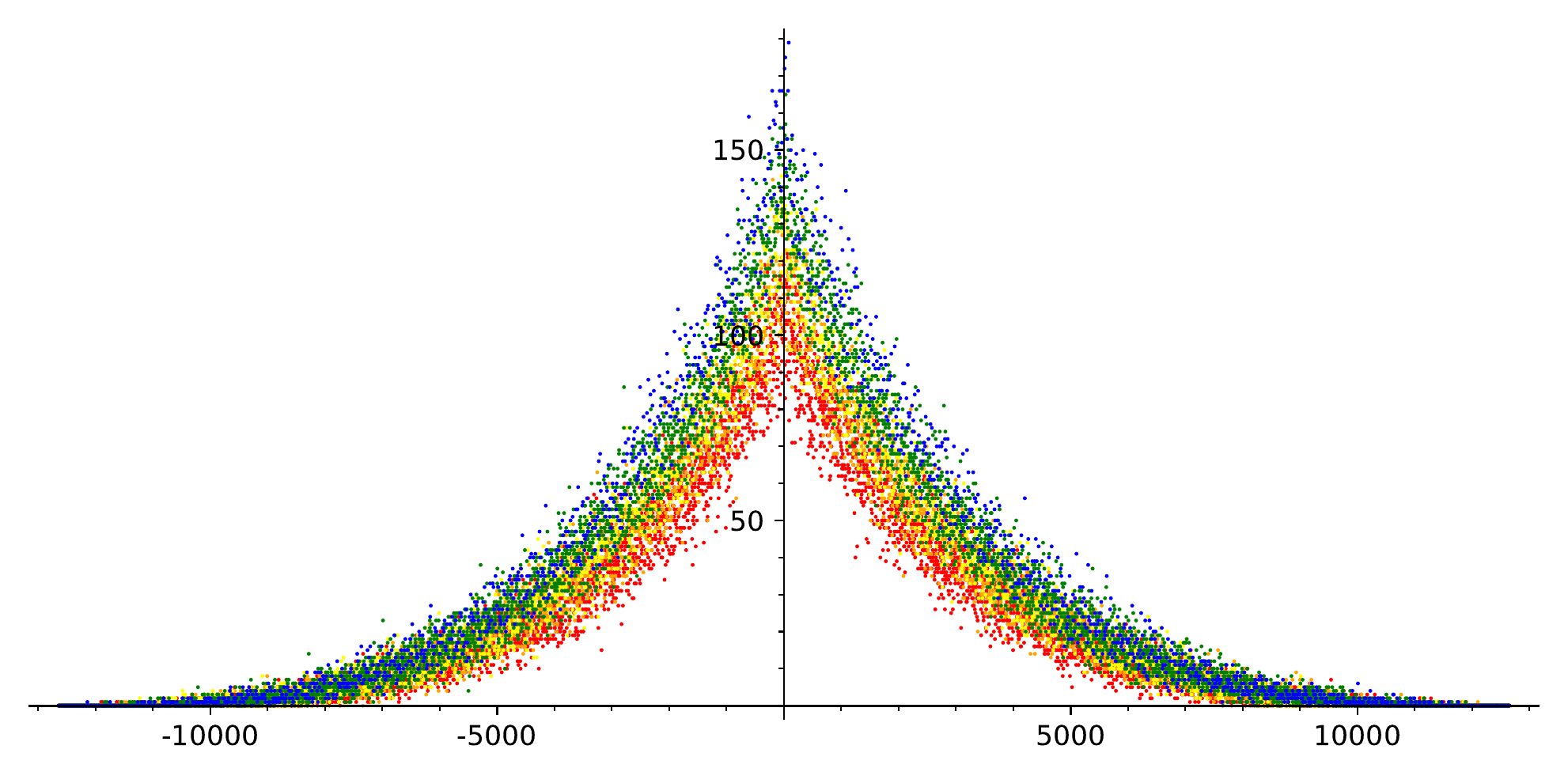}
\caption{Colored values of $\pi_{E_{79} \times E_{107}, T}^{\actual}(10^7)$}
\label{fig:ActualExtreme311}
\end{figure}

\printbibliography[heading=bibintoc]

@article {CDSS,
    AUTHOR = {Cojocaru, Alina Carmen and Davis, Rachel and Silverberg, Alice
              and Stange, Katherine E.},
     TITLE = {Arithmetic properties of the {F}robenius traces defined by a
              rational abelian variety (with two appendices by {J}-{P}.
              {S}erre)},
   JOURNAL = {Int. Math. Res. Not. IMRN},
  FJOURNAL = {International Mathematics Research Notices. IMRN},
      YEAR = {2017},
    NUMBER = {12},
     PAGES = {3557--3602},
      ISSN = {1073-7928},
   MRCLASS = {11G10 (14G25 14H52 14K15)},
  MRNUMBER = {3693659},
MRREVIEWER = {Joseph H. Silverman},
       DOI = {10.1093/imrn/rnw058},
       URL = {https://doi.org/10.1093/imrn/rnw058},
}

@manual{SAGEMATH,
  Key          = {SageMath},
  Author       = {{The Sage Developers}},
  Title        = {{S}ageMath, the {S}age {M}athematics {S}oftware {S}ystem ({V}ersion 9.0)},
  note         = {\url{ https://www.sagemath.org}},
  Year         = {2020},
}

@article {MAGMA,
    AUTHOR = {Bosma, Wieb and Cannon, John and Playoust, Catherine},
     TITLE = {The {M}agma algebra system. {I}. {T}he user language},
      NOTE = {Computational algebra and number theory (London, 1993)},
   JOURNAL = {J. Symbolic Comput.},
  FJOURNAL = {Journal of Symbolic Computation},
    VOLUME = {24},
      YEAR = {1997},
    NUMBER = {3-4},
     PAGES = {235--265},
      ISSN = {0747-7171},
   MRCLASS = {68Q40},
  MRNUMBER = {MR1484478},
       DOI = {10.1006/jsco.1996.0125},
       URL = {http://dx.doi.org/10.1006/jsco.1996.0125},
}

@article{fouvry_murty_1996, title={On the Distribution of Supersingular Primes}, volume={48}, DOI={10.4153/CJM-1996-004-7}, number={1}, journal={Canadian Journal of Mathematics}, publisher={Cambridge University Press}, author={Fouvry, Etienne and Murty, M. Ram}, year={1996}, pages={81–104}}

@article {DavidPapp,
    AUTHOR = {David, Chantal and Pappalardi, Francesco},
     TITLE = {Average {F}robenius distributions of elliptic curves},
   JOURNAL = {Internat. Math. Res. Notices},
  FJOURNAL = {International Mathematics Research Notices},
      YEAR = {1999},
    NUMBER = {4},
     PAGES = {165--183},
      ISSN = {1073-7928},
   MRCLASS = {11G05 (11F30 11N36)},
  MRNUMBER = {1677267},
MRREVIEWER = {M. Ram Murty},
       DOI = {10.1155/S1073792899000082},
       URL = {https://doi.org/10.1155/S1073792899000082},
}

@inproceedings {Lagarias1977EffectiveVO,
    AUTHOR = {Lagarias, J. C. and Odlyzko, A. M.},
     TITLE = {Effective versions of the {C}hebotarev density theorem},
 BOOKTITLE = {Algebraic number fields: {$L$}-functions and {G}alois
              properties ({P}roc. {S}ympos., {U}niv. {D}urham, {D}urham,
              1975)},
     PAGES = {409--464},
      YEAR = {1977},
   MRCLASS = {12A75},
  MRNUMBER = {0447191},
MRREVIEWER = {Matti Jutila},
}

@article{baierjones,
author = {Baier, Stephan and Jones, Nathan},
year = {2008},
month = {02},
pages = {},
title = {A Refined Version of the Lang-Trotter Conjecture},
volume = {2009},
journal = {International Mathematics Research Notices},
doi = {10.1093/imrn/rnn136}
}

@unpublished{future_upperbounds,
author = {Cojocaru, Alina Carmen and Jones, Nathan and Serban, Vlad and Wang, Tian},
title = {Bounds for distributions of Frobenius traces on abelian varieties with small Sato-Tate group},
note = {(in preparation)}
}

@incollection {Harris2009,
    AUTHOR = {Harris, Michael},
     TITLE = {Potential automorphy of odd-dimensional symmetric powers of
              elliptic curves and applications},
 BOOKTITLE = {Algebra, arithmetic, and geometry: in honor of {Y}u. {I}.
              {M}anin. {V}ol. {II}},
    SERIES = {Progr. Math.},
    VOLUME = {270},
     PAGES = {1--21},
 PUBLISHER = {Birkh\"auser Boston, Inc., Boston, MA},
      YEAR = {2009},
   MRCLASS = {11F80 (11F70 11G05 22E55)},
  MRNUMBER = {2641185},
       DOI = {10.1007/978-0-8176-4747-6_1}
}

@book {LangTrotter,
    AUTHOR = {Lang, Serge and Trotter, Hale},
     TITLE = {Frobenius distributions in {${\rm GL}_{2}$}-extensions},
    SERIES = {Lecture Notes in Mathematics, Vol. 504},
 PUBLISHER = {Springer-Verlag, Berlin-New York},
      YEAR = {1976},
   MRCLASS = {12A50 (10K05)},
  MRNUMBER = {0568299},
}

@article {MR3557121,
    AUTHOR = {Daniels, Harris B. and Hatley, Jeffrey and Ricci, James},
     TITLE = {Elliptic curves with maximally disjoint division fields},
   JOURNAL = {Acta Arith.},
  FJOURNAL = {Acta Arithmetica},
    VOLUME = {175},
      YEAR = {2016},
    NUMBER = {3},
     PAGES = {211--223},
      ISSN = {0065-1036},
   MRCLASS = {14H52 (11F80)},
  MRNUMBER = {3557121},
MRREVIEWER = {Rupam Barman},
}

@incollection {MR1638488,
    AUTHOR = {Silverberg, Alice},
     TITLE = {Explicit families of elliptic curves with prescribed mod {$N$}
              representations},
 BOOKTITLE = {Modular forms and {F}ermat's last theorem ({B}oston, {MA},
              1995)},
     PAGES = {447--461},
 PUBLISHER = {Springer, New York},
      YEAR = {1997},
   MRCLASS = {11G05 (11F80 11G18)},
  MRNUMBER = {1638488},
}

@article {MR3349445,
    AUTHOR = {Daniels, Harris B.},
     TITLE = {An infinite family of {S}erre curves},
   JOURNAL = {J. Number Theory},
  FJOURNAL = {Journal of Number Theory},
    VOLUME = {155},
      YEAR = {2015},
     PAGES = {226--247},
      ISSN = {0022-314X},
   MRCLASS = {11G05 (11F70 14G25)},
  MRNUMBER = {3349445},
MRREVIEWER = {Iv\'{a}n Blanco-Chac\'{o}n},
       DOI = {10.1016/j.jnt.2015.03.016},
       URL = {https://doi.org/10.1016/j.jnt.2015.03.016},
}

@article {MR2837018,
    AUTHOR = {Cojocaru, Alina-Carmen and Grant, David and Jones, Nathan},
     TITLE = {One-parameter families of elliptic curves over {$\Bbb Q$} with
              maximal {G}alois representations},
   JOURNAL = {Proc. Lond. Math. Soc. (3)},
  FJOURNAL = {Proceedings of the London Mathematical Society. Third Series},
    VOLUME = {103},
      YEAR = {2011},
    NUMBER = {4},
     PAGES = {654--675},
      ISSN = {0024-6115},
   MRCLASS = {11G05 (11F80 11G30)},
  MRNUMBER = {2837018},
MRREVIEWER = {Joseph H. Silverman},
       DOI = {10.1112/plms/pdr001},
       URL = {https://doi.org/10.1112/plms/pdr001},
}

@article {Serre1972,
    AUTHOR = {Serre, Jean-Pierre},
     TITLE = {Propri\'et\'es galoisiennes des points d'ordre fini des
              courbes elliptiques},
   JOURNAL = {Invent. Math.},
  FJOURNAL = {Inventiones Mathematicae},
    VOLUME = {15},
      YEAR = {1972},
    NUMBER = {4},
     PAGES = {259--331},
      ISSN = {0020-9910},
   MRCLASS = {14G25 (14K15)},
  MRNUMBER = {0387283},
}

@article {AkbaryParks,
    AUTHOR = {Akbary, Amir and Parks, James},
     TITLE = {On the {L}ang-{T}rotter conjecture for two elliptic curves},
   JOURNAL = {Ramanujan J.},
  FJOURNAL = {Ramanujan Journal. An International Journal Devoted to the
              Areas of Mathematics Influenced by Ramanujan},
    VOLUME = {49},
      YEAR = {2019},
    NUMBER = {3},
     PAGES = {585--623},
      ISSN = {1382-4090},
   MRCLASS = {11G05 (11M41)},
  MRNUMBER = {3979693},
MRREVIEWER = {Joseph H. Silverman},
       DOI = {10.1007/s11139-018-0050-7},
       URL = {https://doi.org/10.1007/s11139-018-0050-7},
}

@incollection {BaierPitankar,
    AUTHOR = {Baier, S. and Patankar, Vijay M.},
     TITLE = {Applications of the square sieve to a conjecture of {L}ang and
              {T}rotter for a pair of elliptic curves over the rationals},
 BOOKTITLE = {Geometry, algebra, number theory, and their information
              technology applications},
    SERIES = {Springer Proc. Math. Stat.},
    VOLUME = {251},
     PAGES = {39--57},
 PUBLISHER = {Springer, Cham},
      YEAR = {2018},
   MRCLASS = {11N36 (11G05 11N45 11R45)},
  MRNUMBER = {3880382},
MRREVIEWER = {Volker Ziegler},
       DOI = {10.1007/978-3-319-97379-1_3},
       URL = {https://doi.org/10.1007/978-3-319-97379-1_3},
}

@article {Oesterle,
    AUTHOR = {Oesterl\'{e}, Joseph},
     TITLE = {R\'{e}duction modulo {$p^{n}$} des sous-ensembles analytiques
              ferm\'{e}s de {${\bf Z}^{N}_{p}$}},
   JOURNAL = {Invent. Math.},
  FJOURNAL = {Inventiones Mathematicae},
    VOLUME = {66},
      YEAR = {1982},
    NUMBER = {2},
     PAGES = {325--341},
      ISSN = {0020-9910},
   MRCLASS = {12B99 (10D99 10K40)},
  MRNUMBER = {656627},
MRREVIEWER = {W. Bartenwerfer},
       DOI = {10.1007/BF01389398},
       URL = {https://doi.org/10.1007/BF01389398},
}

@article {SerreCebo,
    AUTHOR = {Serre, Jean-Pierre},
     TITLE = {Quelques applications du th\'{e}or\`eme de densit\'{e} de {C}hebotarev},
   JOURNAL = {Inst. Hautes \'{E}tudes Sci. Publ. Math.},
  FJOURNAL = {Institut des Hautes \'{E}tudes Scientifiques. Publications
              Math\'{e}matiques},
    NUMBER = {54},
      YEAR = {1981},
     PAGES = {323--401},
      ISSN = {0073-8301},
   MRCLASS = {12A75 (10D99 10H25 14G25)},
  MRNUMBER = {644559},
MRREVIEWER = {J. Tunnell},
       URL = {http://archive.numdam.org/article/PMIHES_1981__54__123_0.pdf},
}

@misc{lmfdb,
  shorthand    = {LMFDB},
  author       = { {The LMFDB Collaboration} },
  title        =  {The L-functions and Modular Forms Database},
  howpublished = {\url{http://www.lmfdb.org}},
  year         = {2013},
  note         = {[Online; accessed 16 September 2013]},
}

@article {MR2563740,
    AUTHOR = {Jones, Nathan},
     TITLE = {Almost all elliptic curves are {S}erre curves},
   JOURNAL = {Trans. Amer. Math. Soc.},
  FJOURNAL = {Transactions of the American Mathematical Society},
    VOLUME = {362},
      YEAR = {2010},
    NUMBER = {3},
     PAGES = {1547--1570},
      ISSN = {0002-9947},
   MRCLASS = {11G05 (11F80 11R45)},
  MRNUMBER = {2563740},
MRREVIEWER = {Ravi K. Ramakrishna},
       DOI = {10.1090/S0002-9947-09-04804-1},
       URL = {https://doi.org/10.1090/S0002-9947-09-04804-1},
}

@article {MR3071819,
    AUTHOR = {Jones, Nathan},
     TITLE = {Pairs of elliptic curves with maximal {G}alois
              representations},
   JOURNAL = {J. Number Theory},
  FJOURNAL = {Journal of Number Theory},
    VOLUME = {133},
      YEAR = {2013},
    NUMBER = {10},
     PAGES = {3381--3393},
      ISSN = {0022-314X},
   MRCLASS = {11G05 (11F80 14H52)},
  MRNUMBER = {3071819},
MRREVIEWER = {\'{A}lvaro Lozano-Robledo},
       DOI = {10.1016/j.jnt.2013.03.002},
       URL = {https://doi.org/10.1016/j.jnt.2013.03.002},
}

@article {MR2534114,
    AUTHOR = {Jones, Nathan},
     TITLE = {Averages of elliptic curve constants},
   JOURNAL = {Math. Ann.},
  FJOURNAL = {Mathematische Annalen},
    VOLUME = {345},
      YEAR = {2009},
    NUMBER = {3},
     PAGES = {685--710},
      ISSN = {0025-5831},
   MRCLASS = {11G05 (11N05 11R32)},
  MRNUMBER = {2534114},
MRREVIEWER = {M. Ram Murty},
       DOI = {10.1007/s00208-009-0373-1},
       URL = {https://doi.org/10.1007/s00208-009-0373-1},
}

@article{RubinSilverberg,
author = {Rubin, Karl and Silverberg, Alice},
year = {1999},
month = {02},
pages = {},
title = {Mod 2 representations of elliptic curves},
volume = {129},
journal = {Proceedings of the American Mathematical Society},
doi = {10.2307/2669028}
}

@incollection {MurtyFrobdistribs,
    AUTHOR = {Murty, V. Kumar},
     TITLE = {Frobenius distributions and {G}alois representations},
 BOOKTITLE = {Automorphic forms, automorphic representations, and arithmetic
              ({F}ort {W}orth, {TX}, 1996)},
    SERIES = {Proc. Sympos. Pure Math.},
    VOLUME = {66},
     PAGES = {193--211},
 PUBLISHER = {Amer. Math. Soc., Providence, RI},
      YEAR = {1999},
   MRCLASS = {11F80 (11G05 11R45)},
  MRNUMBER = {1703751},
MRREVIEWER = {Ian Kiming},
}

@article {Deuring,
    AUTHOR = {Deuring, Max},
     TITLE = {Die {T}ypen der {M}ultiplikatorenringe elliptischer
              {F}unktionenk\"{o}rper},
   JOURNAL = {Abh. Math. Sem. Hansischen Univ.},
    VOLUME = {14},
      YEAR = {1941},
     PAGES = {197--272},
   MRCLASS = {09.1X},
  MRNUMBER = {0005125},
MRREVIEWER = {S. Mac Lane},
       DOI = {10.1007/BF02940746},
       URL = {https://doi.org/10.1007/BF02940746},
}

@article {KatzLTrevisited,
    AUTHOR = {Katz, Nicholas M.},
     TITLE = {Lang-{T}rotter revisited},
   JOURNAL = {Bull. Amer. Math. Soc. (N.S.)},
  FJOURNAL = {American Mathematical Society. Bulletin. New Series},
    VOLUME = {46},
      YEAR = {2009},
    NUMBER = {3},
     PAGES = {413--457},
      ISSN = {0273-0979},
   MRCLASS = {11G05 (11F80 11G18 11G20)},
  MRNUMBER = {2507277},
MRREVIEWER = {\'{A}lvaro Lozano-Robledo},
       DOI = {10.1090/S0273-0979-09-01257-9},
       URL = {https://doi.org/10.1090/S0273-0979-09-01257-9},
}

@article {KedlayaFite+,
    AUTHOR = {Fit\'{e}, Francesc and Kedlaya, Kiran S. and Rotger, V\'{\i}ctor and
              Sutherland, Andrew V.},
     TITLE = {Sato-{T}ate distributions and {G}alois endomorphism modules in
              genus 2},
   JOURNAL = {Compos. Math.},
  FJOURNAL = {Compositio Mathematica},
    VOLUME = {148},
      YEAR = {2012},
    NUMBER = {5},
     PAGES = {1390--1442},
      ISSN = {0010-437X},
   MRCLASS = {11M50 (11G10 11G20 14G10 14K15)},
  MRNUMBER = {2982436},
MRREVIEWER = {Imin Chen},
       DOI = {10.1112/S0010437X12000279},
      URL = {https://doi.org/10.1112/S0010437X12000279},
}
\end{document}